\documentclass [11pt,reqno]{amsart}
\usepackage {amsmath,amssymb,epsfig,verbatim,geometry}
\usepackage[all]{xy}

\usepackage[usenames,dvipsnames]{color}

\newcommand{\A}{\mathbf{A}}

\newcommand{\G}{\mathbf{G}}

\newcommand{\Q}{\mathbf{Q}}

\newcommand{\R}{\mathbf{R}}
\newcommand{\Z}{\mathbf{Z}}
\newcommand{\N}{\mathbf{N}}
\renewcommand{\P}{\mathbf{P}}

\newcommand{\fa}{\mathfrak{a}}
\newcommand{\fb}{\mathfrak{b}}
\newcommand{\fm}{\mathfrak{m}}

\newcommand{\tP}{\widetilde{P}}

\newcommand{\tcX}{\widetilde{\mathcal{X}}}

\newcommand{\cA}{\mathcal{A}}
\newcommand{\cB}{\mathcal{B}}
\newcommand{\cC}{\mathcal{C}}
\newcommand{\cD}{\mathcal{D}}
\newcommand{\cE}{\mathcal{E}}
\newcommand{\cF}{\mathcal{F}}

\newcommand{\cJ}{\mathcal{J}}

\newcommand{\cL}{\mathcal{L}}
\newcommand{\cM}{\mathcal{M}}

\newcommand{\cO}{\mathcal{O}}

\newcommand{\cT}{\mathcal{T}}
\newcommand{\cU}{\mathcal{U}}

\newcommand{\cX}{\mathcal{X}}
\newcommand{\cY}{\mathcal{Y}}
\newcommand{\cZ}{\mathcal{Z}}
\newcommand{\hcO}{\widehat{\mathcal{O}}}
\newcommand{\hcX}{\widehat{\mathcal{X}}}

\newcommand{\Xa}{X}
\newcommand{\Xan}{X}
\newcommand{\Xdiv}{X^{\mathrm{div}}}
\newcommand{\Xqm}{X^{\mathrm{qm}}}
\newcommand{\Lan}{L}
\newcommand{\Uan}{U^{\mathrm{an}}}

\newcommand{\unipar}{\varpi}

\newcommand{\an}{\mathrm{an}}

\newcommand{\cent}{c}
\newcommand{\retr}{p}

\renewcommand{\a}{\alpha}
\renewcommand{\b}{\beta}
\renewcommand{\d}{\delta}
\newcommand{\e}{\varepsilon}
\newcommand{\f}{\varphi}

\newcommand{\om}{\omega}
\newcommand{\p}{\psi}

\newcommand{\D}{\Delta}

\newcommand{\red}{\mathrm{red}}

\newcommand{\cf}{\textrm{cf.}\ }

\newcommand{\ie}{\textrm{i.e.}\ }

\renewcommand{\=}{:=}

\newcommand{\cro}[1]{[\![#1]\!]}

\DeclareMathOperator{\Spec}{Spec}

\DeclareMathOperator{\supp}{Supp}

\DeclareMathOperator{\Div}{Div}

\renewcommand{\div}{\mathrm{div}}

\DeclareMathOperator{\PA}{PA}

\DeclareMathOperator{\diam}{diam}

\DeclareMathOperator{\lip}{Lip}
\DeclareMathOperator{\sta}{Star}
\DeclareMathOperator{\rel}{ri}
\DeclareMathOperator{\Pic}{Pic}

\DeclareMathOperator{\CH}{CH}

\DeclareMathOperator{\Conv}{Conv}
\DeclareMathOperator{\ord}{ord}
\DeclareMathOperator{\emb}{emb}
\DeclareMathOperator{\ev}{ev}
\DeclareMathOperator{\id}{id}

\DeclareMathOperator{\Nef}{Nef}

\DeclareMathOperator{\PSH}{PSH}

\DeclareMathOperator{\tr}{Tr}
\DeclareMathOperator{\val}{val}

\numberwithin{equation}{section}       

\newtheorem{prop} {Proposition} [section]
\newtheorem{thm}[prop] {Theorem}
\newtheorem{defi}[prop] {Definition}
\newtheorem{lem}[prop] {Lemma}
\newtheorem{cor}[prop]{Corollary}
\newtheorem{prop-def}[prop]{Proposition-Definition}

\newtheorem*{thmA}{Theorem A}
\newtheorem*{thmB}{Theorem B}
\newtheorem*{thmC}{Theorem C}

\newtheorem*{defi*} {Definition}

\newtheorem{rmk}[prop]{Remark}

\theoremstyle{remark}
\newtheorem*{ackn}{Acknowledgment}

\title{Singular semipositive metrics in non-Archimedean geometry}
\date{\today}

\author{S{\'e}bastien Boucksom
  \and
  Charles Favre
  \and
  Mattias Jonsson}

\address{CNRS-Universit{\'e} Paris 6\\
  Institut de Math{\'e}matiques\\
  F-75251 Paris Cedex 05\\
  France}
\email{boucksom@math.jussieu.fr}

\address{CNRS- CMLS\\ \'Ecole Polytechnique\\91128 Palaiseau Cedex France}
\email{favre@math.polytechnique.fr}

\address{Dept of Mathematics\\
  University of Michigan\\
  Ann Arbor, MI 48109-1043\\
  USA}
\email{mattiasj@umich.edu}

\thanks{The first author was partially supported by the ANR grants MACK and POSITIVE.
  The second author was partially supported by the ANR-grant BERKO and the ERC-Starting grant "Nonarcomp" no. 307856.
  The third author was partially supported by the CNRS and the NSF}

\begin{document}

\begin{abstract}
  Let $X$ be a smooth projective Berkovich space over a complete discrete valuation field $K$ of residue characteristic zero, endowed with an ample line bundle $L$.
We introduce a general notion of (possibly singular) semipositive (or plurisubharmonic) metrics on $L$, and prove the analogue of the following two basic results in the complex case: the set of semipositive metrics is compact modulo scaling, and each semipositive metric is a decreasing limit of smooth semipositive ones. In particular, for continuous metrics, our definition agrees with the one by S.-W. Zhang. The proofs use multiplier ideals and the construction of suitable models of $X$ over the valuation ring of $K$, using toroidal techniques. 
\end{abstract}

\maketitle

\setcounter{tocdepth}{1}
\tableofcontents

%
%
%
%
%
%


\section*{Introduction}
The notions of plurisubharmonic (psh) functions and positive currents lie at the heart of complex analysis. The study of these objects is usually referred to as \emph{pluripotential theory}, and it has become apparent in recent years that pluripotential theory should admit an analogue in the context of non-Archimedean analytic spaces in the sense of Berkovich. 

Potential theory on non-Archimedean curves is by now well-established, thanks to the work of Thuillier~\cite{thuphd} (see also~\cite{valtree,BR}). 
In higher dimensions, it should in principle be possible to mimic the complex case and define a plurisubharmonic function as an upper semicontinuous function whose restriction to any curve is subharmonic. 
While this approach is yet to be developed, a general notion of 
\emph{continuous} plurisubharmonic functions was very 
recently\footnote{In fact,~\cite{CLD} was posted after the first version of the present paper.}
introduced by Chambert-Loir and Ducros in~\cite{CLD}, based on their notion of positive currents. Their definition both localizes and generalizes the previously introduced notion of a \emph{semipositive metric} on a line bundle~\cite{Zha,Gub98,CL1}. 
In this paper we propose a general (global) definition of singular (\ie not necessarily continuous) semipositive metrics on ample line bundles, and prove basic compactness and regularization results for such metrics. 

\medskip
In a sequel to this paper~\cite{nama} we rely on the results obtained here to adapt to the non-Archimedean case the variational approach to complex Monge-Amp\`ere equations developed in~\cite{BBGZ} and prove a version of the celebrated Calabi-Yau theorem. 

\medskip
 In order to better explain our construction, let us briefly recall some facts from the complex case~\cite{bayreuth}.  Let $X$ be (the analytification of) a smooth projective complex variety and let $L$ be an \emph{ample} line bundle on $X$. A smooth metric $\|\cdot\|$ on $L$ is given in every local trivialization of $L$ by $|\cdot|\, e^{-\f}$ for some local smooth function $\f$, called the local weight of the metric. 
The metric is said to be semipositive if its curvature, which locally is given by $dd^c\f$, 
is a semipositive $(1,1)$-form, that is, $\f$ is psh.
More generally, one defines the notion of singular semipositive metrics by allowing $\f$ to be a general psh function, in which case the curvature is a positive closed $(1,1)$-current. 
 
It is a basic fact that every psh function is locally the decreasing limit of a sequence of smooth psh functions. The global analogue of this result for singular semipositive metrics fails for general line bundles, but a deep result of Demailly shows that every singular semipositive metric on an \emph{ample} line bundle $L$ is indeed a \emph{monotone} limit of smooth semipositive metrics (cf.~\cite{regularization} or~\cite[Theorem 8.1]{GZ},~\cite{BK07} for more recent accounts). In particular, every continuous semipositive metric on $L$ is a \emph{uniform}  limit on $X$ of smooth semipositive metrics, thanks to Dini's lemma.

A fundamental aspect of singular semipositive metrics is that they form
a compact space modulo scaling. This fact can be conveniently understood in terms of global weights as follows. Fixing a smooth metric on $L$ with curvature $\theta$ allows one to identify the set of singular  semipositive metrics  with the set $\PSH(\Xan,\theta)$ of $\theta$-psh functions. The latter are upper semicontinuous (usc) functions $\f: X \to [-\infty, +\infty)$ such that $\theta + dd^c \f $ is a positive closed $(1,1)$-current. Modulo scaling, $\PSH(\Xan,\theta)$, endowed with the $L^1$-topology, is homeomorphic to the space of closed positive $(1,1)$-currents lying in the cohomology class $c_1(L)$ (with its weak topology), and hence is compact.

\medskip
Let us now turn to the non-Archimedean case. Fix a complete, discrete
valuation field $K$ with valuation ring $R$ and residue field $k$, and
set $S:=\Spec R$.  We assume that $k$ (and hence $K$) has
characteristic zero, which means, concretely, that $R$ is (non-uniquely)
isomorphic to $k\cro{t}$. Let $\Xan$ be a smooth projective $K$-analytic
space in the sense of Berkovich, so that $\Xan$ is the analytification
of a smooth projective $K$-variety by the GAGA principle. Recall that
the underlying topological space of $\Xan$ is compact Hausdorff. A
\emph{model} $\cX$ of $X$ is a normal flat projective $S$-scheme
together with an isomorphism of the analytification of its generic
fiber $\cX_K$ with $\Xan$. Each line bundle $L$ on $\Xan$ is (again,
by GAGA) the analytification of a line bundle on $\cX_K$ (that we also denote by $L$), and a \emph{model metric} is a metric $\|\cdot\|_\cL$ on $\Lan$ that is naturally induced by the choice of a $\Q$-line bundle $\cL\in\Pic(\cX)_\Q$ such that $\cL|_{\cX_K}=L$ in $\Pic(\cX_K)_\Q$. A \emph{model function} $\f$ on $\Xan$ is a function such that $e^{-\f}$ is a model metric on the trivial line bundle. The set $\cD(\Xa)$ of model functions is then dense in $C^0(\Xa)$, a well-known consequence of the Stone-Weierstrass theorem. 

Following S.-W.Zhang~\cite[3.1]{Zha} (see also~\cite[6.2.1]{KT},~\cite[7.13]{Gub98},~\cite[2.2]{CL1}), we shall say that a model metric $\|\cdot\|_\cL$ on $\Lan$ is \emph{semipositive} if $\cL\in\Pic(\cX)_\Q$ is nef on the special fiber $\cX_0$ of $\cX$, \ie $\cL\cdot C\ge 0$ for all projective curves $C$ in $\cX_0$. A continuous metric on $\Lan$ is then \emph{semipositive in the sense of Zhang} if it can be written as a uniform limit over $\Xan$ of a sequence of semipositive model metrics. The reader may consult~\cite{CL2} for a nice survey on these notions. 

\medskip
In order to define a general notion of singular semipositive metrics,
we use the finer description of $\Xan$ as an inverse limit of dual
complexes (called \emph{skeletons} in Berkovich's terminology).  Since
the residue field $k$ of $R$ has characteristic zero, it follows
from~\cite{Tem} that each model of $X$ is dominated by a \emph{SNC
  model} $\cX$, by which we understand a regular model whose special
fiber $\cX_0$ has simple normal crossing support (plus a harmless
irreducibility condition that we impose for convenience). To each SNC
model $\cX$ corresponds its \emph{dual complex} $\D_\cX$, a compact
simplicial complex which encodes the incidence properties of the
irreducible components of $\cX_0$. The dual complex $\D_\cX$ embeds
canonically  in $\Xan$ and for the purposes of this
introduction we shall view $\D_\cX$ as a compact subset of $\Xan$. There is furthermore a 
retraction $\retr_\cX:\Xan \to \D_\cX$. These maps are compatible with respect to domination of models, and we thus get a map
$$
\Xan \to\varprojlim_\cX\D_\cX
$$
which is known to be a homeomorphism (see, for instance,~\cite[p.77, Theorem 10]{KS}).

\medskip
Following the philosophy of~\cite{BGS}, we define the space of \emph{closed $(1,1)$-forms} on $\Xan$ as the direct limit over all models $\cX$ of $X$ of the spaces $N^1(\cX/S)$ of codimension one numerical equivalence classes. Any model metric gives rise to a \emph{curvature form} lying in this space. The main reason for working with numerical equivalence (instead of rational equivalence as in~\cite{BGS}) is that we can then adapt a result of~\cite{Kun} to show that a line bundle $L\in\Pic(X)$ has vanishing first Chern class $c_1(L)\in N^1(X)$ iff it admits a model metric with zero curvature (cf.~Corollary~\ref{cor:flat}).

Fix a reference model metric $\|\cdot\|$ on $\Lan$ with curvature form $\theta$. Any other metric can be written $\|\cdot\|e^{-\f}$ for some function $\f$ on $\Xan$. When $\|\cdot\|e^{-\f}$ is a semipositive model metric, we say that the model function $\f$ is \emph{$\theta$-psh}. 

\begin{defi*}\label{defi:singular} Let $L$ be an ample line bundle on a smooth projective $K$-analytic variety $X$. Fix a model metric $\|\cdot\|$ on $\Lan$ with curvature form $\theta$. A \emph{$\theta$-plurisubharmonic function} on $\Xan$ is then a function $\f: \Xan \to [-\infty,+\infty)$
such that:
\begin{itemize}
\item $\f$ is upper semicontinuous (usc). 
\item $\f\le \f \circ \retr_\cX$ for each SNC model $\cX$.
\item $\f$ is a uniform limit, on each dual complex $\D_\cX$, of $\theta$-psh model functions. 
\end{itemize}
A \emph{singular semipositive metric} is a metric $\|\cdot\|e^{-\f}$ with $\f$ a $\theta$-psh function.
\end{defi*}
The consistency of the definition for model functions will be guaranteed by Theorem~\ref{thm:closed} below. Let $\f$ be a $\theta$-psh function. Since $\f$ is usc and each $\f \circ \retr_\cX$ is continuous, it follows immediately that $\f=\inf_\cX \f \circ \retr_\cX$, so that $\f$ is uniquely determined by its restriction to the dense subset $\bigcup_\cX\D_\cX$ of $\Xan$. We may therefore endow the set $\PSH(\Xa, \theta)$ of all
$\theta$-psh functions (or, equivalently, of all singular semipositive metrics on $L$) with the topology of uniform convergence on dual complexes. We view this topology as an analogue of the $L^1$-topology in the complex case. Our first main theorem shows that this space is indeed compact modulo additive constants, something that was also announced in the unpublished manuscript by Kontsevich and Tschinkel~\cite{KT}.

\begin{thmA} Let $L$ be an ample line bundle on a smooth projective $K$-analytic variety $X$ endowed with a model metric with curvature form $\theta$. Then $\PSH(\Xa,\theta)/\R$ is compact. 
\end{thmA}
In other words, the set of singular semipositive metrics on $L$ modulo scaling is compact. For curves, this result is a consequence of the work of Thuillier~\cite{thuphd}, and follows from basic properties of subharmonic functions on metrized graphs.

\smallskip

Our second main result is the following analogue of Demailly's global regularization theorem.

\begin{thmB} Let $L$ be an ample line bundle on a smooth projective $K$-analytic variety $X$ endowed with a model metric with curvature form $\theta$. Then every $\theta$-psh function $\f$ is the pointwise limit on $X$ of a decreasing net of $\theta$-psh model functions. 
\end{thmB}

When $\dim X=1$, Theorem B is a special case of~\cite[Th\'eor\`eme
3.4.19]{thuphd}. Thanks to Dini's lemma, Theorem B implies a non-Archimedean version of the Demailly-Richberg theorem, stating that every continuous $\theta$-psh function is a \emph{uniform} limit over $X$ of $\theta$-psh model functions. In other words, for continuous metrics our definition of semipositivity agrees with Zhang's. 

\smallskip

Let us briefly explain how Theorem~A above is proved. The first important fact is that any dual complex $\D_\cX$ comes equipped with a natural affine structure~\cite{KS} such that any $\theta$-psh function  is convex on the faces of $\D_\cX$. On this complex we put any euclidean metric compatible with the affine structure. The statement that we actually prove, and which implies Theorem A, is

\begin{thmC}
For each dual complex $\D_\cX$ there exists a constant $C>0$, such that 
$\f|_{\D_\cX}$ is \emph{Lipschitz continuous} with Lipschitz constant at most $C$
for any $\theta$-psh model function $\f$.
\end{thmC}
This result is proved in two steps. 
Assuming, as we may,  that $\sup_{\D_\cX}\f=0$, we first bound $\f$ from below on the vertices of $\D_\cX$. This is done by exploiting the non-negativity of certain intersection numbers, a direct consequence of the model metric $\|\cdot\|e^{-\f}$ being determined by a nef line bundle on some model. 

The next step is to prove the uniform Lipschitz bound. Here again, the general idea is to exploit the non-negativity of certain intersection numbers, but the argument is more subtle than in the first step. This time, the intersection numbers are computed on (possibly singular) blowups of $\cX$ corresponding to carefully chosen combinatorial decompositions of $\D_\cX$, in the spirit of the toroidal constructions of~\cite{KKMS}.
In~\cite{izumi} we adapt these techniques to prove a uniform version of Izumi's theorem~\cite{Izu85}.  

\medskip
The proof of Theorem~B is of a different nature. Using Theorem A, we first show that the usc upper envelope of any family of $\theta$-psh functions remains $\theta$-psh, a basic property of $\theta$-psh functions in the complex case. As a consequence, given any continuous function $u\in C^0(\Xa)$ the set of all $\theta$-psh functions $\psi$ such that $\psi\le u$ on $\Xan$ admits a largest element, called the \emph{$\theta$-psh envelope} of $u$ and denoted by $P_\theta(u)$. On the other hand, by density of $\cD(\Xa)$ in $C^0(\Xa)$, we may write any given function $\f\in\PSH(\Xa,\theta)$ as the pointwise limit of a decreasing family of (\emph{a priori} not $\theta$-psh)
model functions $u_j$, using only the upper semicontinuity of $\f$. It is not difficult to see that $P_\theta(u_j)$ decreases to $\f$, and we are thus reduced to showing that the $\theta$-psh envelope of any model function is the uniform limit of a sequence of $\theta$-psh model functions. This is proved using multiplier ideals, in the spirit of~\cite{DEL,ELS,hiro}. The required properties of multiplier ideals are shown to hold on regular models in Appendix~B, the key point being to show that the expected version of the Kodaira vanishing theorem holds in this context.\footnote{A simpler proof of the relevant vanishing theorem appeared very recently in \cite{mustata-nicaise}, which was posted after the first version of the present paper.}

\medskip
Let us comment on our assumptions on the field $K$.
It is expected that pluripotential theory can be developed on
Berkovich spaces over arbitrary non-Archimedean fields, and as we
mentioned before the first steps in that direction have been taken
in~\cite{CLD} (see also~\cite{gub13} for a nice account on these developments), where the notions of positive forms and  currents are
introduced, partially building upon ideas by Lagerberg~\cite{Lag12}. 

Our approach is geometric, and we restrict our attention to the discretely valued case in order to avoid
the use of formal models over non-noetherian rings. 
We refer to~\cite{Gub98,Gub03} for related works on semipositive metrics in the general setting of a complete
non-Archimedean field. More importantly, Appendix~B relies on the discreteness assumption. 

Furthermore, we use the assumption that $K$ has residue characteristic zero in two ways, through the existence of SNC models and through the cohomology vanishing properties of multiplier ideals. In positive residue characteristic, the existence of SNC models is not known and it is much harder to construct retractions of $\Xan$ onto suitable  complexes (embedded or not). We refer to Berkovich~\cite{Ber2}, Hrushovski-Loeser~\cite{HL}, and Thuillier~\cite{Thu11} for important contributions to the understanding of this problem.

\medskip
The paper is organized as follows. The first two sections present the necessary background on Berkovich spaces and models. The exposition is largely self-contained (hence perhaps a bit lengthy), as we feel that the easy arguments that our setting allows are worth being explained. Section 3 is devoted to dual complexes. The main technical result is Theorem~\ref{T201}, on the existence of blowups attached to decompositions of a dual complex. Section 4 deals with closed $(1,1)$-forms, defined using a numerical equivalence variant of the approach of~\cite{BGS}. In Section 5 we prove some basic properties of $\theta$-psh model functions. Section 6 contains the proof of Theorem A. Section 7 is devoted to the first properties of general $\theta$-psh functions. Theorem B is proved in Section 8. Finally, Appendix A contains a technical result on Lipschitz constants of convex functions, while Appendix B establishes the expected cohomology vanishing properties of multiplier ideals in our setting.

\begin{ackn}
  We would like to thank Jean-Beno\^it Bost, Antoine Chambert-Loir,
  Antoine Ducros, Christophe Soul\'e, and Amaury Thuillier for
  interesting discussions related to the contents of the paper. We are
  particularly grateful to J\'anos Koll\'ar, Osamu Fujino and Mircea
  Musta\c{t}\v{a} for their help regarding Appendix B. 
  Finally, we would like to thank the anonymous referee for a careful reading
  and many useful remarks.

  Our work was carried out at several institutions including the IHES, IMJ,
  the \'Ecole Polytechnique, and the University of Michigan. 
  We gratefully acknowledge their support. 
\end{ackn}



\section{Models of varieties over discrete valuation fields}

\subsection{$S$-varieties}\label{sec:mod}
All schemes considered in this paper are separated and Noetherian, and all ideal sheaves are coherent. Let $R$ be a complete discrete valuation ring with fraction field $K$ and residue field $k$. We shall assume that $k$ has characteristic zero (but we don't require it to be algebraically closed). Let $\unipar\in R$ be a uniformizing parameter and normalize the corresponding 
absolute value on $K$ by $\log|\unipar|^{-1}=1$. Each choice of a field of representatives of $k$ in $R$ then induces an isomorphism $R\simeq k\cro{\unipar}$ by Cohen's structure theorem. 

Write $S:=\Spec R$.  We will use the following terminology. An \emph{$S$-variety} is a flat integral $S$-scheme $\cX$ of finite type. We denote by $\cX_0$ its special fiber and by $\cX_K$ its generic fiber, and we write $\kappa(\xi)$ for the residue field of a point $\xi\in\cX$.  An ideal sheaf $\fa$ on $\cX$ is \emph{vertical} if it is 
co-supported on the special fiber, and a fractional ideal sheaf $\fa$
is vertical if $\unipar^m\fa$ is a vertical ideal sheaf for some
positive integer $m$. 
A \emph{vertical blowup} $\cX'\to\cX$ is the
normalized blowup along a vertical ideal sheaf $\fa$; 
this is the same as the blowup along the integral closure of $\fa$.
We will occasionally consider a blowup along a fractional ideal sheaf $\fa$, which simply means the blowup along $\varpi^m\fa$ for any $m\in\N$ such that $\varpi^m\fa$ is an actual ideal sheaf. 

Except for Appendix B, we will use additive notation for Picard groups, and we write $\cL+\cM:=\cL\otimes\cM$, and $m\cL:=\cL^{\otimes m}$ for $\cL,\cM\in\Pic(\cX)$. We denote by $\Div_0(\cX)$ the group of \emph{vertical Cartier divisors} of $\cX$, \ie those Cartier divisors on $\cX$ that are supported on the special fiber. When $\cX$ is normal, it is easy to see that $\Div_0(\cX)$ is a free $\Z$-module of finite rank and that the natural sequence
$$
0\to\Z\cX_0\to\Div_0(\cX)\to\Pic(\cX)\to\Pic(\cX_K)
$$ 
is exact. The last arrow to the right is (by definition) surjective if $\cX$ is semi-factorial. 
This happens, for instance, when $\cX$ is regular. 
The existence of semi-factorial models is proved over any DVR by P\'epin in~\cite{pepin}. 
In our setting of residue characteristic zero, the existence of
regular models is guaranteed by a result of Temkin, see below. 

\smallskip
Given an $S$-variety $\cX$ let $(E_i)_{i\in I}$ be the (finite) set of
irreducible components of its special fiber $\cX_0$. Endow
each $E_i$ with the reduced scheme structure.
For each subset $J\subset I$, set $E_J:=\bigcap_{j\in J} E_j$. 
\begin{defi}\label{defi:snc}
  Let $\cX$ be an $S$-variety. 
  We say that $\cX$ is \emph{vertically $\Q$-factorial} if each component $E_i$ is $\Q$-Cartier. 
  We say that $\cX$ is \emph{SNC} if:
  \begin{itemize}
  \item[(i)] the special fiber $\cX_0$ has \emph{simple normal crossing support};
  \item[(ii)] $E_J$ is irreducible (or empty) for each $J\subset I$.
  \end{itemize} 
\end{defi}
Condition (i) is equivalent to the following two conditions. 
First $\cX$ is regular. Given a point $\xi\in\cX_0$, let $I_\xi\subset I$ be the set indices $i\in I$ for which
$\xi\in E_i$, and pick a local equation
$z_i\in\cO_{\cX,\xi}$ of $E_i$ at $\xi$ for each $i\in I_\xi$. 
We then also impose that  $\{z_i,\,i\in I_\xi\}$ can be completed
to a regular system of parameters of $\cO_{\cX,\xi}$. 

Condition (ii) is not imposed in the usual definition of a simple normal crossing divisor, but can always be achieved from~(i) by further blowing-up along components of the possibly non-connected $E_J$'s. Since $k$ has characteristic zero, each $S$-variety is a $\Q$-scheme, which is furthermore excellent since it has finite type over $S$. It therefore follows from~\cite{Tem} that for any $S$-variety $\cX$ with smooth generic fiber, there exists a vertical blowup $\cX'\to\cX$ such that $\cX'$ is SNC.

\subsection{Numerical classes and positivity}\label{S205}
Let $\cX$ be a normal projective $S$-variety.

\begin{lem}\label{lem:def-c1}
Assume that $\cL\in\Pic(\cX)$ is \emph{nef} on $\cX_0$, \ie $\cL\cdot C\ge 0$ for all $k$-proper curves $C$ in $\cX_0$. Then $\cL$ is also nef on $\cX_K$, \ie $\cL\cdot C\ge 0$ for all $K$-proper curves $C$ in $\cX_K$ as well.
\end{lem}
We will then simply say that $\cL$ is nef. A \emph{curve} is by definition reduced and irreducible. 
\begin{proof} Let $C$ be $K$-proper curve in $\cX_K$ and let $\cC$ be its closure in $\cX$ equipped with its reduced structure. By \cite[Proposition III.9.7]{Har}, $\cC$ is flat over $S$. The degrees of $\cL|_\cC$ on the generic fiber and on the special fiber therefore coincide, which reads $\cL\cdot C=\cL\cdot\cC_0$. Now $\cC_0$ is an effective linear combination of vertical curves, and the result follows.
\end{proof}

We recall the following standard notions.
\begin{defi}\label{D201}
  Let $\cX$ be a normal projective $S$-variety. 
  \begin{itemize}
  \item[(i)] 
    The space $N^1(\cX/S)$ of \emph{codimension 1 numerical classes} 
    is defined as the quotient of $\Pic(\cX)_\R$ by the subspace spanned by 
    \emph{numerically trivial line bundles}, \ie those $\cL\in\Pic(\cX)$ 
    such that $\cL\cdot C=0$ for all projective curves contained in a 
    fiber of $\cX\to S$.
  \item[(ii)] The \emph{nef cone} $\Nef(\cX/S)\subset N^1(\cX/S)$ 
    is defined as the set of numerical classes $\a\in N^1(\cX/S)$ 
    such that $\a\cdot C\ge 0$ for all projective curves contained in a fiber of $\cX\to S$. 
  \end{itemize}
\end{defi}
Note that the $\R$-vector space $N^1(\cX/S)$ is finite dimensional. Indeed Lemma~\ref{lem:def-c1} shows that the restriction map $N^1(\cX/S)\to N^1(\cX_0/k)$ is
injective, and the latter space is finite dimensional 
since $\cX_0$ is projective over $k$. 
Observe also that $\Nef(\cX/S)$ is a closed convex cone of $N^1(\cX/S)$. Lemma~\ref{lem:def-c1} implies that $\Nef(\cX/S)=\Nef(\cX_0/k)\cap N^1(\cX/S)$ under the injection $N^1(\cX/S)\to N^1(\cX_0/k)$. 

We have the following standard fact:
\begin{lem}\label{lem:ample} 
  Let $\pi:\cX'\to\cX$ be a vertical blowup. 
   \begin{itemize}
  \item[(i)] There exists a $\pi$-ample divisor $A\in\Div_0(\cX')$. 
  \item[(ii)] If $\cL\in\Pic(\cX)$ is ample then there exists $m\in\N$ such that $\pi^*(m\cL|_{\cX_K})$ extends to an ample line bundle $\cL'$ on $\cX'$.
  \end{itemize}
\end{lem}
\begin{proof} 
  By definition, there exists a vertical ideal sheaf $\fa$ on $\cX$ such that 
  $\pi$ is obtained as the blowup of $\cX$ along $\fa$. 
  The universal property of blowups yields a $\pi$-ample Cartier divisor 
  $A$ on $\cX'$ such that $\fa\cdot\cO_{\cX'}=\cO_{\cX'}(A)$, and $A$ 
  is also vertical since $\fa$ is, which proves~(i). If $\cL$ is ample on $\cX$ then $m\pi^*\cL+A$ is ample on $\cX'$ for $m\gg 1$, and~(ii) follows. 
\end{proof}

Recall that an $\R$-line bundle on $\cX$ (resp.\ $\cX_K$) is ample if it can be written as a positive linear combination of ample line bundles. 
\begin{cor}\label{cor:ample} 
  If $L\in\Pic(\cX_K)_\R$ is ample, then $L$ extends to an ample $\R$-line bundle $\cL'\in\Pic(\cX')_\R$ for all sufficiently high vertical blowups $\cX'\to\cX$. 
\end{cor}  
\begin{proof} 
Write $L=\sum_ic_iL_i$ where $c_i\in\R_{>0}$ and $L_i\in\Pic(\cX_K)$ 
is ample for all $i$. We may assume that $L_i$ is very ample for each $i$, 
so that the linear system $|L_i|$ embeds $\cX$ into a suitable
projective space $\P^{r_i}_K$ over $K$. 
Let $\cX_i$ be the normalization of the closure of $\cX_K$ in $\P^{r_i}_S$ and 
$\cL_i$ the restriction of $\cO(1)$ to $\cX_i$. 
Let $\cX'$ be any normal $S$-variety dominating $\cX$ as well as all the $\cX_i$, 
and write $\pi_i:\cX'\to\cX_i$ for the associated vertical blowups.
By Lemma~\ref{lem:ample}~(ii) we can find $m_i\in\N$ and 
ample line bundles $\cL'_i$ on $\cX'$ such that 
$\cL'_i|_{\cX'_K}=\pi_i^*(m_i\cL_i|_{\cX_{i,K}})$.
We can then pick $\cL':=\sum_i\frac{c_i}{m_i}\cL'_i$.
\end{proof}

We shall use the following version of the 
\emph{Negativity Lemma}, \cf~\cite[Lemma 3.39]{KM}. The proof we give is a variant of the  argument used in~\cite[Proposition~2.12]{BdFF}. 

\begin{lem}\label{lem:neg} 
  Assume that $\cX$ is vertically $\Q$-factorial and let $\pi:\cX'\to\cX$ be a vertical blowup. If $D\in\Div_0(\cX')_\R$ is $\pi$-nef then $\pi^*\pi_*D-D$ is effective. 
\end{lem}
The condition that $\cX$ is vertically $\Q$-factorial guarantees that the pull-back of $\pi_*D$ by $\pi$
is well defined.
 
\begin{proof} As a first step, we reduce the assertion to the case where $D\in\Div_0(\cX')_\Q$ is $\pi$-ample. Indeed, the set of vertical $\pi$-ample $\R$-divisors, which is an open convex cone in $\Div_0(\cX')_\R$, is non-empty by (i) of Lemma~\ref{lem:ample}. We may thus choose a basis $A_1,\dots,A_r$ of $\Div_0(\cX')_\R$ made up of $\pi$-ample Cartier divisors. Let $\e=(\e_i)\in\R_+^r$ be such that $D_\e:=D+\sum_i\e_i A_i$ is a $\Q$-divisor. The fact that $D$ is $\pi$-nef means that $D\cdot C\ge 0$ for each curve $C$ contained in a fiber of $\pi$. Since each $A_i$ is $\pi$-ample, it follows from Kleiman's criterion~\cite{Kle} that $D_\e$ is $\pi$-ample on the projective $k$-scheme $\cX_0'$, hence $D_\e$ is also $\pi$-ample on $\cX'$ by~\cite[III.4.7.1]{EGA}. Upon replacing $D$ with $D_\e$ for $\e$ arbitrarily small we may thus assume as desired that $D\in\Div_0(\cX')_\Q$ is $\pi$-ample. 

Now choose $m\gg 1$ such that $\cO_{\cX'}(mD)$ is $\pi$-globally generated, which means that the vertical fractional ideal sheaf $\fa:=\pi_*\cO_{\cX'}(mD)$ satisfies $\fa\cdot\cO_{\cX'}=\cO_{\cX'}(mD)$. It is obvious that $\fa\subset\cO_{\cX}(m\pi_*D)$, hence $\cO_{\cX'}(mD)\subset\cO_{\cX'}(m\pi^*\pi_*D)$, and the result follows. 
\end{proof}



\section{Projective Berkovich spaces and model functions}

\subsection{Analytifications}\label{sec:berko}
Let $Y$ be a proper $K$-scheme. 
As a topological space, its $K$-analytification $Y^\mathrm{an}$ 
in the sense of Berkovich is compact and can be described as follows (cf.~\cite[Theorem 3.4.1]{Ber}). 
Choose a finite cover of $Y$ by Zariski open subsets of the form 
$U=\Spec A$ where $A$ is a finitely generated $K$-algebra.
The Berkovich space $\Uan$ is defined as the set of all multiplicative 
seminorms $|\cdot|_x:A\to\R_+$ extending the given absolute value of $K$, 
endowed with the topology of pointwise convergence. 
It is common usage to write $|f(x)|:=|f|_x$ for $f\in A$ and $x\in U^\an$. 
The space $Y^\mathrm{an}$ is then obtained by gluing together the open sets $\Uan$. 

There is a canonical continuous map $s:Y^\mathrm{an}\to Y$, locally defined on $\Uan$ 
by setting 
$$
s(x)=\left\{f\in A\mid |f(x)|=0\right\}.
$$
The seminorm $|\cdot|_x$ defines a norm on the 
residue field $\kappa(s(x))$, extending the given absolute value on $K$.
In particular, when $Y$ is integral, the set of points $x\in Y^\mathrm{an}$ for which $s(x)$ is the generic
point of $Y$ can be identified with the set of norms on the function field of $Y$ extending the given norm on $K$.

\subsection{GAGA}\label{sec:gaga}
The Berkovich space $Y^\mathrm{an}$ naturally comes with a structure sheaf that we will not
define nor explicitly use. 
We will also not define general $K$-analytic spaces~\cite{Ber,BerIHES} here.
However, we will make use of the general GAGA results in~\cite[3.4]{Ber}. 
For example, any projective $K$-analytic space $X$ is the analytification of a projective
$K$-scheme $Y$, that is, $X=Y^\mathrm{an}$. Further, all line bundles on 
projective Berkovich spaces are induced by line bundles on the underlying scheme.
Similarly, morphisms between projective Berkovich spaces arise from morphisms between the
underlying schemes.

\subsection{Centers}\label{sec:centers}
Now let $\cX$ be a proper $S$-variety. Its generic fiber $\cX_K$ is, in particular, a proper $K$-scheme, so the discussion above applies. Write $X=\cX_K^\an$ for the analytification of $\cX_K$ and 
$s:X\to\cX_K\subset\cX$ for the continuous map defined in~\S\ref{sec:berko}.
Given $x\in X$, denote by $R_x$ 
the corresponding valuation ring in the residue field $\kappa(s(x))$. 
By the valuative criterion of properness,
the map $T_x:=\Spec R_x\to S$ admits a unique lift $T_x\to\cX$ mapping the 
generic point to $s(x)$. In line with valuative terminology~\cite{vaquie}, 
we call the image of the closed point of $T_x$ in $\cX$ the \emph{center} of $x$ 
on $\cX,$ and denote it by $\cent_\cX(x)$. It is a specialization of $s(x)$ 
in $\cX$. It also belongs to $\cX_0$ since it maps to the closed point of 
$S$ by construction. The map $\cent_\cX:\cX_K^\mathrm{an}\to\cX_0$ so defined is 
anti-continuous, \ie preimages of open sets are closed and vice versa. 
It is referred to as  the \emph{reduction map} in rigid geometry. 

\subsection{Models}\label{sec:modberko}
From now on we let $X$ be a given smooth connected projective $K$-analytic space in the sense of Berkovich. By a \emph{model} of $X$ we will mean a \emph{normal and projective} $S$-variety $\cX$ together with the datum of an isomorphism $\cX^\mathrm{an}_K\simeq X$. By GAGA, the latter is equivalent to an isomorphism between $\cX_K$ and the (smooth, connected) algebraic variety $Y$ underlying $X=Y^\mathrm{an}$. 

In particular, the set $\cM_X$ of models of $X$ is non-empty. Indeed, given an embedding $Y$ into a suitable projective space $\P^m_K$ we can take $\cX$ as the normalization of the closure of $Y$ in $\P^m_S$.
A similar construction shows that $\cM_X$ becomes a directed set by 
declaring $\cX'\ge\cX$ if there exists a vertical blowup $\cX'\to\cX$ (which is then unique).

For any model $\cX$ of $X$ and any irreducible component $E$ of $\cX_0$, there exists a unique point $x_E\in \Xan\simeq\cX_K^\mathrm{an}$ whose 
center on $\cX$ is the generic point of $E$. Such points will be called 
\emph{divisorial points}.\footnote{Divisorial points are called \emph{Shilov boundaries} in~\cite{YZ09}. They are usually referred to as Type II points when $X$ is a curve, see~\cite[1.4.4]{Ber}}
Observe that the local ring of the scheme $\cX$ at the generic point of $E$ is precisely the valuation ring of the valuation $x_E$.

The set $\Xdiv$ of divisorial points is dense in 
$\Xan$, see Corollary~\ref{C208} below and also~\cite{poineau}.

\subsection{Model functions}\label{sec:modelfnts}
Let $\cX$ be a model of $X$. By Noetherianity, each vertical fractional ideal sheaf $\fa$ on $\cX$ is locally generated by a finite set of rational functions on $\cX$, and thus defines a continuous function $\log|\fa|\in C^0(\Xa)$ by setting
\begin{equation}\label{equ:log}
\log|\fa|(x):=\max\left\{\log|f(x)|\mid f\in\fa_{\cent_\cX(x)}\right\}. 
\end{equation}
In particular, each vertical Cartier divisor $D\in\Div_0(\cX)$ defines a vertical fractional ideal sheaf $\cO_\cX(D)$, hence a continuous function 
$$
\f_D:=\log|\cO_\cX(D)|.
$$ 
Note that $\f_{\cX_0}$ is the constant function $1$ since $\log|\unipar|^{-1}=1$. 
Since models are assumed to be normal, a vertical divisor $D$ is uniquely determined by the values $\f_D(x_E)$ at divisorial points $x_E$, and we have in particular $\f_D\ge0$ iff $D$ is effective.
The map $D\mapsto\f_D$ extends by linearity to an injective map $\Div_0(\cX)_\R\to C^0(\Xa)$. 

Following~\cite{Y} we introduce the following terminology. 
\begin{defi}
A \emph{model function}\footnote{Model functions are called \emph{algebraic} in~\cite{CL1} and \emph{smooth} in~\cite{CL2}.} is a function $\f$ on $\Xan$ such that there exists a model $\cX$ and a divisor $D \in \Div_0(\cX)_\Q$ with $\f = \f_D$. We let $\cD(\Xa)=\cD(\Xa)_\Q$ be the space of model functions on $\Xan$.
\end{defi}

We shall also occasionally consider the similarly defined spaces $\cD(\Xa)_\Z$ and $\cD(\Xa)_\R$. 

As a matter of terminology, we say that a model $\cX$ is a \emph{determination} of a model function $\f$ if $\f=\f_D$ for some $D\in\Div_0(\cX)_\Q$. By the above remarks we have a natural isomorphism
$$
\varinjlim_{\cX \in \cM_X} \Div_0(\cX)_\Q\simeq\cD(\Xa)\subset C^0(X).
$$

\medskip

The next result summarizes the key properties of model functions. Since our setting does not require any machinery from rigid geometry we provide direct arguments for the convenience of the reader. 

\begin{prop}\label{prop:gen} 
  For each model $\cX$ of $X$, the subgroup of $C^0(X)$ spanned by the functions $\log|\fa|$, with $\fa$ ranging over all vertical (fractional) ideal sheaves of $\cX$, coincides with $\cD(\Xa)_\Z$. 
It is furthermore stable under max and separates points of $\Xan$. 
\end{prop}
\begin{proof} If $\fa$ is a vertical fractional ideal sheaf on a given model $\cX$ then $\fa':=\unipar^m\fa$ is a vertical ideal sheaf for some $m\in\N$ and we have $\log|\fa|=\log|\fa'|-m$, so it is enough to consider vertical ideal sheaves. 

Observe first that $\log|\fa|$ belongs to $\cD(\Xa)_\Z$. Indeed if $\cX'\to\cX$ denotes the normalization of the blowup of $\cX$ along $\fa$, then the Cartier divisor $D$ on $\cX'$ such that $\fa\cdot\cO_{\cX'}=\cO_{\cX'}(D)$ satisfies $\f_D=\log|\fa|$. Conversely, let $\f\in\cD(\Xa)_\Z$, and let us show that $\f$ can be written as 
$$
\f=\log|\fa|-\log|\fb|
$$
with $\fa,\fb$ vertical ideal sheaves on $\cX$. By definition, $\f$ is determined by $D\in\Div_0(\cX')$ for some vertical blowup $\pi:\cX'\to\cX$. By Lemma~\ref{lem:ample} we may choose a $\pi$-ample vertical Cartier divisor $A\in\Div_0(\cX')$. Both sheaves $\cO_{\cX'}(mA)$ and $\cO_{\cX'}(D+mA)$ are then $\pi$-globally generated for $m\gg 1$. If we introduce the vertical fractional ideal sheaves 
$\fa:=\pi_*\cO_{\cX'}(D+mA)$ 
and 
$\fb:=\pi_*\cO_{\cX'}(mA)$,
then the $\pi$-global generation property yields 
$\fa\cdot\cO_{\cX'}=\cO_{\cX'}(D+mA)$
and  
$\fb\cdot\cO_{\cX'}=\cO_{\cX'}(mA)$.
It follows that $\f_{mA}=\log|\fa|$ and $\f_{D+mA}=\log|\fb|$, and hence $\f_D=\log|\fa|-\log|\fb|$. It remains to replace $\fa$ and $\fb$ with $\unipar^p\fa$ and $\unipar^p\fb$ with $p\gg 1$, so that they become actual ideal sheaves. 

We next prove that $\cD(X)_\Z$ is stable under max. Given $\f,\f'\in\cD(X)_\Z$, choose a model $\cX$ on which both functions are determined, say by $D,D'\in\Div_0(\cX)$, respectively. We then have
$$
\max\{\f_D,\f_{D'}\}=\log|\fa|,
$$
with $\fa:=\cO_{\cX}(D)+\cO_\cX(D')$, which shows that $\max\{\f_D,\f_{D'}\}\in\cD(X)_\Z$. 

In order to get the separation property, we basically argue as in~\cite[Corollary 7.7]{Gub98}, which in turn relied on~\cite[Lemma 2.6]{BL}. Let $\cX$ be a fixed model and pick two distinct points $x\neq y\in\Xan$. If $\xi:=c_\cX(x)$ is distinct from $c_\cX(y)$ then $\log|\fm_\xi|$ already separates $x$ and $y$. Otherwise, let $\cU=\Spec\cA$ be an open neighborhood of $\xi$ in $\cX$. By definition of $\cU_K^\an$ there exists $f\in\cA$ such that $|f(x)|\neq|f(y)|$. Since the scheme $\cX$ is Noetherian, $\cO_\cU\cdot f$ extends to a (coherent) ideal sheaf $\fa$ on $\cX$. For each positive integer $m$ the ideal sheaf $\fa_m:=\fa+(\varpi^m)$ is vertical on $\cX$, and we have
$$
\log|\fa_m|=\max\{\log|f|,-m\}
$$
at $x$ and $y$, so we see that $\log|\fa_m|\in\cD(X)_\Z$ separates $x$ and $y$ for $m\gg 1$. 
\end{proof} 

Thanks to the ``Boolean ring version'' of the Stone-Weierstrass theorem,  we get as a consequence the following crucial result, which is equivalent to~\cite[Theorem 7.12]{Gub98} (compare~\cite[Lemma 3.5]{Y} and the remark following it).  
\begin{cor}\label{cor:dense}
  The $\Q$-vector space $\cD(\Xa)$ is stable under max and separates points. 
  As a consequence, it is dense in $C^0(\Xa)$ for the topology of uniform convergence. 
\end{cor}

Corollary~\ref{cor:dense} in turn implies the following result, which corresponds to~\cite[Lemma~2.4]{YZ09}. We reproduce the short proof for completeness.
\begin{cor}\label{C208}
  The set $\Xdiv$ of divisorial points is dense in $\Xan$. 
\end{cor}
\begin{proof}
  Since $X$ is a compact Hausdorff space, Urysohn's lemma applies, so 
  it suffices to prove that any continuous function vanishing on $\Xdiv$
  vanishes on all of $X$.

  So pick $\f\in C^0(\Xa)$ vanishing on $\Xdiv$ and $\e>0$ rational. 
  By Corollary~\ref{cor:dense} there exists a model $\cX$ and a 
  divisor $D\in\Div_0(\cX)_\Q$ such that $|\f-\f_D|\le\e$ on $\Xan$. Since $\varphi$ vanishes 
on all divisorial points corresponding to irreducible components of $\cX_0$, it follows that both
divisors $\e\cX_0\pm D\in\Div_0(\cX)_\Q$ are effective, proving
  $|\f_D|\le\e$ and hence $|\f|\le2\e$ on $\Xan$.
\end{proof}

The collection of finite dimensional spaces $\Div_0(\cX)^*_\R \simeq \cD(\cX)^*_\R$ endowed with the transpose of pull-back morphisms on divisors and the topology of the pointwise convergence forms an inductive system, and we have:
\begin{cor}\label{cor:gelfand} For each model $\cX$, let $\ev_\cX:\Xan\to\Div_0(\cX)^*_\R$ be the evaluation map defined by $\langle\ev_\cX(x),D\rangle=\f_D(x)$. Then the induced map 
$$
\ev:\Xan\to\varprojlim_{\cX\in \cM_X}\Div_0(\cX)_\R^*\simeq\cD(\Xa)_\R^*
$$
is a homeomorphism onto its image.
\end{cor}
The image of this map will be described in Corollary~\ref{cor:retract}. 
\begin{proof} 
  The map in question is continuous since any model function is continuous.
  It is injective by Corollary~\ref{cor:dense}.
  Since $\Xan$ is compact and since 
  $\varprojlim_{\cX\in \cM_X}\Div_0(\cX)_\R^*$ is Hausdorff,
  we conclude that the map is a homeomorphism onto its image.
\end{proof}




\section{Dual complexes}
In this section we define, following~\cite{KS}, an embedding 
of the dual complex $\D_\cX$ of an SNC
model $\cX$ into the Berkovich space $\Xan$. This construction is essentially a special case of~\cite{Ber2} (see also~\cite{Thu,ACP12}), but the present setting allows a much more elementary and explicit approach. We also explain how to construct (not necessarily SNC) models dominating $\cX$ from suitable subdivisions of $\D_\cX$, adapting some of the toroidal techniques of~\cite{KKMS}. 

\subsection{The dual complex of an SNC model}
Let $\cX$ be an SNC model of $X$. 
The image of the evaluation map $\ev_\cX:\Xan\to\Div_0(\cX)_\R^*$ defined in Corollary~\ref{cor:gelfand} then 
admits the structure of a rational simplicial complex, defined as follows. 
Write the special fiber as $\cX_0=\sum_{i\in I}b_iE_i$, where $b_i\in\N^*$ 
and $(E_i)_{i\in I}$ are the irreducible components.
Let $x_{E_i}\in \Xan$ be the associated divisorial points and set 
$e_i:=\ev_\cX(x_{E_i})\in\Div_0(\cX)_\Q^*$.
Recall from Definition~\ref{defi:snc} that for each $J\subset I$ the intersection 
$E_J:=\bigcap_{j\in J} E_j$ is either empty or a smooth irreducible $k$-variety.
For each $J \subset I$ such that $E_J\ne\emptyset$, let 
$\hat{\sigma}_J\subset\Div_0(\cX)^*_\R$ be the
simplicial cone defined by 
$\hat{\sigma}_J := \sum_{j \in J} \R_+ e_j$. 
These cones naturally define a (regular) fan $\hat\D_\cX$ in $\Div_0(\cX)^*_\R$. 
Slightly abusively, we shall also denote by $\hat{\D}_\cX$ the 
support of this fan, that is, the union of all the cones $\hat{\sigma}_J$. 
We then define the \emph{dual complex}\footnote{The dual complex is called the \emph{Clemens polytope} in~\cite{KS}.} of $\cX$ by
$$
\D_\cX:= \hat{\D}_\cX \cap \left\{ \langle \cX_0, \cdot \rangle = 1 \right\}.
$$
Each $J\subset I$ such that $E_J\neq\emptyset$ corresponds to a simplicial face 
\begin{equation*}
  \sigma_J
  := \hat{\sigma}_J\cap\left\{ \langle \cX_0, \cdot \rangle = 1 \right\}
  =\Conv\{e_j\mid j\in J\}
\end{equation*}
of dimension $|J|-1$ in $\D_\cX$, where $\Conv$ denotes convex hull.
This endows $\D_\cX$ with the 
structure of a (compact rational) 
simplicial complex, such that $\sigma_J$ is a face of $\sigma_L$ iff $J\supset L$.        

\subsection{Embedding the dual complex in the Berkovich space}
\begin{thm}\label{thm:construct-dual} Let $\cX$ be any SNC model of $X$.
  \begin{itemize}
  \item[(i)] 
    The image of the evaluation map $\ev_\cX: \Xan \to \Div_0(\cX)^*_\R$ 
    coincides with $\D_\cX$.
  \item[(ii)] There exists a unique continuous (injective) map 
    $\emb_\cX: \D_\cX \to \Xan$ such that:
    \begin{itemize}
    \item[(a)]
      $\ev_\cX \circ\emb_\cX$ is the identity on $\D_\cX$;
    \item[(b)]
      for $s\in\D_\cX$, the center of $\emb_\cX(s)$ on $\cX$ is the 
      generic point $\xi_J$ of $E_J$ for the unique subset $J\subset I$ such that $s$ is 
      contained in the relative interior of $\sigma_J$.
    \end{itemize}
  \end{itemize}
\end{thm}
The proof is given in~\S\ref{S206}. Let us derive some consequences.

For any vertical blowup $\pi: \cX \to \cY$ between models, the natural map ${}^t\pi^*: \Div_0(\cX)^* \to \Div_0(\cY)^*$
maps $\D_\cX$ onto $\D_\cY$ since ${}^t\pi^* \circ \ev_\cX = \ev_\cY$ by definition. We may thus form the inverse limit
$\varprojlim_{\cX\ \text{SNC}} \D_\cX$, and we have
\begin{cor}\label{cor:retract}
  The maps $\ev_\cX: \Xan \to \D_\cX\subset\Div_0(\cX)_\R^*$ induce a homeomorphism
  \begin{equation}\label{e201}
    \ev:\Xan \to \varprojlim_{\cX\  \text{SNC}} \D_\cX.
  \end{equation}
\end{cor}
\begin{rmk}\label{R303}
  Corollary~\ref{cor:retract} can be found in~\cite[p.77, Theorem 10]{KS} and is an example
  of a result exhibiting a non-Archimedean space as an inverse limit of polyhedral objects.
  Other examples can be found in~\cite[Theorem~6.22]{valtree},
  \cite[Theorem~1.13]{hiro},
  \cite[Theorem~1.1]{Payne},
  \cite[Theorem~2.21]{BR},
  \cite[Theorem~13.2.4]{HL},
  \cite[Theorem~4.9]{jonmus},
  \cite[Proposition~6.1]{HLP},
  and~\cite[Theorem~2.3]{BdFFU}.
\end{rmk}
\begin{proof}[Proof of Corollary~\ref{cor:retract}]
  The map $\ev$ is well-defined by 
  Theorem~\ref{thm:construct-dual}~(i).
  It  is a homeomorphism onto its image by
  Corollary~\ref{cor:gelfand}
  and the fact that any model is dominated by an SNC model.
  As $\Xan$ is compact,  we only
  need to show that $\ev(\Xan)$ is dense in $\varprojlim\D_\cX$. 
  Pick $s=(s_\cX)_\cX\in\varprojlim\D_\cX$ and
  fix an SNC model $\cX$. If $\cY$ is an SNC model dominated by 
  $\cX$, then $\ev_\cX\circ\emb_\cX=\id$ yields 
  $\ev_\cY(\emb_\cX(s_\cX))=s_\cY$. 
  Hence $s=\lim_\cX\ev(\emb_\cX(s_\cX))\in\overline{\ev(\Xan)}$.
\end{proof}
\begin{defi}
  For any SNC model $\cX$ we define a continuous map
  $\retr_\cX: \Xan \to \Xan$ by
  \begin{equation*}
    \retr_\cX:=\emb_\cX\circ\ev_\cX.
  \end{equation*}
\end{defi}

It follows from Theorem~\ref{thm:construct-dual} that 
$\retr_\cX$ satisfies $\retr_\cX\circ\retr_\cX=\retr_\cX$ and  $\retr_\cX(x)=x$ iff $x\in\emb_\cX(\D_\cX)$.
Hence we view $\retr_\cX$ as a \emph{retraction} of 
$\Xan$ onto the image of the embedding $\emb_\cX:\D_\cX\to \Xan$.

\begin{lem}\label{L207} 
  The retraction  map $\retr_\cX$ satisfies the following properties: 
\begin{itemize}
\item[(i)] $\cent_\cX(x)\in\overline{\{\cent_\cX(\retr_\cX(x))\}}$ for
  all $x\in\Xan$; 
  more precisely, $c_\cX(\retr_\cX(x))$ is the generic point of $E_J$, where
  $J\subset I$ is the set of indices $j$ for which $c_\cX(x)\in E_j$;
\item[(ii)] $\f_D\circ\retr_\cX=\f_D$ for all $D\in\Div_0(\cX)_\R$. 
\end{itemize} 
\end{lem}
\begin{proof}
  By definition of $\cent_\cX$ we have $\cent_\cX(x)\in E_i$ for a given $i\in I$ iff $\langle\ev_\cX(x),E_i\rangle>0$, and it follows that 
  $\ev_\cX(x)$ lies in the relative interior of the simplex 
  $\sigma_J$ for the maximal $J\subset I$ such that $\cent_\cX(x)\in E_J$.
  Property~(b) in Theorem~\ref{thm:construct-dual} then shows that
  $\cent_\cX(\retr_\cX(x))$ is the generic point of $E_J$, which proves (i). 
  
  Let us prove (ii). For each $x\in\Xan$ we have
   $$
  \f_D(\retr_\cX(x))=\langle D,\ev_\cX(\retr_\cX(x))\rangle=\langle D,\ev_\cX\circ\emb_\cX\circ\ev_\cX(x)\rangle=\langle D,\ev_\cX(x)\rangle=\f_D(x),
  $$
 using the identity $\ev_\cX\circ\emb_\cX=\id$. 
\end{proof}

\begin{prop}\label{P206}
  If $\cX\ge\cY$ are two SNC models, then 
  \begin{itemize}
  \item[(i)] $\ev_\cY\circ\retr_\cX=\ev_\cY$;
  \item[(ii)] $\retr_\cY\circ\retr_\cX=\retr_\cY$;
  \item[(iii)]  $\retr_\cX\circ\emb_\cY=\emb_\cY$;
  \item[(iv)] $\retr_\cX\circ\retr_\cY=\retr_\cY$.
  \end{itemize}
\end{prop}
Note that~(iii) says that the image in $\Xan$ of $\D_\cY$ 
is contained in the image of $\D_\cX$.
\begin{proof} 
  Property (i) amounts to the fact that $\f_D\circ\retr_\cX=\f_D$ for all $D\in\Div_0(\cY)$, 
  which is a special case of Lemma~\ref{L207}~(ii). 
  Postcomposing~(i) with $\emb_\cY$ we then get~(ii).
  
  Let us now prove (iii). The map $\emb'_\cY:=\retr_\cX\circ\emb_\cY:\D_\cY\to\Xan$
  is continuous, and~(i) implies 
  that $\ev_\cY\circ\emb'_\cY=\ev_\cY\circ\emb_\cY=\id$ on $\D_\cY$. 
  By the uniqueness part of 
  Theorem~\ref{thm:construct-dual} it suffices
  to prove that $\cent_\cY\circ\emb'_\cY=\cent_\cY\circ\emb_\cY$ on $\D_\cY$.
  Pick $s\in\D_\cY$ and set $x:=\emb_\cY(s)$, $x':=\emb'_\cY(s)$. 
  On the one hand,~(ii) shows that
    $$
  \retr_\cY(x')=\retr_\cY\circ\retr_\cX\circ\emb_\cY(s)=\retr_\cY\circ\emb_\cY(s)=x,
  $$
  so $\cent_\cY(x')\in\overline{\{\cent_\cY(x)\}}$ by (i) of 
  Lemma~\ref{L207}. On the other hand $\retr_\cX(x)=x'$ by definition, so
  $\cent_\cX(x)\in\overline{\{\cent_\cX(x')\}}$ and hence 
  $\cent_\cY(x)\in\overline{\{\cent_\cY(x')\}}$ by continuity of the map $\cX\to\cY$ for the Zariski topology.

  Finally~(iv) follows by postcomposing~(iii) with~$\ev_\cY$.
\end{proof}
\begin{defi}
  We define the subset $\Xqm\subset\Xan$ of \emph{quasimonomial points} 
  as
  \begin{equation*}
    \Xqm:=\bigcup_\cX\emb_\cX(\Delta_\cX),
  \end{equation*}
  where $\cX$ ranges over SNC models of $\cX$.
\end{defi}  
\begin{rmk}
The set $\Xqm$ coincides with the set of real valuations on the function field of $X$
whose restriction to $K$ agrees with the given valuation, and 
are Abhyankar in the sense that the sum of their rational rank
and their transcendence degree is equal to $\dim X+1$; 
see~\cite[Proposition~3.7]{jonmus}.\footnote{The rational rank of such a valuation $v$ is 
  defined as the dimension of the $\Q$-vector space generated by the value group of $v$.
  The transcendence degree of $v$ is the transcendence degree of the field extension $k(v)/k$,
  where $k(v)=\{v\ge 0\}/\{v>0\}$ is the residue field of $v$.}
\end{rmk}

\begin{cor}\label{C205}
  We have $\lim_\cX \retr_\cX=\id$ pointwise on $\Xan$. Hence $\Xqm$ is dense in $\Xan$.
\end{cor}
Of course, we already knew from Corollary~\ref{C208} that 
$\Xdiv\subset\Xqm$ is dense in $\Xan$.
\begin{proof}
  By Corollary~\ref{cor:retract} it suffices to show that 
  $\lim_\cX\ev\circ \retr_\cX=\ev$, which amounts to proving 
  $\lim_\cX\ev_\cY\circ \retr_\cX=\ev_\cY$ for each $\cY$. This follows from (i) of Proposition~\ref{P206}. 
\end{proof}

\subsection{Proof of Theorem~\ref{thm:construct-dual}}\label{S206}
  Proving the inclusion $\ev_\cX(\Xan)\subset\D_\cX$ is a matter of 
  unwinding definitions. The reverse inclusion will follow from~(a).
  Hence~(ii) implies~(i).

The proof of (ii) is essentially the same as that of~\cite[Proposition 3.1]{jonmus}. It is also closely related to~\cite[Lemma 5.6]{Ber2} and~\cite[Corollaire 3.13]{Thu}. Fix a subset $J\subset I$ with $E_J\neq\emptyset$, let $\xi_J$ be its generic point and let $\sigma_J$ be the corresponding face of $\D_\cX$. It will be enough to show the existence and uniqueness of a continuous map $\emb_\cX:\sigma_J\to\Xan$ satisfying (a) and (b) of Theorem~\ref{thm:construct-dual} for $s\in\sigma_J$. 

For each $j\in J$ pick a local equation $z_j\in\cO_{\cX,\xi_J}$ of $E_j$, so that $(z_j)_{j\in J}$ is a regular system of parameters of $\cO_{\cX,\xi_J}$ thanks to the SNC condition. 
After choosing a field of representatives of $\kappa(\xi_J)$ in $\cO_{\cX,\xi_J}$, Cohen's theorem yields an isomorphism 
\begin{equation}\label{equ:cohen}
\hcO_{\cX,\xi_J}\simeq\kappa(\xi_J)\cro{t_j,j\in J}
\end{equation}
sending $z_j$ to $t_j$. 

To prove uniqueness, suppose $\emb_\cX,\emb'_\cX:\sigma_J\to\Xan$ 
are two continuous maps satisfying (a) and (b) for $s\in\sigma_J$. 
Property (a) means that the valuations $\val_{\cX,s}$ and $\val'_{\cX,s}$ 
on $\hcO_{\cX,\xi_J}$ defined by
$$
\val_{\cX,s}(f):=-\log|f(\emb_\cX(s)|
\quad\text{and}\quad
\val'_{\cX,s}(f):=-\log|f(\emb'_\cX(s)|
$$ 
both take value $s_j$ on $z_j$. 
By property~(b) it follows that when $s$ belongs to the relative interior $\rel(\sigma_J)$, 
the valuations $\val_{\cX,s}$, $\val_{\cX,s}'$ have center $\xi_J$ on $\cX$, and hence extend by continuity to the completion $\hcO_{\cX,\xi_J}$. The isomorphism (\ref{equ:cohen}) enables us to write any given $f\in\hcO_{\cX,\xi_J}$ as $f = \sum_{\alpha \in \N^{J}} f_\alpha z^\alpha$ with $f_\alpha \in\hcO_{\cX,\xi_J}$, in such a way that each non-zero 
  $f_\a$ is a unit. For any $s\in\rel(\sigma_J)$ we then have 
  $$
  \val_{\cX,s}(f_\alpha z^\alpha)=\langle s,\alpha\rangle=\val'_{\cX,s}(f_\a z^\a)
  $$
  for each $\a\in\N^J$. If $(s_j)_{j\in J}$ is $\Q$-linearly independent then these numbers are furthermore mutually distinct as $\a$ ranges over $\N^J$, and the ultrametric property yields
\begin{equation}\label{equ:monoval}
  \val_{\cX,s}(f)=\min_{\a\in\N^J}\langle s,\alpha\rangle=\val'_{\cX,s}(f).
 \end{equation}
  We conclude that $\emb_\cX(s)=\emb'_\cX(s)$ on the dense set of points $s\in\rel(\sigma_J)$ such that $(s_j)_{j\in J}$ is $\Q$-linearly independent; hence $\emb_\cX=\emb_\cX'$ on $\sigma_J$ by continuity.  
  
Now we turn to existence.
Given $s\in\sigma_J$, define $v$ as the 
monomial valuation on the ring of formal power series
$\kappa(\xi_J)\cro{t_j,\,j\in J}$ taking values $v(t_j)=s_j$, $j\in J$. 
In other words, the value of $v$ on an element
$$
g=\sum_{\a\in\N^J}g_\a t^\a\in\kappa(\xi_J)\cro{t_j,\,j\in J}
$$ 
is given by
\begin{equation}\label{equ:mon}
v(g)=\min\{\langle s,\a\rangle\mid g_\a\neq 0\}.
\end{equation}
Using the isomorphism (\ref{equ:cohen}) we may thus define
$\val_{\cX,s}$ as the restriction to $\cO_{\cX,\xi_J}$ of the 
pull-back of $v$.
The center of $\val_{\cX,s}$ is then equal to the generic point of
$\bigcap_{s_j>0}\left\{z_j=0\right\}$, \ie the generic point of
$E_{J'}$ where $\sigma_{J'}$ is the face containing $s$ in its
relative interior. The continuity of $s\mapsto\val_{\cX,s}(f)$ on
$\sigma_J$ is also easy to see using (\ref{equ:mon}). Setting 
\begin{equation*}
  \emb_\cX(s):=\exp\left(-\val_{\cX,s}\right)
\end{equation*}
therefore concludes the proof.   
  
\begin{rmk}\label{rmk:otherpoint} 
  For each $\xi\in\cX_0$ let $I_\xi$ be the set of indices $i\in I$ for
  which $\xi\in E_i$. Arguing as above shows that there exists a
  unique way to define, for each $s\in\sigma_{I_\xi}$, a valuation 
  $\val_{\hcX_\xi,s}$ on $\hcX_\xi:=\Spec\hcO_{\cX,\xi}$, if we impose that:
  \begin{itemize}
  \item $\val_{\hcX_\xi,s}$ is centered at $\xi_J$ for $s\in\rel(\sigma_J)\subset\sigma_{I_\xi}$;
  \item $\val_{\hcX_\xi,s}(E_i)=s_i$ for each $i\in I_\xi$;
  \item $s\mapsto\val_{\hcX_\xi,s}(f)$ is continuous for each $f\in\hcO_{\cX,\xi}$. 
  \end{itemize}
  Indeed, choose a regular system of parameters $(z_i)_{i\in L}$ of
  $\cO_{\cX,\xi}$ 
  such that $z_i$ is a local equation of $E_i$ for $i\in I_\xi\subset L$, 
  and a field of representatives of $\kappa(\xi)$ in $\cO_{\cX,\xi}$. 
  We then have an isomorphism
  $\hcO_{\cX,\xi}\simeq\kappa(\xi)\cro{t_i,\,i\in L}$ 
  under which $\val_{\hcX_\xi,s}$ corresponds to the 
  monomial valuation taking value $s_i$ on $t_i$ 
  for $I\in I_\xi$, and $0$ on $t_i$ for $i\in L\setminus I_\xi$. 
  Note that $\val_{\cX,s}$ is then the image of $\val_{\hcX_\xi,s}$ 
  under the natural morphism $\hcX_\xi\to\cX$. 
\end{rmk}  

\begin{rmk}
Although we shall not use it, there exists a deformation retraction of $X$ onto
the image of the dual complex $\emb_\cX(\D_\cX)$ in $X$, see~\cite[Theorem~3.26]{Thu} and  
\cite[Theorem~3.1.3]{nicaise-xu}.
\end{rmk}  

\subsection{Functions on dual complexes}

\begin{prop}\label{prop:PA} 
  Let $\cX$ be an SNC model of $X$ and let $\fa$ be a vertical fractional 
  ideal sheaf on $\cX$. Then $\f:=\log|\fa|\in\cD(\Xa)$ satisfies:
  \begin{itemize} 
  \item[(i)] 
    $\f\circ\emb_\cX$ is piecewise affine and convex 
    on each face of $\D_\cX$;
  \item[(ii)]
    $\f\le\f\circ \retr_\cX$.
  \end{itemize}
 \end{prop}
\begin{cor}\label{C304}
  Let $\cX$ be an SNC model of $X$ and 
  $\p\in\cD(\Xan)$ a model function. Then 
  $\p\circ\emb_\cX$ is piecewise affine on the faces of $\D_\cX$.
  Further, $\p\circ\emb_\cX$ is affine on all faces iff $\p$ is determined on $\cX$.
\end{cor}

\begin{proof}
  By Proposition~\ref{prop:gen}, we can write $\p=\sum_{i=1}^mc_i\log|\fa_i|$
  for vertical ideal sheaves $\fa_i$ on $\cX$ and rational numbers $c_i$. 
  According to Proposition~\ref{prop:PA}, each function $\log|\fa_i|\circ\emb_\cX$ is piecewise 
  affine on the faces of $\D_\cX$; hence so is $\p$.

  The second point follows from the fact that a function on $\D_\cX$ is affine on all the faces of $\D_\cX\subset\Div_0(\cX)^*_\R$ iff it comes from a linear form, \ie an element of $\Div_0(\cX)_\R$. 
\end{proof}

\begin{proof}[Proof of Proposition~\ref{prop:PA}]
  Upon multiplying by $\unipar^m$ with $m\gg 1$, we may assume that 
  $\fa\subset\cO_\cX$ is a vertical ideal sheaf. Pick $J\subset I$ 
  such that $E_J$ is non-empty, choose a point $\xi\in E_J$ and let $f_1,\dots,f_M$ be generators of $\fa\cdot\cO_{\cX,\xi}$. With the notation introduced in the proof of Theorem~\ref{thm:construct-dual} we then have   
  \begin{equation}\label{equ:logfa}
    \log|\fa|(\emb_\cX(s))=\max_{1\le m\le M}-\val_{\cX,s}(f_m).
  \end{equation}
  By (\ref{equ:mon}) each function $s\mapsto -\val_{\cX,s}(f_m)$ 
  is piecewise affine and convex on $\sigma_J$, proving~(i).

  To prove~(ii), pick any $x\in\Xan$, set $\xi:=\cent_\cX(x)$ and let $I_\xi\subset I$ be the set of indices $i\in I$ such that $\xi\in E_i$. Arguing similarly with generators of $\fa\cdot\cO_{\cX,\xi}$, it is enough to show that $|f(x)|\le|f(\retr_\cX(x))|$ for each $f\in\cO_{\cX,\xi}$. Note that the seminorm $f\mapsto|f(x)|$ extends by continuity to $\hcO_{\cX,\xi}$ since $\xi=c_\cX(x)$. Writing, in the notation of Remark~\ref{rmk:otherpoint}, $f=\sum_{\a\in\N^L}f_\a z^\a\in\hcO_{\cX,\xi}$ we then have 
  $$
  |f(x)|\le\sup_{f_\a\neq 0}\prod_{i\in I_\xi}|z_j(x)|^{\a_i}
  $$ 
  by the ultrametric property, using that $|f_\a(x)|=1$ since each
  non-zero $f_\a\in\hcO_{\cX,\xi}$ is a unit. On the other hand, if we
  set $s_i:=-\log|z_i(x)|$ for $i\in I_\xi$ then we have by definition
  $p_\cX(x)=\emb_\cX(s)$; hence
  $$
  \sup_{f_\a\neq 0}\prod_{i\in I_\xi}|z_i(x)|^{\a_i}=|f(p_\cX(x))|
  $$
  and the result follows. 
  \end{proof}
Let $\PA(\D_\cX)_\Z$ be the set of all continuous functions $h:\D_\cX\to\R$ whose restriction to each face of $\D_\cX$ is piecewise affine, with gradients given by $\Z$-divisors $D\in\Div_0(\cX)$. 

\begin{defi}\label{defi:fah} Let $h\in\PA(\D_\cX)_\Z$. For each $J\subset I$ such that $E_J\neq\emptyset$ we set $\cX_J:=\cX\setminus\bigcup_{i\in I\setminus J}E_i$ and define a vertical fractional ideal sheaf $\fa_h$ on $\cX$ by letting for each $J$ 
\begin{equation}\label{equ:faxj}
  \fa_h|_{\cX_J}:=\sum\left\{\cO_{\cX_J}(D)\mid D\in\Div_0(\cX)\
    \text{and $\langle D,\cdot\rangle\le h$ on $\sigma_J$}\right\}.
\end{equation}
\end{defi}
Note that these locally defined sheaves glue well together, and that $\log |\fa_h| \circ \emb_\cX$  is equal to the convex envelope of $h$ on each face of $\D_\cX$.

\subsection{Subdivisions and vertical blowups}\label{S201}
Let $\cX$ be an SNC model.
A \emph{subdivision} $\D'$ of $\D_\cX$ is a compact rational polyhedral complex of $\Div_0(\cX)_\R^*$ refining $\D_\cX$. Each subdivision $\D'$ is thus of the form $\hat{\D}'\cap\{\langle\cX_0,\cdot \rangle=1\}$ where $\hat{\D}'$ is a rational fan refining $\hat{\D}_\cX$. A subdivision $\D'$ is \emph{simplicial} if its faces are simplices.

A subdivision $\D'$ is \emph{projective} if it admits a 
\emph{strictly convex support function}, 
that is, a function  $h\in\PA(\D_\cX)_\Z$ that is convex on each face 
of $\D_\cX$ and such that $\D'$ is the coarsest subdivision of $\D_\cX$ on each of whose faces $h$ is affine.

\begin{thm}\label{T201}
  Let $\cX$ be an SNC model of $X$ and let $\D'$ be a simplicial projective subdivision of $\D_\cX$. Then there exists a vertical blowup $\pi:\cX'\to\cX$ with the following properties:
  \begin{itemize}
  \item[(i)] $\cX'$ is normal and vertically $\Q$-factorial. 
   \item[(ii)] The vertices $(e'_i)_{i\in I'}$ of $\D'$ are in bijection with the irreducible components $(E'_i)_{i\in I'}$ of $\cX'_0$, in such a way that $c_{\cX'}(\emb_\cX(e'_i))$ is the generic point of $E'_i$ for each $i\in I'$. 
   \item[(iii)]
    If $J'\subset I'$, then $E'_{J'}:=\bigcap_{j\in J'}E'_j$ is normal, irreducible, and nonempty iff the corresponding vertices $e'_j$, $j\in J'$ of $\Delta'$ 
    span a face $\sigma'_{J'}$ of $\Delta'$. In this case, $E'_{J'}$ has codimension $|J'|$ and its generic point is the center of $\emb_\cX(s)$ on $\cX'$ for all $s$ in the relative interior of $\sigma'_{J'}$. 
  \item[(iv)] For each $D\in\Div_0(\cX')$ the function $\f_D\circ\emb_\cX$ is affine on the faces of $\D'$.
  \end{itemize}

\end{thm}
This result is in essence contained in the toroidal theory
of~\cite{KKMS}. However, strictly speaking, these authors only deal
with varieties over an algebraically closed field and with toroidal
$S$-varieties, neither of which appears to adequately handle the case
of SNC $S$-varieties when the special fiber is non-reduced. Since
Theorem~\ref{T201} is one of the crucial ingredients in the proof of
Theorem A, we therefore provide a complete proof, mostly
adapting~\cite[pp.76-82]{KKMS}. See also~\cite{Thu} for a similar
construction.

\begin{proof}
~

\smallskip
  \textbf{Step 1}. Given a finite set $L$ and a field $\kappa$, we rely on basic toric geometry (cf.~\cite{KKMS,Ful,Oda}) to show that $Z:=\A^L_\kappa=\Spec\kappa[t_i,\,i\in L]$ and its coordinate hyperplanes $(H_i)_{i\in L}$ satisfy an analogue of (i)-(iv).
Set $T:=(\G_{m,\kappa})^L$  to be the multiplicative split torus of dimension $L$ over $\kappa$. The fan $\Sigma$ of the toric $\kappa$-variety $Z$ consists of the cones $\hat\sigma_J=\sum_{j\in J}\R_+ e_j$, $J\subset L$. For each $s\in\R_+^L$ let 
$$
\val_{Z,s}:\kappa\cro{t_i,\,i\in L}\to\R_+
$$ 
be the monomial valuation with $\val_{Z,s}(t_i)=s_i$ for $i\in L$, so that the center of $\val_{Z,s}$ on $Z$ is the generic point of $H_J:=\bigcap_{j\in J}H_j$ for all $s$ in the relative interior of $\hat\sigma_J$. 
  
  Let $\Sigma'$ be a simplicial fan decomposition of $\Sigma$. The toric $\kappa$-variety $Z'$ attached to $\Sigma'$ comes with a $T$-equivariant proper birational morphism $\rho:Z'\to Z$ satisfying the following properties:
\begin{itemize}
\item[(a)] $Z'$ is normal (because it is toric), and all toric Weil divisors of $Z'$ are $\Q$-Cartier (since $\Sigma'$ is simplicial). 
\item[(b)]  There is a bijection between the set of rays $(R_i)_{i\in L'}$ of $\Sigma'$ and the toric prime divisors $(H'_i)_{i\in L'}$ of $Z'$, in such a way that for each $s\in R_i\setminus\{0\}$ the center of $\val_{Z,s}$ on $Z'$ is the generic point of $H'_i$. 
\item[(c)] For each $J'\subset L'$ the intersection $H'_{J'}:=\bigcap_{j\in J'} H'_j$ is normal, irreducible, and non-empty iff $\hat\sigma'_{J'}=\sum_{j\in J'} R_j$ is a cone of $\Sigma'$. In this case $H'_{J'}$ has codimension $|J'|$, and its generic point is the center of $\val_{Z,s}$ on $Z'$ for all $s$ in the relative interior of $\hat\sigma'_{J'}$. 
\item[(d)] For each toric divisor $G$ of $Z'$, the map $s\mapsto\val_{Z,s}(G)$ is linear on each cone of $\Sigma'$. Here we use the fact that $mG$ is Cartier for some non-zero $m\in\N$ by (a) to set $\val_{Z,s}(G):=\tfrac 1 m\val_{Z,s}(f)$, with $f$ a local equation of $mG$ at the center of $\val_{Z,s}$. 
\end{itemize}
With the notation of (c), assume that $H'_{J'}$ is non-empty and let $\hat\sigma_J$ be the smallest cone of $\Sigma$ containing $\hat\sigma'_{J'}$. We then have $\rho(H'_{J'})=H_J$, and we claim that 
\begin{equation}\label{equ:connected}
\rho_*\cO_{H'_{J'}}=\cO_{H_J}.
\end{equation}
Indeed, denote by $\zeta'_{J'}$ and $\zeta_J$ the generic points of $H'_{J'}$ and $H_J$, respectively. Since $H_J$ is normal,  (\ref{equ:connected}) will follow from the fact that $\kappa(\zeta_J)$ is algebraically closed in $\kappa(\zeta'_{J'})$ (cf.~\cite[III.4.3.12]{EGA}). But $H'_{J'}$ is the closure of a $T$-orbit $(H'_{J'})^0$ in $Z'$, mapping to the $T$-orbit $H_J^0:=(\bigcap_{j\in J} H_j)\setminus(\bigcup_{j\notin J} H_j)$ in $Z$. The stabilizer of $H_J^0$ in $T$ is $(\G_{m,\kappa})^J$, so the $T$-equivariant morphism $(H'_{J'})^0\to H_J^0$ has geometrically integral fibers. In particular $\zeta'_{J'}$ is the generic point of the fiber over $\zeta_J$, and $\kappa(\zeta_J)$ is algebraically closed in $\kappa(\zeta'_{J'})$ by~\cite[IV.4.5.9]{EGA}. 
  
\smallskip
  \textbf{Step 2}. Let $h\in\PA(\D_\cX)_\Z$ be a strictly convex support function for $\D'$. We define $\cX'$ as the blowup of $\cX$ along the fractional ideal sheaf $\fa_h$ given in Definition~\ref{defi:fah} (see \S\ref{sec:mod}). Note that $\cX$ is normal since $\fa_h$ is integrally closed, being defined by valuative conditions. 
  
Let $\xi\in\cX_0$ be a given point and use the notation of Remark~\ref{rmk:otherpoint}. Since $\cX$ and $Z:=\A^L_{\kappa(\xi)}$ are excellent we get a diagram
$$
\xymatrix{\cX & \hcX_\xi\ar[l]_{p}\ar[r]^q & Z}
$$
where $p$ and $q$ are \emph{regular}, \ie flat and with (geometrically) regular fibers (but typically not of finite type, as opposed to a smooth morphism). By Remark~\ref{rmk:otherpoint} we have
\begin{equation}\label{equ:valcomp}
p_*\val_{\hcX_\xi,s}=\val_{\cX,s}\ \text{and $q_*\val_{\hcX_\xi,s}=\val_{Z,s}$}
\end{equation}
for all $s\in\sigma_{I_\xi}$. The subdivision of $\sigma_{I_\xi}$ defined by $\D'$ induces a simplicial fan decomposition $\Sigma'$ of $\R_+^L$, to which the results of Step 1 apply. Since $h$ is a support function of $\D'$, the toric $\kappa(\xi)$-variety $Z'$ attached to $\Sigma'$ coincides in fact with the blowup of $Z$ along the toric fractional ideal sheaf
$$
\fb_h:=\sum\{\cO_Z(H_m),\,m\in\Z^{I_\xi},\,\langle m,\cdot\rangle\le h\text{ on }\sigma_\xi\},
$$
where we have set $H_m:=\sum_{i\in I_\xi} m_i H_i$. Comparing with (\ref{equ:faxj}), we see that 
$$
p^{-1}\fa_h\cdot\hcO_{\cX,\xi}=q^{-1}\fb_h\cdot\hcO_{\cX,\xi}.
$$ 
Since blowups commute with flat base change (cf.~\cite[8.1.12]{Liu}), $\tcX'_\xi:=\cX'\times_{\cX}\Spec\hcX_\xi$ sits in a commutative diagram
\begin{equation}\label{equ:diagram1}
\xymatrix{\cX'\ar[d]_{\pi} & \tcX'_\xi\ar[l]_{p'}\ar[r]^{q'}\ar[d] & Z'\ar[d]_{\rho}\\ 
\cX & \hcX_\xi\ar[l]_p\ar[r]^q & Z} 
\end{equation}
where the two squares are Cartesian. The morphisms $p'$ and $q'$ are also regular, since the latter property is preserved under finite type base change (cf.~\cite[IV.6.8.3]{EGA}). 

Let $(e'_i)_{i\in I'_\xi}$ be the set of vertices of $\D'$ contained
in $\sigma_{I_\xi}$, so that each ray $\R_+e'_i$ belongs to the fan
$\Sigma'$. If we let $H'_i$ be the corresponding toric prime divisor
of $Z'$ and pick $J'\subset I'_\xi$ then $H'_{J'}=\bigcap_{j\in
  J'}H_j$ is normal, irreducible, and non-empty iff the $e'_j$, $j\in
J'$ span a face $\sigma'_{J'}$ of $\D'$, by property (c). Since $q'$
is regular, if follows that $q'^{-1}(H'_{J'})$ is normal and is either
empty or of codimension $|J'|$. It is furthermore \emph{irreducible}
(and thus nonempty), by (\ref{equ:connected}) and Lemma~\ref{lem:irred} below. In particular, $(q'^{-1}(H'_i))_{i\in I'_\xi}$ is exactly the set of irreducible components of the special fiber of $\tcX'_\xi$. Using that the special fiber of $\tcX'_\xi$ is precisely the union of the zero loci of the $z_i$'s for $i\in I_\xi$, it is now easy to obtain the analogue of (i)-(iv) of Theorem~\ref{T201} with $\tcX'_\xi$, $\sigma_{I_\xi}$ and $\val_{\hcX_\xi}$ in place of $\cX'$, $\D_\cX$ and $\val_\cX$. 

In particular, $\tcX'_\xi$ is normal for each $\xi\in\cX_0$, which shows that $\cX'$ is normal, hence a model of $X$. 

On the other hand, for each irreducible component $E'$ of $\cX'_0$
such that $\pi(E')$ contains $\xi$, we claim that the divisor $p'^{-1}(E')$ is
irreducible. Indeed, each irreducible component of the divisor
$p'^{-1}(E')$ is of the form $q'^{-1}(H'_i)$ for some $i\in I'_\xi$,
since it is contained in the special fiber by construction and of
codimension one by flatness. 
If we denote by $\xi'$ and $\eta'_i$ the generic points of $E'$ and $p'^{-1}(H_i')$, respectively, then we have on the one hand $p'(\eta'_i)=\xi'$ since $p'$ is flat. On the other hand, $\eta_i'$ is the center of $\val_{\hcX_\xi,e'_i}$ on $\tcX'_\xi$, hence $p'(\eta_i')=c_{\cX'}(\val_{\cX,e'_i})$ thanks to (\ref{equ:valcomp}). For dimension reason it follows that $\emb_{\cX}(e'_i)=x_{E'}\in\Xan$, and the injectivity of $\emb_\cX$ shows that $i$ is uniquely determined by $E'$, which implies as desired that $p'^{-1}(E')$ is irreducible. 

We may thus write the irreducible components of $\cX'_0$ that are mapped to $\overline\{\xi\}$ as $(E'_i)_{i\in I'_\xi}$, with the property that 
$$
p'^{-1}(E'_i)=q'^{-1}(H'_i).
$$
By flat descent it follows that $E'_i$ is normal at each point of the fiber of $\xi$. It is also $\Q$-Cartier, since a Weil divisor is Cartier at a point iff its restriction to the formal neighborhood of that point is Cartier  (see e.g.~\cite[Proposition~1]{samuel})
It is now easy to conclude the proof of (i)-(iv), using the analogous properties for $\hcX'_\xi$ together with (\ref{equ:valcomp}). 
\end{proof}

\begin{lem}\label{lem:irred} Assume that 
$$ 
\xymatrix{U'\ar[d]_{f}\ar[r] & V'\ar[d]^g\\ U\ar[r] & V} 
$$
is a Cartesian square of Noetherian schemes such that the vertical arrows are proper and surjective and the horizontal morphisms are regular. If $U$, $V$ are $V'$ are irreducible, $V$ and $V'$ are normal and $g_*\cO_{V'}=\cO_V$ then $U'$ is normal and irreducible.  
\end{lem}

\begin{proof} Note first that $U$ and $U'$ are normal
  by~\cite[IV.6.5.4]{EGA}. Since direct images commute with flat base
  change we have $f_*\cO_{U'}=\cO_U$, which implies that $f$ has
  connected fibers as a consequence of the theorem on formal functions
  (cf.~\cite[III.4.3.2]{EGA}). Since $U$ is connected and nonempty,
  and $f$ is closed, surjective and has connected fibers, it follows
  that $U'$ is connected and nonempty, hence irreducible since it is normal. 
\end{proof} 

\begin{cor}\label{cor:ratdiv}
For each SNC model $\cX$, the set of rational points of $\D_\cX$
coincides with the set $\emb_\cX^{-1}(\Xdiv)\cap\D_\cX$;
hence $\emb_{\cX}(\D_\cX)\cap\Xdiv$ is dense in $\emb_{\cX}(\D_\cX)$.
\end{cor}
\begin{proof} If $s\in\D_\cX$ is a rational point,
  Theorem~\ref{T201} yields a vertical blowup $\cX'$ such that
  $\emb_{\cX'}(s)=x_{E'}$ for some irreducible component $E'$ of
  $\cX'_0$. Conversely, if $\emb_\cX(s)$ is a divisorial point, then the
  corresponding valuation takes rational values on the local equations
  of the irreducible components of $\cX_0$, which shows that $s$ is a rational point of $\D_\cX$.
\end{proof}

\section{Metrics on line bundles and closed $(1,1)$-forms}\label{S202}

\subsection{Metrics}\label{S207}
We refer to~\cite{CL2} for a general discussion of metrized
line bundles in a non-Archimedean context.
Suffice it to say that a \emph{continuous} metric $\|\cdot\|$ on a line bundle $L$ on $X$
is a way to produce a continuous function
$\|s\|$ on (the Berkovich space) $\Xan$ from any local section $s$ of $L$. 
Given a continuous metric $\|\cdot\|$, any other continuous metric on $L$
is of the form $\|\cdot\|e^{-\f}$, with $\f\in C^0(\Xan)$. 
If we in this expression allow an arbitrary function 
$\f:\Xan\to[-\infty,+\infty[$, then we obtain a 
\emph{singular metric} on $L$.

Let $\cX$ be a model and $\cL$ a line bundle on $\cX$ such that 
$\cL|_X=L$. 
To this data one can associate 
a unique metric $\|\cdot\|_\cL$ on $L$ with the following property:
if $s$ is a nonvanishing local section of $\cL$ on an 
open set $\cU\subset\cX$, then $\|s\|_\cL\equiv1$ 
on $U:=\cU\cap X$. This makes sense since such a section 
$s$ is uniquely defined up to multiplication by an element of 
$\Gamma(\cU,\cO_\cX^*)$ and such elements have norm 1.

More generally, any $\cL\in\Pic(\cX)_\Q$ such that 
$\cL|_X=L$ in $\Pic(X)_\Q$ induces  a metric $\|\cdot\|_\cL$ on $L$ by setting $\|s\|_\cL=\|s^{\otimes m}\|_{m\cL}^{1/m}$ for any non-zero 
$m\in\N$ such that $m\cL$ is an actual line bundle. 
By definition, a \emph{model metric}\footnote{See Table~\ref{table:modmet} on page~\pageref{table:modmet} for alternative terminology regarding metrics used in the literature.}
on $L$ is a metric of the form $\|\cdot\|_\cL$ with 
$\cL\in\Pic(\cX)_\Q$ for some model $\cX$ such that $\cL|_X=L$. 
Model metrics are clearly continuous. 
If $\|\cdot\|$ is a model metric, then $\|\cdot\|e^{-\f}$ is 
a model metric iff $\f$ is a model function.

If we denote by $\widehat{\Pic}(\Xa)$ the group of isomorphism classes of line bundles on $X$ endowed with a model metric, then it is easy to check that there is a natural isomorphism
\begin{equation}\label{equ:pichat}
  \varinjlim_{\cX\in\cM_X}\Pic(\cX)_\Q\simeq\widehat{\Pic}(\Xa)_\Q
\end{equation} 
and that the natural sequence
\begin{equation}\label{equ:exactpic}
  0
  \to\Q\cX_0
  \to\cD(\Xa)
  \to\widehat{\Pic}(\Xa)_\Q
  \to\Pic(X)_\Q\to 0
\end{equation}
is exact.
%
%

\subsection{Closed $(1,1)$-forms}
Recall that $N^1(\cX/S)$ is the set of $\R$-line bundles on a model $\cX$ modulo
those that are numerically trivial on the special fiber.
\begin{defi} 
  The space of \emph{closed $(1,1)$-forms} on $\Xan$ is defined as the direct limit 
  $$
  \cZ^{1,1}(\Xa):=\varinjlim_{\cX\in\cM_X} N^1(\cX/S)
  $$
\end{defi}
As with model functions, we say that $\theta\in\cZ^{1,1}(\Xa)$ is \emph{determined} on a given model $\cX$ if it is the image of an element $\theta_\cX\in N^1(\cX/S)$. By definition, two classes $\a\in N^1(\cX/S)$ and $\a'\in N^1(\cX'/S)$
define the same element in $\cZ^{1,1}(\Xa)$ iff they pull-back to the same class on a model dominating both $\cX$ and $\cX'$. 

\begin{rmk}\label{R203} The previous definition is directly inspired from~\cite{BGS}, where closed forms and currents are defined in the non-Archimedean setting. We choose however to work modulo numerical equivalence instead of rational equivalence. One justification for this choice is Corollary~\ref{cor:flat} below. The fact that each space $N^1(\cX/S)$ is endowed with a natural topology as a finite dimensional vector space is another reason. 
\end{rmk}

\smallskip

The isomorphism (\ref{equ:pichat}) shows that there is a natural map
$$
\widehat{\Pic}(\Xa)\to\cZ^{1,1}(\Xa).
$$
The image of $(L,\|\cdot\|)\in\widehat{\Pic}(\Xa)$ under this map is denoted by $c_1(L,\|\cdot\|)\in\cZ^{1,1}(\Xa)$ and called the \emph{curvature form} of the metrized line bundle $(L,\|\cdot\|)$. 

By definition, any model function $\f \in \cD(\Xa)$ is determined on some
model $\cX$ by some divisor $D\in\Div_0(\cX)_\R$. 
We set $dd^c \f$ to be the form determined by the numerical class of $D$ in 
$N^1(\cX/S)$. In this way, we get a natural linear map 
$$
dd^c:\cD(\Xa)\to\cZ^{1,1}(\Xa).
$$
On the other hand, the restriction maps $N^1(\cX/S) \to N^1(X):=N^1(\cX_K/K)$ induce a linear map
$$
\{\cdot\} : \cZ^{1,1}(\Xa) = \varinjlim_{\cM_X} N^1(\cX/S) \to N^1(X).
$$
We call $\{\theta\}\in N^1(X)$ the 
\emph{de Rham class} of the closed $(1,1)$-form $\theta$. 
Note that
$$
\{c_1(L,\|\cdot\|)\}=c_1(L)
$$
for each metrized line bundle $(L,\|\cdot\|)\in\widehat{\Pic}(\Xa)$. 
The next result is an analogue of the $dd^c$-lemma in the complex setting.

\begin{thm}\label{thm:ddc} Let $X$ be a smooth connected projective $K$-analytic variety. Then the natural sequence
$$
0\to\R\to \cD(\Xa)_\R \mathop{\longrightarrow}\limits^{dd^c} \cZ^{1,1}(\Xa)\to N^1(X)\to 0
$$
is exact.
\end{thm}

The following equivalent reformulation is also familiar in the complex setting. 
\begin{cor}\label{cor:flat} Let $L$ be a line bundle on $X$. Then $c_1(L)\in N^1(X)$ vanishes iff $L$ admits a model metric with zero curvature. Such a metric is then unique up to a constant. 
\end{cor}

Theorem~\ref{thm:ddc} is more difficult than its rather straightforward analogue (\ref{equ:exactpic}), whose proof is valid without any assumption of the residue field. Here the existence of regular models is used. Exactness at $\cD(\Xa)$ follows from a rather standard Hodge-index type argument 
(compare~\cite[Theorem 2.1]{YZ09} and see~\cite{YZ13a,YZ13b} for far-reaching generalizations), 
whereas exactness at $\cZ^{1,1}(\Xa)$ is essentially a reformulation of a result by K\"unneman~\cite[Lemma 8.1]{Kun}; see also \cite[Theorem 8.9]{Gub03}. We provide some details for the convenience of the reader. 

\begin{proof}[Proof of Theorem~\ref{thm:ddc}.] 
  We are going to prove the stronger assertion that
  $$
  0\to\R\cX_0\to\Div_0(\cX)_\R\to N^1(\cX/S)\to N^1(\cX_K/K)\to 0
  $$
  is exact for every \emph{regular} model $\cX$ of $X$. 
  We first prove the exactness at  $\Div_0(\cX)_\R$. 
  Let $\cX_0=\sum_{i\in I}b_iE_i$ be the irreducible decomposition of the special fiber. 
  We claim that $\cX_0$ is connected. Since $X\simeq\cX_K^\mathrm{an}$
  is connected by assumption, the (easy direction of the) GAGA principle implies that $\cX_K$ is also connected. If $\cX_0$ were disconnected then $H^0(\cX,\cO_\cX)$ 
  would split as a product by the Grothendieck-Zariski theorem on formal
  functions~\cite[Theorem 11.1]{Har}, 
  which would contradict the connectedness of $\cX_K$. 
  Since $\cX$ is regular, each $E_i$ is Cartier. Pick any ample divisor $\cA$ on $\cX$ 
  and define a quadratic form $q$ on $\R^I$ by setting 
  $$
  q(a):=-\left(\sum_i a_i E_i\right)^2\cdot\cA^{\dim X-1}.
  $$
  We have $q_{ij}\le 0$ for $i\neq j$, and the matrix $(q_{ij})$ is indecomposable 
  since $\cX_0$ is connected. By~\cite[Lemma 2.10]{BPV} it follows that $b$ 
  spans the kernel of $q$. Now let $D=\sum_i a_i E_i$ be a vertical $\R$-divisor 
  whose numerical class on $\cX_0$ is $0$. It follows that $a$ 
  belongs to the kernel of $q$, hence is proportional to $b$, which precisely 
  means that $D\in\R\cX_0$ as desired. 

  \smallskip
  Let us now turn to exactness at $N^1(\cX/S)$, which amounts to the 
  following assertion: every numerically trivial $L\in\Pic(\cX_K)$ 
  admits a numerically trivial extension $\cL\in\Pic(\cX)_\Q$. 

  Arguing as in~\cite[Lemma 8.1]{Kun}, assume first that $X$ is one-dimensional. 
  Let $\cL\in\Pic(\cX)_\Q$ be an arbitrary extension of $L$ to the 
  regular model $\cX$. In the notation above we 
  have $\sum_i b_i(\cL\cdot E_i)=0$ since $\cL$ is numerically trivial 
  on the generic fiber $X$. Since $b=(b_i)_{i\in I}$ spans the kernel of the 
  intersection matrix $(E_i\cdot E_j)$, we may thus find $a\in\Q^I$ 
  such that $\sum_i a_i E_i\cdot E_j=\cL\cdot E_j$ for $j\in I$, 
  which shows that $\cL-\sum_i a_i E_i$ is a numerically trivial 
  extension of $L$ to $\cX$. 

  We now consider the general case, again following~\cite[Lemma
  8.1]{Kun}. Given any $S$-scheme $Y$ we write $\cX_Y:=\cX\times_S
  Y$. Since $L$ is numerically trivial on $\cX_K$, some multiple $mL$ belongs to $\Pic^0(\cX_{\bar K})$ by~\cite{Mat}, and hence there exists a finite extension $K'/K$ such that the pull-back of $mL$ to $\cX_{K'}$
  is algebraically equivalent to $0$. This implies that there exists a smooth projective $K'$-curve $T$, a numerically trivial $\Q$-line bundle $M$ on $T$ and a (Cartier) divisor $D$ on $\cX_T$ such that 
$$
L=q_*\left(p^*M\cdot D\right)
$$ 
in $\Pic(\cX_K)_\Q$, where $p:\cX_T\to T$ and $q:\cX_T\to\cX_K$ are the natural morphisms. Now let $\cT$ be a regular model of $T$ over the integral closure $S'$ of $S$ in $K'$, and consider the commutative diagram
\begin{equation}\label{equ:diagram2}
\xymatrix{T\ar[d] & \cX_T \ar[l]_{p}\ar[r]^{q}\ar[d] & \cX_K\ar[d]\\ 
\cT & \cX_\cT\ar[l]_p\ar[r]^q & \cX} 
\end{equation}
where we also use for simplicity $p$ and $q$ to denote the natural projections $\cX_\cT\to\cT$ and $\cX_\cT\to\cX$. By the one-dimensional case, $M$ extends to a numerically trivial $\Q$-line bundle $\cM\in\Pic(\cT)_\Q$. Let also $\cD$ be the closure of $D$ in $\cX_\cT$, which is \emph{a priori} merely a Weil divisor. We may then set
$$
\cL:=q_*\left(p^*\cM\cdot\cD\right).
$$
Note that $\cL$ belongs to $\CH^1(\cX)_\Q=\Pic(\cX)_\Q$ since $\cX$ is
regular. It is clear that $\cL$ extends $L$, and it remains to show
that $\deg(\cL\cdot C)=0$ for each vertical projective curve $C$ on
$\cX$. Since $\cX$ is regular, $\CH(\cX)_\Q$ is a graded commutative
algebra with respect to cup-product, by~\cite[\S8.3]{GS87}. As
in~\cite[\S2.3]{GS92} one can then define the cap-product $\a\cdot_q\b$
of $\a\in\CH(\cX)_\Q$ and $\b\in\CH(\cX_\cT)_\Q$, which turns
$\CH(\cX_\cT)_\Q$ into a graded $\CH(\cX)_\Q$-module such that both
$q_*:\CH(\cX_\cT)_\Q\to\CH(\cX)_\Q$ and multiplication with
$\b'\in\Pic(\cX_\cT)_\Q$ are maps of $\CH(\cX)_\Q$-modules. Applying
this with $\b'=p^*\cM\in\Pic(\cX_\cT)_\Q$, which is numerically
trivial on the special fiber of $\cX_\cT$,  we get
  \begin{equation*}
    \deg\left(C\cdot\cL\right)
    =\deg\left(C\cdot_q\left(\b'\cdot\cD\right)\right)
    =\deg\left(\b'\cdot\left(C\cdot_q\cD\right)\right)
    =0.
  \end{equation*}

Finally, the surjectivity of $N^1(\cX/S)\to N^1(\cX_K/K)$ is clear since $N^1(\cX_K/K)$ is spanned by classes of Cartier divisors on $X$, the closures in $\cX$ of which are also Cartier since $\cX$ is regular.
\end{proof}

%
%
\section{Positivity of forms and metrics}

\subsection{Positive closed $(1,1)$-forms and metrics}
The following definition extends the ones in~\cite{Zha,Gub98,CL1}. 
\begin{defi}
  A closed $(1,1)$-form $\theta$ is said to be:
  \begin{itemize}
  \item[(i)] 
    \emph{semipositive} if $\theta_\cX\in N^1(\cX/S)$ is nef for some
    (or, equivalently, any) determination $\cX$ of $\theta$;
  \item[(ii)] 
    \emph{$\cX$-positive} if $\cX\in\cM_X$ is a determination of $\theta$ 
    and $\theta_\cX\in N^1(\cX/S)$ is ample. 
  \end{itemize}
  A model metric $\|\cdot\|$ on a line bundle $L$ is said to be 
  semipositive if the curvature form $c_1(L,\|\cdot\|)$ is semipositive.
\end{defi}
The equivalence in (i) follows from the following standard fact: if $\a\in N^1(\cX/S)$ is a numerical class and $\pi:\cX'\to\cX$ is a vertical blowup then $\pi^*\a$ is nef iff $\a$ is nef. On the other hand, the analogous result is obviously wrong for ample classes, so that it is indeed necessary to specify the model in (ii). If $\om$ is $\cX$-positive and $\theta\in\cZ^{1,1}(\Xa)$ is determined on $\cX$, then $\om+\e\theta$ is also $\cX$-positive for all $0<\e\ll 1$. 

Note that if a closed $(1,1)$-form $\theta$ is semipositive, then its de Rham class 
$\{\theta\}\in N^1(X)$
is automatically nef. See Remark~\ref{R301} for a more precise statement.

The set of all semipositive closed $(1,1)$-forms is a convex cone 
$\cZ^{1,1}_+(\Xa)$ of $\cZ^{1,1}(\Xa)$ that can be equivalently defined as
$$
\cZ^{1,1}_+(\Xa):=\varinjlim_\cX\Nef (\cX/S).
$$ 
\begin{prop} \label{prop:positiverep}
  Let $\theta$ be a closed $(1,1)$-form whose de Rham class 
  $\{\theta\}\in N^1(X)$ is ample. 
  For every sufficiently high model $\cX$, 
  we may then find a model function $\f$ such that $\theta+dd^c\f$ is $\cX$-positive. If $\theta$ is furthermore semipositive, then we may 
  also arrange that $-\e\le\f\le 0$ for any given $\e>0$. 
\end{prop}
\begin{proof} 
  Let $\cX'$ be a determination of $\theta$ and let $\cL'\in\Pic(\cX')_\R$ be a 
  representative of $\theta$. The assumption implies that the $\R$-line bundle 
  $L:=\cL'|_{\cX_K}$ is ample. By Corollary~\ref{cor:ample} we may thus assume that 
  $\cX'$ has been chosen so that $L$ admits an ample extension 
  $\cL\in\Pic(\cX)_\R$ for each model $\cX$ dominating $\cX'$.  
  If $\pi:\cX\to\cX'$ denotes the corresponding vertical blowup 
  then $\cL-\pi^*\cL'=D$ for some $D\in\Div_0(\cX)_\R$, 
  and $\f=\f_D$ is a model function such that $\theta+dd^c\f$ 
  is $\cX$-positive. 

  Now suppose $\theta$ is semipositive and pick $\cX$, $\f$  
  as above.  Upon replacing $\f$ by $\f-\sup_X\f$ we may assume that $\f\le 0$.
  Then the closed $(1,1)$-form 
  $$
  \theta+dd^c(\e\f)=\e(\theta+dd^c\f)+(1-\e)\theta
  $$
  is also $\cX$-positive for each $0<\e<1$, completing the proof since
  $\f$ is bounded. 
\end{proof}
Since the nef cone of $N^1(X)$ is the closure of the ample cone, we get as a consequence:

\begin{cor}\label{cor:closure} The closure of the image of $\cZ_+^{1,1}(\Xa)$ in $N^1(X)$ coincides with the nef cone of $N^1(X)$.
\end{cor}
\begin{rmk}\label{R301}
  In the complex case, it is not always possible to find a smooth
  semipositive form in a nef class, so the image of 
  $\cZ^{1,1}_+(\Xa)$ in $N^1(X)$ is \emph{strictly}  
  contained in $\Nef(X)$ in general, see~\cite[Example~1.7]{DPS}. 
  In the non-Archimedean setting, the situation is unclear.
\end{rmk}




\subsection{$\theta$-psh model functions}\label{S203}
By analogy with the complex case, we introduce:
\begin{defi}\label{defi:psh-model} 
  Let $\theta\in\cZ^{1,1}(\Xa)$ be a closed $(1,1)$-form. A model function $\f\in\cD(X)$ is said to be $\theta$-plurisubharmonic ($\theta$-psh for short) if the closed $(1,1)$-form $\theta+dd^c\f$ is semipositive. 
\end{defi}
Note that constant functions are $\theta$-psh model functions iff $\theta$ is semipositive.
Moreover, the existence of a $\theta$-psh model function implies that the de Rham class
$\{\theta\}\in N^1(X)$ is nef.
Also note that if $\psi\in\cD(\Xa)$, then $\f$ is a $\theta$-psh model function iff
$\f-\psi$ is $(\theta+dd^c\psi)$-psh. 

We will need two technical results relating $\theta$-psh model functions to fractional ideal sheaves. 

\begin{lem}\label{lem:gen} 
  Let $\cL\in\Pic(\cX)$ and let $\|\cdot\|$ be the corresponding model metric on 
  $L:=\cL|_{\cX_K}$. If $\fa$ is a vertical fractional ideal sheaf 
  on $\cX$ such that $\cL\otimes\fa$ is generated by its 
  global sections, then $\log|\fa|$ is a $c_1(L,\|\cdot\|)$-psh model function. 
\end{lem}
Since $S$ is affine, the direct image on $S$ of a coherent sheaf $\cF$ on $\cX$ is always generated by its global sections; hence $\cF$ is globally generated in the absolute sense iff it is globally generated in the relative sense. 

\begin{proof} Let $\pi:\cX'\to\cX$ be the normalization of the blowup of $\cX$ along $\fa$ and let $D\in\Div_0(\cX')$ be the vertical Cartier divisor such that $\fa\cdot\cO_{\cX'}=\cO_{\cX'}(D)$. The assumption implies that $\pi^*\cL\otimes\cO_{\cX'}(D)$ is also generated by its global sections, so that $\pi^*\cL+D$ is nef. The result follows since the model function $\log|\fa|$ is determined on $\cX'$ by $D$. 
\end{proof}

\begin{lem}\label{lem:approx} Let $\theta$ be a closed $(1,1)$-form and let $\cX$ be a determination of $\theta$. Then each $\theta$-psh model function $\f\in\cD(\Xa)$ is a uniform limit on $\Xan$ of functions of the form $\tfrac 1 m\log|\fa|$ with $m\in\N^*$ and $\fa$ a vertical fractional ideal sheaf on $\cX$.
\end{lem}
\begin{proof} 
  Let $\pi:\cX'\to\cX$ be a (normalized) vertical blowup such that $\f=\f_D$ for some 
  $D\in\Div_0(\cX')_\Q$. Since $\theta$ is determined by 
  $\theta_\cX\in N^1(\cX/S)$, the assumption that $\f$ is $\theta$-psh 
  implies that $D$ is $\pi$-nef. 
  By Lemma~\ref{lem:ample} and Kleiman's criterion~\cite{Kle}, we may find a 
  vertical $\pi$-ample $\Q$-divisor $A\in\Div_0(\cX')_\Q$ arbitrarily close to $D$. 
  It is then clear that $\f_A$ is uniformly close to $\f=\f_D$ on $\Xan$
  (see the proof of Corollary~\ref{C208}).
  Since $A$ is $\pi$-ample we may find $m\gg 1$ such that $\cO_{\cX'}(mA)$ is 
  $\pi$-globally generated. If we set $\fa:=\pi_*\cO_{\cX'}(mA)$ we then have 
  $\f_A=\tfrac 1 m\log|\fa|$, which concludes the proof. 
\end{proof} 

We are now in a position to establish the first properties of $\theta$-psh model functions. 
\begin{prop}\label{prop:maxmodel} 
  Let $\theta\in\cZ^{1,1}(X)$ be a closed $(1,1)$-form. 
  Then the set of $\theta$-psh model functions 
  $\f\in\cD(\Xa)$ is ($\Q$-)convex and stable under max. 
\end{prop}
\begin{proof}
  Convexity is clear from the definition. To prove stability under maxima,
  let $\f_1,\f_2\in\cD(\Xa)$ be $\theta$-psh, pick 
  a common determination $\cX$ of $\theta$ and the $\f_i$'s
  and let $D_i\in\Div_0(\cX)_\Q$ be a representative of $\f_i$ for $i=1,2$. 
  
  Since the ample cone of $N^1(\cX/S)$ is open, we may find ample line bundles
  $\cA_1,\dots,\cA_r\in\Pic(\cX)$ whose numerical classes 
  $\a_1,\dots,\a_r$ form a basis of $N^1(\cX/S)$. 
  We may thus pick $t_1,\dots,t_r\in\R$ such that $\cL:=\sum_j t_j\cA_j$ 
  is a representative of $\theta$ in $\Pic(\cX)_\R$. Let $\e_1,\dots,\e_r>0$ be 
  (small) positive numbers such that $t_j+\e_j\in\Q$ for each $j$ and set 
  $\cL_\e:=\sum_j(t_j+\e_j)\cA_j$. Since $\f_i$ is $\theta$-psh it follows that 
  $\cL_\e+D_i$ is an ample $\Q$-divisor on $\cX$ for $i=1,2$. 
  We may thus find a positive integer $m$ such that 
  $m\cL_\e\in\Pic(\cX)$, $mD_i\in\Div_0(\cX)$ and both sheaves
  $\cO_\cX\left(m\left(\cL_\e+D_i\right)\right)$, $i=1,2$ 
  are generated by their global sections on $\cX$. 
  If we introduce the vertical fractional ideal sheaf 
  $$
  \fa_m:=\cO_\cX(mD_1)+\cO_\cX(mD_2)
  $$ 
  then it follows that $\cO_\cX(m\cL_\e)\otimes\fa_m$ is also generated 
  by its global sections. By Lemma~\ref{lem:gen}, 
  $\log|\fa_m|=m\max\left\{\f_1,\f_2\right\}$ is thus psh with respect 
  to $m(\theta+\sum_j\e_j\a_j)$, that is,
  $$
  \theta+\sum_j\e_j\a_j+dd^c\max\{\f_1,\f_2\}\ge 0.
  $$
  Letting $\e_j\to0$, we conclude as desired that $\theta+dd^c\max\{\f_1,\f_2\}\ge 0$. 
\end{proof}

\begin{prop}\label{prop:convex} 
   Let $\theta\in\cZ^{1,1}(X)$ be a closed $(1,1)$-form and let $\cX$ be a SNC model on 
  which $\theta$ is determined. Then each $\theta$-psh model function 
  $\f\in\cD(\Xa)$ satisfies:
  \begin{itemize}
  \item[(i)] 
    $\f\circ\emb_\cX$ is piecewise affine and convex on each face of $\D_\cX$;
  \item[(ii)] 
    $\f\le\f\circ \retr_\cX$ with equality if $\f$ is determined on $\cX$.
  \end{itemize}
\end{prop}
\begin{proof} 
  This follows directly from Lemma~\ref{L207}~(ii), Proposition~\ref{prop:PA} and Lemma~\ref{lem:approx} 
\end{proof} 

Finally we show that $\theta$-psh model functions are plentiful as
soon as $\{\theta\}$ is ample. 
\begin{prop}\label{prop:generate}
   Let $\theta\in\cZ^{1,1}(X)$ be a closed $(1,1)$-form
   whose de Rham class 
  $\{\theta\}\in N^1(X)$ is ample. 
  Then $\cD(\Xa)$ is spanned by $\theta$-psh model functions.
\end{prop}
\begin{proof} 
  Let $\f\in\cD(\Xa)$. By Proposition~\ref{prop:positiverep} we may 
  find a model $\cX$ and a model function $\psi$ such that 
  $\theta$, $\f$ and $\psi$ are all determined on $\cX$ and such that
  $\theta+dd^c\psi$ is $\cX$-positive. 
  Since the closed $(1,1)$-form $dd^c\f$ is determined on $\cX$ 
  we may thus find a rational number $0<\e\ll 1$ such that 
  $\theta+dd^c(\psi+\e\f)\ge 0$. It follows that $\e\f=(\psi+\e\f)-\psi$ 
  is a difference of $\theta$-psh model functions, and the result follows.
\end{proof}

\subsection{Closedness of $\theta$-psh model functions}\label{sec:closed}
The next result will be used to show that the definition of $\theta$-psh functions in Section~\ref{S208} below extends the one for model functions. 

\begin{thm}\label{thm:closed} Let $\theta\in\cZ^{1,1}(X)$ be a closed $(1,1)$-form. Then the set of $\theta$-psh model functions is closed in $\cD(X)$ with respect to the topology of pointwise convergence on $\Xdiv$. 
\end{thm}

This theorem in particular implies that S.-W. Zhang's definition of continuous semipositive metrics as uniform limits of semipositive model metrics (cf.~\cite[3.1]{Zha}) is consistent when applied to model metrics. Another argument for this, valid in arbitrary residue characteristic, has been communicated to the authors by A. Thuillier. This argument uses a theorem of Tate to reduce to the case of curves. 

We start the proof with the following special case. 
\begin{lem}\label{lem:closed} Let $\cX$ be an SNC model and pick $\cL\in\Pic(\cX)$ such that $L:=\cL|_{\cX_K}$ is ample. Assume that the model metric $\|\cdot\|_\cL$ is a pointwise limit over $\Xdiv$ of semipositive model metrics on $L$. Then $\|\cdot\|_\cL$ itself is semipositive, \ie $\cL$ is nef.
\end{lem} 
\begin{proof} 
\smallskip
  \textbf{Step 1}. For each $m\ge 0$ let $\fa_m\subset\cO_\cX$ be the base-ideal of 
  $\cO_\cX(m\cL)$, \ie the image of the evaluation map 
 $$
 H^0(\cO_\cX(m\cL))\otimes\cO_\cX(-m\cL)\to\cO_\cX.
 $$
 We are going to show that $\tfrac 1 m\log|\fa_m|$ converges pointwise to $0$ on $\Xdiv$. Note that $\fa_m$ is vertical for $m\gg 1$ since $\cL$ is ample on the generic fiber of $\cX$. The sequence $\fa_\bullet=(\fa_m)_{m\ge 0}$ is a graded sequence of ideals, \ie we have $\fa_m\cdot\fa_l\subset\fa_{m+l}$ for all $m,l$. It follows that $(\log|\fa_m|)_m$ is a super-additive sequence, which implies that 
\begin{equation}\label{equ:convsup}
\lim_{m\to\infty}\frac 1 m\log|\fa_m|=\sup_m\frac 1 m\log|\fa_m|\le 0
\end{equation}
pointwise on $\Xan$. Pick a rational number $\e>0$ and $x\in\Xdiv$. Let $\theta$ be the curvature form of $\|\cdot\|_\cL$. Since $0$ is by assumption a pointwise limit of $\theta$-psh model functions, there exists a vertical blowup $\pi:\cX'\to\cX$ and $D\in\Div_0(\cX')_\Q$ such that $\f_D$ is $\theta$-psh, $\f_D(x)\ge-\e$ and $\f_D(x_{E_i})\le\e$ for each irreducible component $E_i$ of our given model $\cX$.  By Proposition~\ref{prop:convex} the latter condition yields $\f_D\le\e$ on $\Xan$, so that $D':=D+\e\cX'_0\in\Div_0(\cX')$ satisfies $D'\le 0$ and $\f_{D'}(x)\ge -2\e$. On the other hand, we may assume that $\cX'$ has been chosen high enough to apply Proposition~\ref{prop:positiverep} and get $D''\in\Div_0(\cX')_\Q$ with $D''\le 0$, $\f_{D''}\ge -\e$ on $\Xan$ and $\pi^*\cL+D'+D ''$ ample. Since $D'+D''\le 0$ we then have 
$$
\cO_{\cX'}(m\left(\pi^*\cL+D'+D''\right))\subset\cO_{\cX'}(m\pi^*\cL)
$$
for all $m\in\N$. Now the left-hand side is globally generated for some $m$. Since $\pi_*\cO_{\cX'}=\cO_\cX$, all sections of $m\pi^*\cL$ are pull-backs of sections of $m\cL$, and the projection formula therefore yields 
$$
\cO_{\cX'}(m (D'+D''))\subset\cO_{\cX'}\cdot \fa_m,
$$ 
hence
$$
-3\e\le\f_{D'+D''}(x)\le \frac1m \log|\fa_m| (x). 
$$
We have thus shown that $\sup_m\tfrac 1 m\log|\fa_m|\ge 0$ at each $x\in\Xdiv$, which implies as desired that $\tfrac 1 m\log|\fa_m|$ converges to $0$ pointwise on $\Xdiv$ thanks to (\ref{equ:convsup}). 

\smallskip
  \textbf{Step 2}. Let us now show that $\cL$ is nef. For each $c>0$ let $\cJ(\fa_\bullet^c)\subset\cO_\cX$ be the multiplier ideal attached to the graded sequence $\fa_\bullet$ (cf.~Appendix B). We have the elementary inclusion $\fa_m\subset\cJ(\fa_\bullet^m)$ for all $m\in\N$, whereas the subadditivity property (cf.~Theorem~\ref{thm:subadd}) implies $\cJ(\fa_\bullet^{ml})\subset\cJ(\fa_\bullet^m)^l$ for all $l,m\in\N$. We infer that $\fa_{ml}\subset\cJ(\fa_\bullet^m)^l$ for any $m,l$ and hence
$$
\sup_l\tfrac 1{l}\log|\fa_{ml}|\le\log|\cJ(\fa_\bullet^m)|\le 0. 
$$
By Step 1 we conclude that $\log|\cJ(\fa_\bullet^m)|=0$, \ie $\cJ(\fa_\bullet^m)=\cO_\cX$ since multiplier ideals are integrally closed by definition. The uniform global generation property of multiplier ideals (Theorem~\ref{thm:uniform}) now yields an (ample) line bundle $\cA\in\Pic(\cX)$ independent of $m$ such that $m\cL+\cA$ is globally generated (and hence nef) for all $m\in\N$. This immediately shows that $\cL$ is nef, since the latter is a closed condition. 
\end{proof}

\begin{proof}[Proof of Theorem~\ref{thm:closed}]
Suppose that $\f\in\cD(X)$ is a pointwise limit of $\theta$-psh model functions. Our goal is to show that $\f$ is $\theta$-psh. Upon replacing $\theta$ with $\theta+dd^c\f$ we may assume that $\f=0$. Note that the existence of at least one $\theta$-psh model function implies that the de Rham class $\{\theta\}\in N^1(X)$ is nef. 
Let $\cX$ be a determination of $\theta$. Thus $\theta_\cX|_{\cX_K}$ is nef. As in Proposition~\ref{prop:maxmodel} we can choose finitely many ample line bundles $\cA_i\in\Pic(\cX)$ such that their numerical classes $\a_i\in N^1(\cX/S)$ form a basis of $N^1(\cX/S)$. There exist arbitrarily small positive numbers $\e_i$ such that $\theta_\cX+\sum_i\e_i\a_i$ is a rational class, hence the class of a $\Q$-line bundle $\cL_\e$ on $\cX$, whose restriction to $\cX_K$ is ample. Since $0$ is a pointwise limit of $\theta$-psh model functions and since $\|\cdot\|_{\cL_\e}e^{-\psi}$ is semipositive for each $\theta$-psh model function $\psi$, we may now apply Lemma~\ref{lem:closed} to conclude that $\cL_\e$ is nef. It follows that $\theta_\cX\in\Nef(\cX/S)$ by closedness of the nef cone. 
\end{proof}

\begin{rmk}
The use of multiplier ideals in Step 2 is similar to~\cite[Proposition 2.8]{ELMNP}, and
very much in the spirit of the arguments we shall use to prove Theorem~B.
It would be interesting to have a proof along the lines of~\cite[p.178, Proposition 8]{Good}. 
\end{rmk}

\subsection{Comparison of terminology}
The terminology for (semipositive) model metrics is unfortunately not uniform across the literature. Here is a tentative summary. 

\medskip
\begin{table}[h]
  \begin{tabular}{|l| l| }
    \hline
    \emph{Model metric}:~\cite{YZ09}  
    & 
    \emph{Semipositive continuous metric}:~\cite{CL1,CL2}
    \\
    Algebraic metric:~\cite{BPS,CL1,Liu11} 
    & 
    Approachable metric:~\cite{BPS}
    \\
    Smooth metric:~\cite{CL2}
    & 
    Semipositive metric:~\cite{YZ09,YZ13a,YZ13b,Liu11} 
    \\
    Root of an algebraic metric:~\cite{Gub98} 
    & 
    Semipositive admissible metric:~\cite{Gub98}\\
    \hline 
  \end{tabular}
  \caption{Terminology for metrics on line bundles.}\label{table:modmet}
\end{table}


\section{Equicontinuity}\label{S209}
The following result is the key to the compactness property in Theorem~A.
\begin{thm}\label{thm:main} 
  Let $X$ be a smooth connected projective $K$-analytic space. 
  Let $\cX$ be a SNC model of $X$ 
  and  $\theta\in\cZ^{1,1}(\Xa)$ a closed $(1,1)$-form determined on $\cX$.
  Then there exists a constant $C=C(\cX,\theta)>0$ such that for every 
  $\theta$-psh model function $\f$, the composition 
  $\f\circ\emb_\cX$ is convex, piecewise affine 
  and $C$-Lipschitz continuous on each face of  $\D_\cX$.  
\end{thm}
\begin{cor}\label{C201}
  With the same notation, the family
  \begin{equation*}
    \left\{\f\circ\emb_\cX\mid\f\ \text{a $\theta$-psh model function}\right\}
    \subset C^0(\D_\cX)
  \end{equation*}
   is equicontinuous on $\D_\cX$. 
\end{cor}
The rest of this section is devoted to the proof of Theorem~\ref{thm:main}.
For the sake of notational simplicity we will (in this section only) 
ignore the map $\emb_\cX$ and simply view $\D:=\D_\cX$ as a subset of~$\Xan$. 

\medskip
Let us first set some notation.
Let $\cX_0=\sum_{i\in I}b_iE_i$ be the irreducible decomposition 
of the special fiber of $\cX$ and 
$e_i=\ev_\cX(x_{E_i})$ the vertex of $\D$ corresponding to $E_i$. 
Recall that the faces $\sigma_J$ of the simplical complex $\D$ 
correspond to subsets $J\subset I$ such that $E_J=\bigcap_{j\in J}E_j$ 
is non-empty, in such a way that $\sigma_ J$ is a simplex with 
$\{e_j,\,j\in J\}$ as its vertices. 
The \emph{star} $\sta(\sigma)$ of a face $\sigma$ of $\D$ is defined as usual as the
union of all faces of $\D$ containing $\sigma$. 
An irreducible component $E_i$ intersects $E_J$ iff the corresponding vertex 
$e_i$ of $\D$ belongs to $\sta (\sigma_J)$; the intersection
is proper iff $e_i\not\in\sigma_J$.

Fix an ample line bundle $\cA$ on $\cX$.
In what follows we denote by $C>0$ a dummy constant, 
which may vary from line to line but only depends on 
$\cX$, $\theta$ and $\cA$. 

Let $\f\in\cD(\Xa)$ be a $\theta$-psh model function
and $\pi:\cY\to\cX$ a vertical blowup such 
that $\f=\f_G$ for some $G\in\Div_0(\cY)_\Q$. 
By Proposition~\ref{prop:convex} we have $\sup_{\Xan}\f=\max_{i\in I}\f(e_i)$. Upon replacing $G$ with $G-(\max_{i\in I}\f(e_i))\cY_0$ we may thus assume that $\f$ is normalized by $\sup_{\Xan}\f=0$. 

\subsection{Bounding the values on vertices} \label{secbounding}
We first prove
\begin{equation}\label{e:bd-val-vert}
  \max_{i\in I}|\f(e_i)|\le C. 
\end{equation}
Recall that we have normalized $\f$ so that $\max_{i\in I}\f(e_i)=0$. We may therefore assume that $\cX_0$ has at least two irreducible components. Observe that the push-forward of the $\Q$-Weil divisor $G$
is given by $\pi_*G=\sum_j b_j\f(e_j)E_j$. 
For each $i\in I$, the projection formula shows that 
$$
(\theta_\cX+\pi_*G)\cdot E_i\cdot\cA^{n-1}
=(\pi^*\theta_\cX+G)\cdot\pi^*E_i\cdot(\pi^*\cA)^{n-1},
$$
which is non-negative since $\pi^*E_i\in\Div_0(\cY)$ is effective, and
both classes $\pi^*\cA$ and $\pi^*\theta_\cX+G$ are nef (the latter because $\varphi$ is 
$\theta$-psh).
It follows that there exists $C=C(\cX,\theta,\cA)$ such that
\begin{equation}\label{equ:sumf}
\sum_jb_j\f(e_j)(E_i\cdot E_j\cdot\cA^{n-1}) + C \ge 0
\end{equation}
for all $i$.
Note that $E_i\cdot E_j\cdot\cA^{n-1}\ge 0$ for all
$i \neq j$, with strict inequality if $E_i\cap E_j\ne\emptyset$.
Thus
$$
b_iE_i \cdot E_i \cdot \cA^{n-1}  = E_i \cdot (b_iE_i-\cX_0) \cdot \cA^{n-1}
= - \sum_{j\neq i} b_j\,  E_i \cdot E_j \cdot \cA^{n-1} \le -1
$$
for all $i$, since $\cX_0$ has 
connected support and contains at least two irreducible components.

Now pick 
$i_0,\dots,i_M$ such that $\f(e_{i_0}) = 0$, $\f(e_{i_M}) =\min_{i\in I}\f(e_i)$, and
$e_{i_m}$ and $e_{i_{m+1}}$ are connected by a $1$-dimensional face, so that 
$E_{i_m}\cdot E_{i_{m+1}}\cdot\cA^{n-1} \ge 1$. 
Write
$$
\lambda:= \max_{i\in I}\left\{-b_i E_i^2 \cdot \cA^{n-1}\right\}\ge 1.
$$ 
Applying~\eqref{equ:sumf} to $i=i_m$, $0\le m<M$, we get
\begin{multline*}
  \lambda\f(e_{i_m})
  \le-b_{i_m}\f(e_{i_m})(E_{i_m}^2\cdot\cA^{n-1})
  \le C+\sum_{j\ne i_m}b_j\f(e_j)(E_{i_m}\cdot E_j\cdot\cA^{n-1})\\
  \le C+b_{i_{m+1}}\f(e_{i_{m+1}})(E_{i_m}\cdot E_{i_{m+1}}\cdot\cA^{n-1})
  \le C+\f(e_{i_{m+1}}),
\end{multline*}
so that 
\begin{multline*}
  0\ge \f(e_{i_M}) 
  \ge -C+\lambda\f(e_{i_{M-1}})
  \ge-C-C\lambda+\lambda^2\f(e_{i_{M-2}})
  \ge\dots\ge\\
  \ge-C\sum_{m=0}^{M-1}\lambda^m+\lambda^M\f(e_{i_0})
  =-C\sum_{m=0}^{M-1}\lambda^m,
\end{multline*}
which proves~\eqref{e:bd-val-vert}.

\subsection{Special subdivisions}\label{sec:special}
We shall need the following construction, see Figure~\ref{F2}.
Let $\sigma=\sigma_J$ be a face of $\D$ and 
$L\subset I$ the set of vertices of $\D$ contained in
$\sta_\D(\sigma)$.
Consider a rational point $v$ in the relative interior  of $\sigma$.
Given $0<\e<1$ rational and $j\in L$ set $e_j^\e:=\e e_j+(1-\e)v$. 
We shall define a projective simplicial subdivision $\D'=\D'(\e,v)$ of $\D$.

To define $\D'$, we first introduce a polyhedral subdivision $\D^\e=\D^\e(v)$ 
of $\D$ leaving the complement of $\sta_\D(\sigma)$ unchanged, as follows. The set of vertices of $\D^\e$ is precisely 
$(e_i)_{i\in I}\cup (e_j^\e)_{j\in L}$ and the faces of 
$\D^\e$ contained in $\sta(\sigma)$ are of one of the following types:
\begin{itemize}
\item
  if the convex hull $\Conv(e_{j_1}, \dots ,  e_{j_m})$ is a face of $\D$ 
  containing $\sigma$, then
  $\Conv( e_{j_1}^\e, \dots ,  e_{j_m}^\e)$ is a face of $\D^\e$;
\item
  if  $\Conv( e_{j_1}, \dots ,  e_{j_m})$ is a face of $\D$ contained in $\sta(\sigma)$ 
  but not containing $\sigma$, then both
  $\Conv( e_{j_1},\dots ,  e_{j_m})$ and $\Conv( e_{j_1}, \dots,e_{j_m},e_{j_1}^\e,\dots,e_{j_m}^\e)$ 
  are faces of $\D^\e$.
\end{itemize}
In a neighborhood of $v$, note that the subdivision $\D^\e$
is obtained by scaling $\D$ by a factor $\e$.
More precisely,  consider the affine map $\psi^\e:\sta(\sigma)\to\sta(\sigma)$ 
defined by $\psi^\e (w) = \e w+(1-\e)v$.
Then $\sigma^\e:=\psi^\e (\sigma)$ is the face of $\D^\e$ containing 
$v$ in its relative interior, and $\psi^\e(\sta_{\D}(\sigma))=\sta_{\D^\e} (\sigma^\e)$.
In particular, even though $\D^\e$ is not simplicial in general, 
all polytopes of $\D^\e$ containing $\sigma^\e$ are simplicial.

We claim that $\D^\e$ is projective. 
To see this, write $v=\sum_{j\in J}s_je_j$, with $s_j>0$ rational 
and $\sum s_j=1$. For $j\in J$, define a linear function $\ell_j$ 
on $\sum_{i\in I}\R_+e_i\supset\D$ by $\ell_j(\sum t_ie_i)=-t_j/s_j$
and set $h=\max\{\max_{j\in J}\ell_j,-(1-\e)\}$. 
A suitable integer multiple of $h$ 
is then a strictly convex 
support function for $\D^\e$ in the sense of~\S\ref{S201}.

Now define $\D'=\D'(\e)$ as a simplicial subdivision of 
$\D^\e$ obtained using repeated barycentric subdivision
in a way that leaves $\sta_{\D^\e}(\sigma^\e)$ unchanged. 
By~\cite[pp.115--117]{KKMS}, $\D'$ is still projective.

Note that $\sigma':=\sigma^\e$ is the face of $\D'$ 
containing $v$ in its relative interior.
For $j\in L$ set $e'_j=e^\e_j$. These are the vertices of $\D'$ contained
in $\sta_{\D'}(\sigma')$.
\begin{figure}[h]
\begin{center}
     \includegraphics[width=10cm]{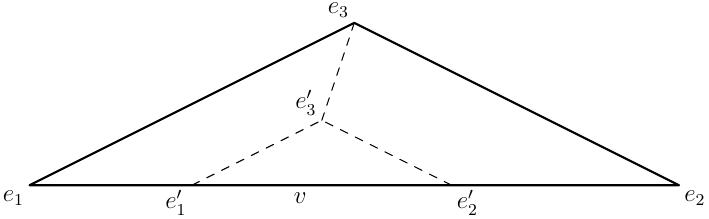}
    \end{center}
    \caption{The subdivision of~\S\ref{sec:special}.
      Here $v$ lies in the relative interior of the simplex $\sigma$ of $\D_\cX$
      with vertices $e_1$ and $e_2$. 
      The picture shows the intermediate subdivision $\D^\e$, where
      $v$ lies in the relative interior of the simplex $\sigma'$ with 
      vertices $e'_1$ and $e'_2$.
      The final subdivision $\D'$ is obtained from $\D^\e$ by barycentric
      subdivision of the quadrilaterals
      $\Conv(e_1,e_3,e'_1,e'_3)$ and $\Conv(e_2,e_3,e'_2,e'_3)$}\label{F2}
  \end{figure}

\subsection{Bounding Lipschitz constants}\label{S302}
Let $\tau$ be a face of $\D$. 
Our aim is to prove by induction on $\dim\tau$ that the $C^{0,1}$-norm of 
$\f$ on $\tau$ is bounded by $C$.
Recall that the $C^{0,1}$-norm is defined as the sum of 
the sup-norm and the Lipschitz norm; see Appendix~A.

The case $\dim\tau=0$ is settled by~\eqref{e:bd-val-vert}, so let us assume
that $\dim\tau>0$.
By Proposition~\ref{prop:convex}, the restriction of $\f$ to $\tau$ is piecewise affine and convex. It therefore admits directional derivatives, and we set as in Appendix A 
$$
D_v\f(w):=\left.\frac{d}{dt}\right|_{t=0+}\f\left((1-t)v+tw\right)
$$
for $v,w\in\tau$. 

Let us say that a codimension $1$ face of $\tau$ is \emph{opposite} a vertex when it is the convex hull of the remaining vertices of $\tau$. This notion is well-defined since $\tau$ is a simplex. 
\begin{prop}\label{prop:keyprop}
  There exists a constant $C>0$ such that 
  $$
  D_v\f(e) \ge -C
  $$
  for any vertex $e$ of $\tau$, and any rational point 
  $v$ in the relative interior of the face $\sigma$ of $\tau$ 
  opposite $e$, such that $\f|_\sigma$ is affine near $v$.
\end{prop}
Note that the assumptions of the proposition are automatically satisfied when 
$\tau$ has dimension 1, with boundary consisting of $e$ and $v$.

Granting this result, let us explain how to conclude the proof.
 By induction we have $\sup_{\partial\tau} |\f| \le C$.
 The convexity of $\f$ and the inductive assumption 
 imply $D_v\f(e)\le\f(e)-\f(v)\le C$ for any  $e,v \in \partial \tau$.
 As for a lower bound, Proposition~\ref{prop:keyprop} gives 
 $D_v\f(e)\ge-C$ for any vertex $e$ of $\tau$ and any rational point $v$ 
 in the relative interior of the face $\sigma$ opposite $e$ such that 
 $\f|_\sigma$ is affine near $v$. 
 Now, by elementary properties of convex functions, 
 $v\mapsto D_v\f(e)$ is upper semicontinuous on 
 the relative interior of $\sigma$, so since $\f$ is piecewise affine,
 the lower bound $D_v\f(e)\ge-C$ holds for any $v$ in the relative interior of $\sigma$. 
 We conclude by Proposition~\ref{prop:lip}
 that the $C^{0,1}$-norm of $\f|_\tau$ is bounded by $C$,
 completing the proof of Theorem~\ref{thm:main}.

 \begin{proof}[Proof of Proposition~\ref{prop:keyprop}]
   Let $I$ be the set of vertices in $\D$, let
   $L\subset I$ be the set of vertices contained in $\sta_\D(\sigma)$
   and $J\subset L$ the set of vertices of $\sigma$.
   Thus $\sigma=\sigma_ J$.
   
   Consider the simplicial projective subdivision $\D'=\D'(\e)$ 
   constructed in~\S\ref{sec:special}. 
   For $j\in L$, $e'_j:=\e e_j+(1-\e)v$ is a vertex of $\D'$.
   Recall that $\sigma'=\sigma'_J$ is the face of $\D'$ containing 
   $v$ in its relative interior.
   Since $\f|_\sigma$ is assumed affine in a neighborhood of 
   $v$,  we may choose $\e>0$ small enough that:
   \begin{itemize}
   \item 
     $\f$ is affine on $\sigma'\subset\sigma$
   \item 
     $\f$ is affine on each segment $[v,e'_j]$, $j\in L$. 
   \end{itemize}

   Let $\rho:\cX'\to\cX$ be the vertical blowup corresponding to the 
   subdivision $\D'$ of $\D$ as in Theorem~\ref{T201}. 
   Note that $\rho$ induces a generically finite map $E'_J\to E_J$ of projective $k$-varieties.
   Indeed, $E_J$ (resp.\ $E'_J$) is the closure of the center of $v$
   on $\cX$ (resp.\ $\cX'$), and both have codimension $|J|$ by Theorem~\ref{T201}. 
   
   Recall that $\f = \f_G$ for some $G \in \Div_0(\cY)$.  We may assume that the determination $\cY$ of $\f$ 
   dominates $\cX'$, so that $\pi$ factors as $\pi=\rho\circ\mu$ with $\mu:\cY\to\cX'$.
   Note that $\mu_*G$ is $\Q$-Cartier since $\cX'$ is vertically $\Q$-factorial.
   
   As we shall see shortly, a first computation shows:
   \begin{lem}\label{lem:pos1} We have
     $$
     \left. \rho^* \left(
         \sum_{j\in L} D_v\f (e_j) \,b_j E_j \right) \right|_{E'_J}
     = (\mu_*G)|_{E'_J}
     $$
     in $\Pic(E'_J)_\Q$. 
   \end{lem}
   
   The key observation is now the following positivity property:
   \begin{lem}\label{lem:pos2}
     If $\cL \in\Pic(\cX')$ is nef then 
     $E'_J\cdot\left(\rho^*\theta_{\cX}+\mu_*G\right)\cdot\cL^{n-|J|-1}\ge0$.
   \end{lem}
   Grant this result for the moment. Lemma~\ref{lem:pos1} and the projection formula yield
   \begin{equation*}
     \deg(\rho|_{E'_J})E_J\cdot\left(
       \theta_{\cX}
       + \sum_{j\in L} D_v\f (e_j) \, b_j E_j
     \right)
     \cdot\cA^{n-|J|-1}
     =E'_J\cdot(\rho^*\theta_\cX+\mu_*G)\cdot\rho^*\cA^{n-|J|-1}.
   \end{equation*}
   Here the right-hand side is non-negative by Lemma~\ref{lem:pos2}, 
   since $\rho^*\cA$ is nef, and we get
   \begin{equation}\label{equ:lowerdv}
    \sum_{j\in L}D_v\f(e_j)\, b_j\left(
      E_J\cdot E_j\cdot\cA^{n-|J|-1}
    \right)
    \ge -(E_J\cdot\theta_\cX\cdot\cA^{n-|J|-1}).
  \end{equation}
  By induction, the $C^{0,1}$-norm of $\f|_\sigma$ is under control.
  Since $v$ belongs to $\sigma=\sigma_J$, this gives
  $$
  |\f(v)|\le C\text{ and }\max_{j\in J}|D_v\f(e_j)|\le C,
  $$
  and (\ref{equ:lowerdv}) yields a lower bound
  \begin{equation}\label{e301}
    \sum_{j\in L\setminus J} D_v \f (e_j)\, b_j\left(E_J\cdot E_j\cdot\cA^{n-|J|-1}\right)\ge-C.
  \end{equation}
Now the convexity of $\f$ and the normalization $\sup_{\Xan}\f=0$ show that
  $$
  \max_{j\in L\setminus J} D_v\f(e_j)\le\max_{j\in L\setminus J}(\f(e_j)-\f(v))\le -\f(v)\le C. 
  $$ 
  Here $E_j|_{E_J}$ is a non-zero effective divisor for $j\notin J$, hence
  $E_J\cdot E_j\cdot\cA^{n-|J|-1}\ge1$ since $\cA$ is ample.
  From~\eqref{e301} we therefore obtain, as desired, 
  that $D_v\f(e)\ge -C$, since $e=e_j$ for some $j\in L\setminus J$. 
\end{proof}
\begin{proof}[Proof of Lemma~\ref{lem:pos1}]
  We write $v = \sum_{j\in J} s_j e_j$ with $s_j>0$ rational and
  $\sum_{j\in J} s_j=1$. Set $s_i=0$ for $i\in I\setminus J$.
  For $i\in I$ let $\f_i$ 
  be the model function induced by the vertical divisor 
  $b_iE_i\in\Div_0(\cX)$. This function is affine on each face of $\D$
  and satisfies $\f_i(e_j) =\d_{ij}$ for all $j\in I$. 
  Since $e'_j=\e e_j+(1-\e) v$ for $j\in L$ we get:
  \begin{equation*}
    \f_i(e'_j) =
    \begin{cases}
      \e+(1-\e) s_i   & \text{if $i= j\in J$}\\
      (1-\e) s_i & \text{if $i \ne j \in J$} \\
      \e & \text{if $i= j\in L\setminus J$}\\
      0 & \text{if $i\ne  j\in L\setminus J$}
    \end{cases} 
  \end{equation*}
  By Theorem~\ref{T201}, 
  $E'_j$ intersects $E'_J$ iff $j\in L$.
  We thus have
  $$
  \rho^*(b_i E_i)|_{E'_J}= \sum_{j\in L} \f_i(e'_j)b'_j\, E'_j|_{E'_J}\ \text{for all $i\in I$}
  $$
  and 
  $$
  (\mu_*G-\f(v)\rho^*\cX_0)|_{E'_J}=\sum_{j\in L}(\f(e'_j)-\f(v))b'_j\, E'_j|_{E'_J} 
  $$
  in $\Pic(E'_J)_\Q$, where we have set $b'_j:=\ord_{E'_j}(\unipar)$. 
  
  Recall also that $\f$ is affine on each segment $[v,e'_i]$, so that
  $D_v\f(e_i) = \e^{-1}\left(\f(e'_i) - \f(v)\right)$
  for $i\in L$. 
  We can now compute in $\Pic(E'_J)_\Q$
  \begin{multline*}
    \left. \rho^* \left(\sum_{i\in L} D_v \f(e_i)\, b_i E_i\right) \right|_{E'_J}
    =\sum_{i\in L} \e^{-1} \left(\f(e'_i) - \f(v)\right)\left(\sum_{j\in L}\f_i(e'_j) \,b'_j\, E'_j |_{E'_J}\right)
    =\\=
    \sum_{i\in J} \e^{-1} (\f(e'_i)-\f(v)) \left(\e\, b'_i \,E'_i|_{E'_J}
      +s_i\sum_{j\in J} (1-\e)\,b'_j\, E'_j|_{E'_J}\right) +\\+
    \sum_{i\in L\setminus J}\e^{-1}(\f(e'_i) - \f(v))\, \e\, b'_i \,E'_i |_{E'_J}
    =\\
    = 
    \sum_{i\in L}(\f(e'_i)-\f(v))\,b'_i\, E'_i |_{E'_J} 
    +
    \e^{-1}(1-\e)\left(\sum_{i\in J} s_i (\f(e'_i) - \f(v))\right)\left(\sum_{j\in J}\,b'_j\, {E'_j}|_{E'_J}\right)
    \\
    =(\mu_*G- \f(v)\rho^*\cX_0) |_{E'_J}
    =\mu_*G|_{E'_J}
  \end{multline*}
  The second to last equality follows from the fact that $\f $ is affine on the simplex 
  $\sigma'_J$ of $\D'$ so that
  $\sum_{i\in J} s_i\f(e'_i)=\f(v) = \sum_{i\in J} s_i \f(v)$.
  This concludes the proof.
\end{proof}
\begin{proof}[Proof of Lemma~\ref{lem:pos2}]
  Set $F:=\mu^*\mu_*G-G\in\Div_0(\cY)_\Q$. 
  The divisor $G$ is $\mu$-nef since $\mu^*(\rho^*\theta_\cX)+G$ 
  is nef by assumption, and Lemma~\ref{lem:neg} therefore implies that $F$ is effective. 
  
  Let $W$ be the closure of the center of $v$ on $\cY$. Since the center of $v$
  on $\cX'$ is the generic point of $E'_J$, we must have
  $\mu(W)=E'_J$. Note, however, that we do not claim $\dim W=\dim E'_J$. 
  
  By Theorem~\ref{T201}, the function $\f_{\mu_*G}$ is affine on the face 
  $\sigma'_J$ of $\D'$. But $\f_G$ is also affine on $\sigma'_J$ by assumption, and we have 
  $$
  \f_G(e'_j)=\frac{1}{b'_j}\ord_{E'_j}(G)=\f_{\mu_*G}(e'_j)
  \quad\text{for all $j\in J$}.
  $$ 
  It follows that $\f_F\equiv0$ on $\sigma'_J$, and in particular $v(F)=0$ (here we view $v$
as an element of the dual of the space of Cartier divisors supported on the special fiber). 
  But this means precisely that $W$ is not contained in $\supp F$, 
  so that $F|_W$ is an effective $\Q$-Cartier divisor. Hence
  $$
  \mu^*(\rho^*\theta_{\cX}+\mu_*G)|_{E'_J}
  =(\pi^*\theta_{\cX}+\mu^*\mu_*G)|_W
  =(\pi^*\theta_{\cX}+G)|_W+F|_W
  $$
  is the sum of a nef class and an effective class. 
  We conclude by Lemma~\ref{lem:pull-psef} below. 
\end{proof}

\begin{lem}\label{lem:pull-psef} 
  Let $\mu:W\to V$ be a surjective morphism between 
  projective varieties over a field $k$ and let 
  $\a\in N^1(V)$. If $\mu^*\a=\gamma+F$ where $\gamma$ is nef and 
  $F$ is effective then $\left(\a\cdot\b^{\dim V-1}\right)_V\ge 0$ 
  for every nef class $\b\in N^1(V)$. 
\end{lem}
\begin{proof} 
  Let $k^a$ be an algebraic closure of $k$ and let $V'$ be an irreducible 
  component of (the reduced scheme associated to) $V^a:=V\otimes k^a$. 
  There exists a component $W'$ of $W^a$ dominating $V'$. 
  Upon replacing $W$ and $V$ by $W'$ and $V'$ we are reduced to the 
  case where $k$ is algebraically closed. Upon taking successive hyperplane 
  sections of $W$ not containing any component of $F$ and choosing an 
  irreducible component dominating $V$ we may then assume that $\mu$ 
  is generically finite. In that case we have
  $$
  (\deg\mu)\left(\a\cdot\b^{\dim V-1}\right)_V=\left(\mu^*\a\cdot\mu^*\b^{\dim V-1}\right)_W
  $$
  and the result follows from Kleiman's theorem (see e.g.~\cite[Theorem~1.4.9]{Laz}) since $\mu^*\b$ is nef. 
\end{proof}




\section{General $\theta$-psh functions and semipositive singular metrics}\label{S208}
We are now ready to introduce the class of general $\theta$-psh functions 
and their cousins: semipositive singular metrics. The equicontinuity result 
in Corollary~\ref{C201} will be used to show Theorem~A, asserting that 
the space of $\theta$-psh functions is compact up to translation.

Throughout this section we let $X$ be as before a smooth connected projective $K$-analytic variety and fix a closed $(1,1)$-form 
$\theta\in\cZ^{1,1}(\Xa)$ whose de Rham class $\{\theta\}\in N^1(X)$ 
is \emph{ample}.  As before we write ``psh'' as a shorthand 
for ``plurisubharmonic''. Similarly, ``usc'' and ``lsc'' will mean ``upper semicontinuous''
and ``lower semicontinuous'', respectively.
\begin{defi}\label{defi:psh} 
  Let $\theta$ be as above. 
  A \emph{$\theta$-psh function} 
  $\f:\Xan\to[-\infty,+\infty[$ 
  is a usc function such that for each SNC model $\cX$ of $X$
  on which $\theta$ is determined we have
  \begin{enumerate}
  \item[(i)] 
    $\f\le\f\circ \retr_{\cX}$ on $\Xan$;
  \item[(ii)] 
    the restriction of $\f$ to $\emb_\cX(\D_\cX)$ is a uniform limit 
    of restrictions of $\theta$-psh model functions. 
  \end{enumerate}
  We write $\PSH(\Xa,\theta)$ for the set of $\theta$-psh functions on $X$.
\end{defi}
We say that $\f:\Xan\to[-\infty,+\infty[\,$ is \emph{quasi-psh} 
if $\f$ is $\theta$-psh for some $\theta$ as above. Thanks to
Theorem~\ref{thm:closed}, the previous definition is consistent with
Definition~\ref{defi:psh-model} when $\f$ is a model function. 
In particular, constant functions are $\theta$-psh iff $\theta$ is 
semipositive.
\begin{rmk}
  A function on a compact (complex) K\"ahler manifold $X$ 
  is quasi-psh if it is locally the sum of a psh function and a smooth function. 
  Given a closed $(1,1)$-form $\theta$, a $\theta$-psh function $\f$ is
  a quasi-psh function such that $\theta+dd^c\f\ge 0$ in the sense of currents. 
  When the de Rham class $\{\theta\}\in H^{1,1}(X)$ is a K\"ahler class, 
  we have a global characterization: $\theta$-psh functions are decreasing 
  limits of sequences of smooth $\theta$-psh functions, 
  see~\cite[Theorem 1.1]{regularization}. In our current non-Archimedean setting, 
  a general local theory of psh functions is still to be developed
  (although the first steps are taken in~\cite{CLD}). For this reason,
  we work globally and assume that $\{\theta\}$ is ample.
\end{rmk}
Recall the definitions from~\S\ref{S202} of singular and model 
metrics on line bundles.
\begin{defi}\label{defi:possingmet} 
  Let $L$ be an ample line bundle on $X$. 
  A singular metric on $L$ is \emph{semipositive} 
  if it is of the form
  $\|\cdot\|e^{-\f}$, where $\|\cdot\|$ is a model 
  metric and $\f$ is $c_1(L,\|\cdot\|)$-psh.
\end{defi}
One checks that this definition does not depend on the choice of 
reference metric $\|\cdot\|$. Below we shall state various properties
of $\theta$-psh functions. We leave it to the reader to formulate
analogous assertions about semipositive singular metrics on ample 
line bundles.

\subsection{Basic properties}
Fix a closed $(1,1)$-form $\theta\in\cZ^{1,1}(\Xa)$ as above.
From Proposition~\ref{prop:maxmodel} and~\ref{prop:convex} 
we obtain:
\begin{prop}\label{P301}
  The set $\PSH(\Xan,\theta)$ is convex.
  If $\f,\p$ are $\theta$-psh and $c\in\R$, then $\f+c$ and 
  $\max\{\f,\p\}$ are $\theta$-psh.
\end{prop}
\begin{prop}\label{P201}
  If $\f\in\PSH(\Xa,\theta)$ and $\cX$ 
  is an SNC model on which $\theta$ is determined, then:
  \begin{itemize}
  \item[(i)]
    $\f\circ\emb_\cX$ is continuous (and finite) on $\D_\cX$ and convex on each face;
  \item[(ii)]
    $\f\circ \retr_\cX$ is continuous on $\Xan$ and $\f\le\f\circ \retr_\cX$.
 \end{itemize}
\end{prop}
The next result shows how to reconstruct a 
$\theta$-psh function from its values on quasimonomial points.
\begin{prop}\label{prop:usc}
  Let $\f\in\PSH(\Xa,\theta)$. 
  Then, as $\cX$ runs through the directed set of SNC
  models on which $\theta$ is determined, 
  $(\f\circ \retr_\cX)_\cX$ forms a decreasing net
  of continuous functions on $\Xan$, converging pointwise to $\f$.
\end{prop}
\begin{proof} 
  Let $\cX'\ge\cX$ be two SNC models on which $\theta$ is determined.
  Then $\retr_\cX\circ \retr_{\cX'}=\retr_\cX$. By Proposition~\ref{P201}~(ii) this implies
  $\f\le\f\circ \retr_{\cX'}\le\f \circ \retr_{\cX} \circ \retr_{\cX'}=\f\circ \retr_{\cX}$,
  with equality on $\emb_\cX(\D_\cX)$.
  Set $\tilde{\f}:=\lim_\cX\f\circ \retr_\cX$. 
  Then $\tilde{\f}\ge\f$. On the other hand, 
  by Corollary~\ref{C205},
  $\retr_\cX$ converges to the identity on $\Xan$, 
  so by upper semicontinuity of $\f$ we have $\f\ge\tilde{\f}$.
\end{proof}
\begin{cor}\label{C302}
  If $\f,\p\in\PSH(\Xan,\theta)$ and $\f=\p$ on $\Xdiv$, then $\f=\p$.
\end{cor}
\begin{proof}
  Let $\cX$ be an SNC model on which $\theta$ is determined.
  By Corollary~\ref{cor:ratdiv}, $\Xdiv$ is dense in $\emb(\D_\cX)$ and 
  by Proposition~\ref{P201}, $\f$ and $\p$ are continuous on $\emb(\D_\cX)$, so
  our assumptions imply that $\f=\p$ on $\emb(\D_\cX)$.
  Thus $\f\circ\retr_\cX=\p\circ\retr_\cX$. Letting $\cX$ run though all
  sufficiently large models, we conclude using  Proposition~\ref{prop:usc}.
\end{proof}

\subsection{Equicontinuity}
The Lipschitz estimates in Theorem~\ref{thm:main} carry over to
general $\theta$-psh functions. As a consequence we have 
\begin{cor}\label{C301}
  For any SNC model on which $\theta$ is determined, the family
  \begin{equation*}
    \{\f\circ\emb_\cX\mid \f\in\PSH(\Xan,\theta)\}
  \end{equation*}
  is an equicontinuous family of functions on $\D_\cX$.
\end{cor}
\begin{cor}\label{C303}
  If $(\f_i)_i$ is a decreasing net in $\PSH(\Xan,\theta)$ and
  $\f=\lim_i\f_i$,
  then either $\f\equiv-\infty$ or $\f\in\PSH(\Xan,\theta)$.
\end{cor}
\begin{proof}
  Assume $\f\not\equiv-\infty$, and the upper
  semicontinuity of $\f_i$ implies that of $\f$.
  Let $\cX$ be any SNC model on which $\theta$ is determined. The inequality $\f_i\le\f_i\circ\retr_\cX$ implies $\f\le\f\circ\retr_\cX$, and hence 
  $$
  \sup_{\D_\cX}(\f\circ\emb_\cX)=\sup_X\f
  $$
  is finite. By Proposition~\ref{P201}, the supremum of each $\f_i$ is 
  attained on the finite set of divisorial valuations associated to
  vertices of $\cD_X$. 
  Hence 
  $$
  \sup_{\D_\cX}(\f_i\circ\emb_\cX)\ge\sup_{\D_\cX}(\f\circ\emb_\cX)>-\infty
  $$
  for all $i$. It then follows from Corollary~\ref{C301} that 
  $\f_i\circ\emb_\cX$ converges uniformly to $\f\circ\emb_\cX$ and 
  that $\f\circ\emb_\cX$ is continuous.
  \end{proof}

\subsection{Compactness}\label{S301}
We endow the set $\PSH(\Xa,\theta)$ of all $\theta$-psh functions 
with the topology of uniform convergence on dual complexes. 
A basis of open neighborhoods of a fixed $\theta$-psh function 
$\f_0$ is then given by $\{\f\mid\sup_{\D_\cX} |\f-\f_0|\le \e\}$ 
where $\cX$ ranges over SNC models on which $\theta$ is determined and
where $\e>0$.
Thanks to Proposition~\ref{prop:usc}, the natural map
\begin{equation*}
  \PSH(\Xa,\theta)\to\prod_\cX C^0(\D_\cX)
\end{equation*}
is a homeomorphism onto its image. 
Note also that $\cD(\Xa)\cap\PSH(\Xa,\theta)$ 
is dense in $\PSH(\Xa,\theta)$ by definition. 
The following result is a more precise version of Theorem~A.
\begin{thm}\label{thm:proper} 
  If $\p\in\cD(X)$ is any model function, then 
  the map $\PSH(\Xa,\theta)\to\R$ defined by $\f\mapsto\sup_{\Xan}(\f-\p)$ 
  is continuous and proper. Hence $\PSH(\Xa,\theta)/\R$ is compact.
  Furthermore, the topology on $\PSH(\Xa,\theta)$ is equivalent to 
  the topology of pointwise convergence on either $\Xqm$ or $\Xdiv$.
\end{thm}
\begin{proof} 
  Let $\cX$ be an SNC model on which $\theta$ and $\p$ are determined.
  Then $\p=\p\circ\retr_\cX$,  so Proposition~\ref{P201} implies that
  $(\f-\p)\circ\retr_\cX$ is continuous and that 
  $(\f-\p)\circ\retr_\cX\ge(\f-\p)$ for any $\f\in\PSH(\Xan,\theta)$.
  Hence the supremum of $\f-\p$ is attained on $\emb_\cX(\D_\cX)$,
  which implies the continuity of $\f\mapsto\sup_\Xan(\f-\p)$.

  To prove properness, we need to show that 
  $$
  \cF_C:=\{\f\in\PSH(\Xa,\theta)\mid |\sup_{\Xan}(\f-\p)|\le C\}
  $$ 
  is compact for any $C>0$.
  Recall that $\PSH(\Xa,\theta)$ embeds in 
  $\prod_\cX C^0(\D_\cX)$. By Tychonoff's theorem, the compactness of $\cF_C$
  is therefore equivalent to the compactness in $C^0(\D_\cX)$ 
  of the closure of the image of $\cF_C$ in $C^0(\D_\cX)$, for each SNC 
  model $\cX$ on which $\theta$ and $\p$ are  determined. 
  But this is a direct consequence of Corollary~\ref{C301} 
  and Ascoli's theorem.

  Picking $\p=0$, we see that $\f\mapsto\sup_X\f$ is proper, which implies
  the compactness of $\PSH(X,\theta)$.

  For the last statement, it is clear that convergence in $\PSH(\Xa,\theta)$
  implies pointwise convergence on $\Xqm$ which in turn implies pointwise
  convergence on $\Xdiv$.
  Now let $(\f_\a)_{\a\in A}$ be a net of $\theta$-psh functions converging 
  pointwise to $\f\in\PSH(\Xa,\theta)$ on $\Xdiv$. 
  Fix any SNC model $\cX$ on which $\theta$ is determined.
  We must show that $\f_\a$ converges uniformly to $\f$
  on $\emb_\cX(\D_\cX)$. But $\Xdiv\cap\emb_\cX(\D_\cX)$ is dense in $\emb_\cX(\D_\cX)$
  by Corollary~\ref{cor:ratdiv}, so this 
  follows from the equicontinuity statement in Corollary~\ref{C301}.
\end{proof}

\subsection{Upper envelopes}
Finally we shall prove the 
following result, whose complex analogue
serves as a basic ingredient of pluripotential theory. 
While we will not go deeper into pluripotential theory here, 
we will use the result below in~\S\ref{S204}.
The proof will in fact not use the compactness result in~\S\ref{S301}.
\begin{thm}\label{thm:sup} 
  Let $(\f_\a)_{\a\in A}$ be an arbitrary family of $\theta$-psh functions on 
  $\Xan$, and assume that $(\f_\a)$ is uniformly bounded from above. 
  If we set $\f(x):=\sup_{\a\in A}\f_\a(x)$ for each $x\in\Xan$, 
  then the usc regularization $\f^*$ of $\f$ is $\theta$-psh and 
  coincides with $\f$ on $\Xqm=\bigcup_\cX\emb_\cX(\D_\cX)$. 
  Further, we have $\f^*=\lim_\cX\f\circ\retr_\cX=\inf_\cX\f\circ\retr_\cX$.
\end{thm}
Recall that the \emph{usc regularization} of a function $u$ on a 
topological space $X$ is the smallest usc function $u^*\ge u$, given by $u^*(x)=\limsup_{y\to x}u(y)$.

\begin{proof}[Proof of Theorem~\ref{thm:sup}]
  Upon considering the new family $\f_I=\max_{\a\in I}\f_\a$ with $I$
  ranging over all finite subsets of $A$, we may assume that $A$ is a 
  directed set and $(\f_\a)$ is an increasing net. 
  We must show that conditions~(i)--(ii) of Definition~\ref{defi:psh} hold for $\f^*$. 

  For each SNC model $\cX$ on which $\theta$ is determined, 
  $\f_\a\circ\emb_\cX$ is convex and continuous on $\D_\cX$ for
  each $\a$. The increasing limit $\f\circ\emb_\cX=\lim_\a\f_\a\circ\emb_\cX$ is therefore
  convex and lsc on $\D_\cX$. 
  On the other hand, any convex function on
  a convex polytope is usc, see~\cite{GKR}.
  Thus $\f\circ\emb_\cX$ is continuous, and, by Dini,
  $\f_\a\circ\emb_\cX$ converges uniformly on $\D_\cX$ to $\f\circ\emb_\cX$.
  Equivalently, $\f\circ\retr_\cX$ is continuous and $\f_\a\circ\retr_\cX$ converges
  uniformly to $\f\circ\retr_\cX$ on $\Xan$.

  Since $\cX\mapsto\f_\a\circ\retr_\cX$ is decreasing, the same is true
  for $\cX\mapsto\f\circ\retr_\cX$. We claim that $\lim_\cX\f\circ\retr_\cX=\f^*$.
  To see this, note that $\p:=\lim_\cX\f\circ\retr_\cX$ is a decreasing limit of
  continuous functions and hence usc. 
  On the one hand, $\f\circ\retr_\cX\ge\f$ for all $\cX$ implies $\p\ge\f$, so that 
  $\p\ge\f^*$ by the definition of the usc regularization.
  On the other hand, Corollary~\ref{C205} and the upper semicontinuity of $\f^*$
  implies
  \begin{equation*}
    \f^*
    \ge\varlimsup_\cX\f^*\circ\retr_\cX
    \ge\varlimsup_\cX\f\circ\retr_\cX
    =\p,
  \end{equation*}
  so that $\f^*=\lim_\cX\f\circ\retr_\cX=\inf_\cX\f\circ\retr_\cX$, hence the claim. 
  
  If $\cY\le\cX$ are SNC models on which $\theta$ is determined, 
  then $\retr_\cX\circ\retr_\cY=\retr_\cY$ by Proposition~\ref{P206}~(iv).
  By what precedes, this implies
  $\f^*\circ\retr_\cY=\lim_\cX\f\circ\retr_\cX\circ\retr_\cY=\f\circ\retr_\cY$.
  In particular, $\f^*=\f$ on $\emb_\cY(\D_\cY)$. Since $\cY$ was arbitrary, we get 
  $\f^*=\f$ on $\Xqm$.

  It remains to prove that $\f^*$ is $\theta$-psh in the sense of 
  Definition~\ref{defi:psh}.
  First, $\f_\a\le\f_\a\circ p_\cX$ 
  for all $\a$ implies $\f\le\f\circ p_\cX$, which gives 
  $\f^*\le(\f\circ p_\cX)^*=\f\circ\retr_\cX=\f^*\circ\retr_\cX$,
  where the first equality follows from the continuity of $\f\circ\retr_\cX$.
  Thus~(i) holds. 
  Second,
  the restriction of each $\f_\a$ to $\emb_\cX(\D_\cX)$ is, by assumption, a
  uniform limit of $\theta$-psh model functions. Since $\f_\a$ converges uniformly
  to $\f=\f^*$ on $\emb_\cX(\D_\cX)$, we see that the restriction of $\f^*$ to $\emb_\cX(\D_\cX)$
  is also a uniform limit of $\theta$-psh model functions. 
  Thus~(ii) also holds and  $\f^*$ is $\theta$-psh.
\end{proof}




\section{Envelopes and regularization}\label{S204}
We continue to assume that $\theta\in\cZ^{1,1}(\Xa)$ is a closed $(1,1)$-form whose de Rham class $\{\theta\}\in N^1(X)$ is ample. This positivity property will be crucial in the arguments to follow.

\subsection{Regularity of envelopes}
As a tool to prove our regularization theorem, we rely on the following envelope construction, whose complex analogue is widely used. 

\begin{defi} The \emph{$\theta$-psh envelope} $P_\theta(u)$ of a continuous function $u\in C^0(\Xa)$ is defined by setting, for each $x\in X$.
$$
P_\theta(u)(x):=\sup\left\{\f(x)\mid\f\in\PSH(\Xan,\theta),\,\f\le u\text{ on }X\right\}.
$$
\end{defi}

Here are a few easy properties of the envelope operator. 
\begin{prop}\label{prop:basicenv} Let $u,u'\in C^0(X)$. 
\begin{itemize}
\item[(i)] $P_\theta(u)$ is $\theta$-psh, and is the largest $\theta$-psh function dominated by $u$ on $X$. 
\item[(ii)] $P_\theta$ is non-decreasing, \ie $u\le v\Rightarrow P_\theta(u)\le P_\theta(v)$. 
\item[(iii)] $P_\theta(u)$ is concave in both arguments, in the sense that 
$$
P_{t\theta+(1-t)\theta'}\left(tu+(1-t)u'\right)\ge t P_\theta(u)+(1-t)P_\theta'(u')
$$ 
for $0\le t\le 1$.
\item[(iv)] For each $c\in\R$ we have $P_\theta(u+c)=P_\theta(u)+c$.
\item[(v)] For each $v\in\cD(X)$ we have $P_{\theta}(u)=P_{\theta+dd^c v}(u-v)+v$. 
\item[(vi)] $P_\theta$ is $1$-Lipschitz continuous with respect to the sup-norm, \ie $\sup_X|P_\theta(u)-P_\theta(v)|\le\sup_X|u-v|$. 
\item[(vii)] Given a determination $\cX$ of $\theta$ and a convergent sequence $\theta_m\to\theta$ in $N^1(\cX/S)$, we have $P_{\theta_m}(u)\to P_{\theta}(u)$ uniformly on $X$. 
\end{itemize}
\end{prop}
\begin{proof} (i) The only thing to show is that $P_\theta(u)$ is $\theta$-psh. Since $P_\theta(u) \le u$ and $u$ is continuous, it follows that the usc regularization satisfies $P_\theta(u)^*\le u$. Now, $P_\theta(u)^*$ is $\theta$-psh by Theorem~\ref{thm:sup}, and is hence a competitor in the definition of $P_\theta(u)$. Thus $P_\theta(u)=P_\theta(u)^*$ is indeed $\theta$-psh. 

(ii) is trivial. 

(iii) follows from the fact that given $\f\in\PSH(\Xan,\theta)$, $\f'\in\PSH(\Xan,\theta')$ with $\f\le u$ and $\f'\le u'$, $t\f+(1-t)\f'$ belongs to $\PSH(\Xan,t\theta+(1-t)\theta')$ and is dominated by $t u+(1-t)u'$. 

(iv) and (v) are seen similarly. 

(vi) is a formal consequence of (ii) and (iv). 

(vii) By Proposition~\ref{prop:positiverep} we may assume, after perhaps passing to a higher model, that there exists a model function $v$ determined on $\cX$ such that $\theta+dd^c v$ is $\cX$-positive, \ie determined by an ample class in $N^1(\cX/S)$. As a consequence, there exists an open neighborhood $V\subset N^1(\cX/S)$ of $\theta$ such that $\tau+dd^c v$ is $\cX$-positive for all $\tau\in V$. 

We claim that $P_\tau(u)$ is uniformly bounded on $X$ for $\tau\in V$. Indeed for each $\tau\in V$ we have $\R\subset\PSH(\Xan,\tau+dd^c v)$, hence $P_{\tau+dd^c v}(u-v)\ge\inf_X u-\sup_X|v|$. By (v) it follows that 
$$
\inf_X u-2\sup_X|v|\le P_{\tau}(u)\le\sup_X u,
$$
which proves the claim. 

Now for each $x\in X$ the function $\tau\mapsto P_\tau(u)(x)$ is concave on $V$, hence locally Lipschitz continuous on $V$, with local Lipschitz constant only depending on $\sup_{\tau\in V}|P_\tau(u)(x)|$, which is in turn bounded independently of $x\in X$, and the result follows. 
\end{proof} 

Our main result in this section is the following regularity property of envelopes. As we shall see, it is in fact equivalent to the monotone regularization theorem. 

\begin{thm}\label{P202}
 For any $u\in C^0(\Xa)$ the $\theta$-psh envelope $P_\theta(u)$ is a uniform limit on $\Xan$ of $\theta$-psh model functions. In particular, $P_\theta(u)$ is continuous.
\end{thm}

Before attacking Theorem~\ref{P202} we shall prove the following
weaker statement.
\begin{lem}\label{L203}
  Let $\tP_\theta(u)$ be the pointwise supremum of all  
  $\theta$-psh \emph{model} functions $\f$ such that $\f\le u$.
  Then $\tP_\theta(u)\le P_\theta(u)$ and equality
  holds on $\Xqm$. 
\end{lem}
\begin{proof}
  The inequality $\tP_\theta(u)\le P_\theta(u)$ is trivial.
  To prove that equality holds on $\Xqm$, pick $\e>0$
  and $x\in\emb_{\cX}(\D_{\cX})$ for some SNC model $\cX$
  on which $\theta$ is determined. 
  By construction, there exists $\psi\in\PSH(\Xa,\theta)$ such that 
  $\psi\le u$ and $\psi(x)\ge P_\theta(u)(x)-\e$. 
  By the definition of $\PSH(\Xa,\theta)$, there then exists a 
  $\theta$-psh model function $\f$ such that $|\f-(\psi-\e)|\le\e$ on 
  $\emb_\cX(\D_\cX)$. 
  Thus $\f\le\f\circ \retr_\cX\le\psi\circ \retr_\cX\le u$ on $\Xan$
  and $\f(x)\ge\psi(x)-2\e\ge P_\theta(u)(x)-3\e$.
  We conclude that $\tP_\theta(u)=P_\theta(u)$ on $\Xqm$.
\end{proof}

\begin{proof}[Proof of Theorem~\ref{P202}]
We shall reduce the statement to a geometric assertion that can be proved using asymptotic multiplier ideals.

First, we may assume that $u\in\cD(X)$, thanks to Corollary~\ref{cor:dense} and (vi) of Proposition~\ref{prop:basicenv}. 

Second, we can reduce to the case when $\theta\in N^1(\cX/S)_\Q$ is a rational class, using (vii) of Proposition~\ref{prop:basicenv}. 

Third, we may further reduce to the case $u=0$, after replacing $\theta$ with $\theta+dd^c u$, using (v) of Proposition~\ref{prop:basicenv}. 

After scaling, we may finally assume that $\theta$ is the curvature form of a model metric determined by a line bundle $\cL$ on some model $\cX$. Now we conclude the proof using the following result.
\end{proof}

\begin{thm}\label{thm:reg} Let $L$ be an ample line bundle on $X$ and $\cL\in\Pic(\cX)$ an extension of $L$ to an SNC model $\cX$. Let $\theta\in\cZ^{1,1}(\Xa)$ be the curvature form of the corresponding model metric on $L$. For $m\gg 1$ let $\fa_m\subset\cO_\cX$ be the (vertical) base-ideal of $m\cL$ and set $\f_m:=\tfrac 1 m\log|\fa_m|$. Then $\f_m$ is a $\theta$-psh model function and $\f_m\to P_\theta(0)$ uniformly on $\Xan$ as $m\to\infty$.
\end{thm}
\begin{proof}[Proof of Theorem~\ref{thm:reg}] For $m\gg 1$, $m\cL|_{\cX_K}$ is globally generated, which shows that the ideal sheaf $\fa_m$ is vertical. Since $\cO_\cX(m\cL)\otimes\fa_m$ is globally generated by the definition of $\fa_m$, it follows that $\f_m\in\cD(\Xa)$ is $\theta$-psh by Lemma~\ref{lem:gen}. Note that $\fa_m\cdot\fa_l\subset \fa_{m+l}$ for all $m,l$. This yields the super-additivity property $m\f_m+l\f_l\le(m+l)\f_{m+l}$. As a consequence, the pointwise limit $\lim_m\f_m$ exists and coincides with $\sup_m\f_m$.

\smallskip
\textbf{Step 1}. Let us first prove that $P_\theta(0)=\sup_m\f_m$ on $\Xqm$. The argument is similar to Step 2 of Theorem~\ref{thm:closed}. Since $\f_m$ is $\theta$-psh and $\f_m\le 0$ for all $m$, we have $\sup_m\f_m\le P_\theta(0)$ on $\Xan$. To see that equality holds on $\Xqm$, pick $\e>0$ and $x\in\emb_{\cX'}(\D_{\cX'})$ for some SNC model $\cX'$ dominating $\cX$. By Lemma~\ref{L203}  there exists a $\theta$-psh \emph{model function} $\f$ such that $\f\le 0$ and $\f(x)\ge P_\theta(0)(x)-\e$. Replacing $\cX'$ by a higher model, we may assume that $\f=\f_D$ is determined by some divisor $D\in\Div_0(\cX')_\Q$. Invoking Proposition~\ref{prop:positiverep} we may also assume that there exists $D'\in\Div_0(\cX')_\Q$ with $-\e\le\f_{D'}\le 0$ on $\Xan$ and  $\pi^*\cL+D+D'$ ample. Since $D+D'\le 0$ we then have 
$$
\cO_{\cX'}(m\pi^*\cL+m(D+D'))\subset\cO_{\cX'}(m\pi^*\cL). 
$$
Now the left-hand side is globally generated for some $m$, and we conclude that 
$$
\cO_{\cX'}(m (D+D'))\subset\cO_{\cX'}\cdot \fa_m;
$$ 
hence
$$
P_\theta(0)(x)\le\f_D(x)+ \e\le\f_{D+D'}(x)+2\e\le \frac1m \log|\fa_m| (x) + 2\e\le\sup_l\f_l(x)+2\e. 
$$
  
\smallskip
\textbf{Step 2}. 
Introduce, for each $m\in\N$, the asymptotic multiplier ideal $\fb_m:=\cJ(\fa_\bullet^m)\subset\cO_\cX$ associated to the graded sequence $\fa_\bullet$. We refer to Appendix B for the definition and the proof of the fundamental properties of multiplier ideals in our present setting. We shall use the following results. First, we have the elementary inclusion $\fa_m\subset\fb_m$ for all $m$. Second, the subadditivity property (cf.~Theorem~\ref{thm:subadd}) implies $\fb_{ml}\subset\fb^l_m$ for any $l,m$. We infer that $\fa_{ml}\subset\fb_{ml}\subset\fb_m^l$ for any $m,l$ and hence
\begin{equation}\label{equ:logcj}
\tfrac 1m\log|\fb_m|\ge\sup_l\tfrac 1{ml}\log|\fa_{ml}|=\sup_l\f_{ml}=P_\theta(0)
\end{equation}
on $\Xqm$ for all $m$, where the last equality follows from Step~1.
  
Since both $P_\theta(0)$ and $\f_m$ remain unchanged when $\cX$ is replaced with a higher model, we may assume that there exists an effective divisor $E\in\Div_0(\cX)_\Q$ such that $\cA:=\cL-E$ is ample on $\cX$. By the uniform global generation property of multiplier ideals  (Theorem~\ref{thm:uniform} and Remark~\ref{rmk:uniform}) we may then choose $m_0\in\N$ such that $m_0\cA$ is ample enough to guarantee that $\cO_\cX(m\cL+m_0\cA)\otimes\fb_m$ is globally generated for all $m$. Since $\cO_\cX(m\cL+m_0\cA)$ injects in $\cO_\cX((m+m_0)\cL)$ by multiplying with the canonical section of $\cO_\cX(m_0E)$, it follows that
$$
\log|\fb_m|\le\log|\fa_{m+m_0}|+m_0\f_E.
$$
Replacing $m$ with $m-m_0$ and using (\ref{equ:logcj}) we infer $(m-m_0)P_\theta(0)\le m\f_m+m_0\f_E$, so that
\begin{equation*}
\f_m\le P_\theta(0)\le\tfrac{m}{m-m_0}\f_m+\tfrac{m_0}{m-m_0}\f_E
\end{equation*}
on $\Xqm$ for $m\gg1$. As $\f_m$, $P_\theta(0)$ and $\f_E$ are all $\theta$-psh, Proposition~\ref{prop:usc} shows that this inequality extends to all of $\Xan$. Now $\f_E$ is bounded and $\f_m$ is uniformly bounded, as follows from $\f_1\le\f_m\le0$, so $\f_m$ converges uniformly on $\Xan$ to $P_\theta(0)$, as was to be shown.
\end{proof}
Let us end this subsection with a result that will be used in~\cite{nama}.
\begin{cor}
If $\f\in\PSH(\Xan,\theta)$ and $v\in C^0(\Xan)$ are such that $\f\le v$ then for every $\e>0$ there exists an $\theta$-psh model function $\psi$ such that $\f\le\psi\le v+\e$.
\end{cor}
\begin{proof}
We may assume $\f=P_\theta(v)$, in which case the result follows from Theorem~\ref{P202}.
\end{proof}

\subsection{Monotone regularization of $\theta$-psh functions}
By our definition, the set $\cD(X)\cap\PSH(\Xan,\theta)$ of $\theta$-psh model functions is dense in $\PSH(\Xan,\theta)$ with respect to its topology of uniform convergence on dual complexes. This property may be seen as an analogue of the fact that every $\theta$-psh function is a $L^1$-limit of smooth $\theta$-psh functions in the complex case, which follows from the much more useful fact that every $\theta$-psh function is a \emph{decreasing} limit of smooth $\theta$-psh functions~\cite{regularization}. The next result gives an analogue of this monotone regularization theorem in our context.

\begin{thm}\label{thm:regmono}
  For each $\theta$-psh function $\f$, there exists a decreasing net 
  $(\f_i)_{i\in I}$ of $\theta$-psh model functions that converges 
  pointwise on $\Xan$ to $\f$. 
\end{thm}
One may hope that there is in fact a decreasing \emph{sequence} $(\f_m)_{m=1}^\infty$ of $\theta$-psh model functions converging to $\f$. 
In the companion paper~\cite{nama}, we will prove---using 
Theorem~\ref{thm:regmono} and capacity estimates---that this is indeed the case.

As a consequence of Theorem~\ref{thm:regmono}, we get at any rate the following version of the Demailly-Richberg regularization theorem.
\begin{cor}\label{cor:richberg} Every continuous $\theta$-psh function $\f$ is the uniform limit on $X$ of a sequence $(\f_m)_{m\in\N}$ of $\theta$-psh model functions. 
\end{cor}
\begin{proof} By Theorem~\ref{thm:regmono} there exists a decreasing net $(\p_j)$ of $\theta$-psh model functions converging pointwise to $\f$. For each $\e>0$ the compact set $X$ is the increasing union of the open sets $\{\p_j<\f+\e\}$, hence $\p_j<\f+\e$ for some $j$ (Dini's lemma). It follows that $\f$ lies in the closure of $\cD(X)\cap\PSH(\Xan,\theta)$ in $C^0(X)$ with respect to the topology of uniform convergence. Since the latter is defined by a norm, the result follows.
\end{proof}

The proof of Theorem~\ref{thm:regmono} reduces immediately to Theorem~\ref{P202}, in view of the following elementary result.

\begin{lem}\label{lem:contenv} The following properties are equivalent.
\begin{itemize}
\item[(i)] Every $\theta$-psh function $\f$ is the pointwise limit of a decreasing net of $\theta$-psh model functions. \item[(ii)] For each $u\in C^0(X)$ we have 
$$
P_\theta(u)=\sup\left\{\f\mid\f\in\cD(X)\cap\PSH(\Xan,\theta),\,\f\le u\text{ on }X\right\}.
$$  
\item[(iii)] For each $u\in C^0(X)$ $P_\theta(u)$ is a uniform limit of $\theta$-psh model functions. 
\end{itemize}
\end{lem}

\begin{proof} (i)$\Longrightarrow$(ii). Let $u\in C^0(X)$. By (i) there exists a decreasing net $(\f_j)$ of $\theta$-psh model functions converging pointwise to $P_\theta(u)$. Since $P_\theta(u)\le u$, we see that the compact set $X$ is for each $\e>0$ the increasing union of the open sets $\{\f_j<u+\e\}$, hence $\f_j<u+\e$ for some $j$. Since $\f_j-\e$ is $\theta$-psh and dominated by $u$, we get $\f_j-\e\le P_\theta(u)$ by definition of the envelope, which proves (ii). 

(ii)$\Longrightarrow$(iii). Since the set of $\f\in\cD(X)\cap\PSH(\Xan,\theta)$ such that $\f\le u$ is stable by max, (ii) shows that we can construct an increasing family $\f_j\in\cD(X)\cap\PSH(\Xan,\theta)$ converging pointwise to $P_\theta(u)$. But $P_\theta(u)-\f_j$ is usc for each $j$, and Dini's lemma therefore shows that the convergence is uniform on $X$. 

(iii)$\Longrightarrow$(i). Let $\f$ be a $\theta$-psh function. We first claim that for each $x\in \Xan$ we have 
\begin{equation}\label{e302}
  \f(x)=\inf\left\{\p(x)\mid\p\in\cD(X)\cap\PSH(\Xan,\theta),\p\ge\f\right\}. 
\end{equation}
Indeed, given $\e>0$ there exists $u\in C^0(X)$ such that $u\ge\f$ and $u(x)\le\f(x)+\e$, simply because $\f$ is usc. Since $\f$ is $\theta$-psh, the maximality property of envelopes implies $\f\le P_\theta(u)$. By (iii) we may then find $\psi\in\cD(X)\cap\PSH(\Xan,\theta)$ such that $P_\theta(u)\le\psi\le P_\theta(u)+\e$. We thus have $\psi\ge\f$ and $\psi(x)\le\f(x)+2\e$, and the claim follows. 
  
Now consider the set $I$ of all $\psi\in\cD(X)\cap\PSH(\Xan,\theta)$ such that $\psi>\f$ on $X$. Note that the latter condition implies that $\psi\ge\f+\e$ for some $\e>0$, since $\f-\psi$ is usc. We claim that $I$ is a directed set, which will conclude the proof, since 
$$
\f=\inf_{\p\in I}\p=\lim_{\p\in I}\p
$$
pointwise on $X$, thanks to (\ref{e302}). To get the claim, let $\psi_1,\psi_2\in I$ and choose $\e>0$ such that $\min\{\psi_1,\psi_2\}\ge\f+3\e$. We then also have $P_\theta(\min\{\psi_1,\psi_2\})\ge\f+3\e$. By (iii) we find $\p_3\in\cD(X)\cap\PSH(\Xan,\theta)$ such that $|\psi_3-P_\theta(\min\{\psi_1, \psi_2\})|\le\e$. 
Then $\f+\e\le\psi_3-\e\le\min\{\psi_1,\psi_2\}$, which concludes the proof.
\end{proof}

\appendix
\section{Lipschitz constants of convex functions}
Let $V$ be a finite dimensional real vector space and $K\subset V$ be a convex body, \ie a compact convex set with nonempty interior. Denote by $\cE(K)$ the set of 
extremal points of $K$. 
Given a norm $\|\cdot\|$ on $V$ the Lipschitz constant 
of a continuous function $\f:K\to\R$ is defined as usual as
$$
\lip_K(\f):=\sup_{v\neq v'}\frac{|\f(v)-\f(v')|}{\|v-v'\|}\in [0,+\infty]
$$
and its $C^{0,1}$-norm is then 
$$
\|\f\|_{C^{0,1}(K)}:=\|\f\|_{C^0(K)}+\lip_K(\f).
$$
This quantity of course depends on the choice of $\|\cdot\|$, but since all norms on $V$ are equivalent, choosing another norm only affects the estimates to follow by an overall multiplicative constant. 

Let $\f:K\to\R$ be a continuous convex function. Our goal is to estimate the $C^{0,1}$-norm of $\f$ on $K$ in terms of $\|\f\|_{C^0(\partial K)}$ and certain directional derivatives of $\f$ at boundary points.
Let us first introduce some notation. 
First, for $v,w\in K$ we define the
directional derivative of $\f$ at $v$ towards $w$ as
\begin{equation}\label{equ:der}
  D_v\f(w):=\left.\frac{d}{dt}\right|_{t=0_+}\f((1-t)v+tw);
\end{equation}
this limit exists by convexity of $\f$.
Second, given a point 
$e\in K$  
we define a projection 
$\pi_e:K\setminus\{e\}\to\partial K$ by setting 
$$
t_e(v):=\sup\left\{t\in\R,\,e+t(v-e)\in K\right\}
$$
and
$$
\pi_e(v):=e+t_e(v)(v-e),
$$ 
so that $\pi_e(v)\in\partial K$ 
is the unique point such that $v\in[e,\pi_e(v)]$. 
\begin{prop}\label{prop:lip}
There exists $C>0$ such that every Lipschitz continuous 
convex function $\f:K\to\R$ satisfies
$$
C^{-1}\|\f\|_{C^{0,1}(K)}\le \|\f\|_{C^0(\partial K)}+\sup_{e\in\cE(K),v\in\mathrm{int}(K)}\left|D_{\pi_e(v)}\f(e)\right|\le C\,\|\f\|_{C^{0,1}(K)}.
$$
\end{prop}
\begin{proof} The right-hand inequality is clear, so we focus on the left-hand one. Given $v\in\mathrm{int}(K)$ and $e\in\cE(K)$ we may write $v=\pi_e(v)+t_0\left(e-\pi_e(v)\right)$ for some $0<t_0<1$. Consider the restriction of $\f$ to the segment $[\pi_e(v),e]$, \ie set $\theta(t):=\f\left(\pi_e(v)+t(e-\pi_e(v)\right)$, $t\in[0,1]$. If we denote by $\theta'(t)$ the right-derivative of $\theta$ at $t$ then the convexity of $\theta$ yields
\begin{equation}\label{equ:min1}
\theta'(0)\le\left(\theta(t_0)-\theta(0)\right)/t_0
\end{equation}
and
\begin{equation}\label{equ:min2}
\theta'(0)\le\theta'(t_0)\le\left(\theta(1)-\theta(t_0)\right)/(1-t_0).
\end{equation}
Now, by definition,
 $\theta'(0)=D_{\pi_e(v)}\f(e)$, $t_0\theta'(0)=D_{\pi_e(v)}\f(v)$ and $(1-t_0)\theta'(t_0)=D_v\f(e)$, so that (\ref{equ:min1}) reads
$$
t_0D_{\pi_e(v)}\f(e)\le\f(v)-\f(\pi_e(v)).
$$
Since we also have $\sup_K\f=\sup_{\partial K}\f$ by convexity, this shows that
$$
\|\f\|_{C^0(K)}\le\|\f\|_{C^0(\partial K)}+\sup_{e\in\cE(K),v\in\mathrm{int}(K)}\left|D_{\pi_e(v)}\f(e)\right|.
$$
On the other hand,~\eqref{equ:min1} combined with (\ref{equ:min2}) yields
$$
(1-t_0)D_{\pi_e(v)}\f(e)\le D_v\f(e)\le\f(e)-\f\left(\pi_e(v)\right)-t_0D_{\pi_e(v)}\f(e)
$$
and we conclude by Lemma~\ref{lem:lip} below.
\end{proof}

\begin{lem}\label{lem:lip}
There exists a constant $C>0$ such that every Lipschitz continuous function $\f:K\to\R$ satisfies
$$
C^ {-1}\lip_K(\f)\le\sup_{e\in\cE(K),v\in A}|D_v\f(e)|\le C\lip_K(\f)
$$
where $A\subset\mathring{K}$ denotes the set of points at which $\f$ is differentiable. 
\end{lem}
\begin{proof} It is clear that $|D_v\f(e)|\le\diam(K)\lip_K(\f)$ for all $e,v$. Conversely it is a standard consequence of Rademacher's theorem that $\lip_K(\f)=\sup_{v\in A}\|\nabla\f(v)\|$. For each $v\in A$ we also have $D_v\f(e)=\langle\nabla\f(v),e-v\rangle$. We now claim that there exists $C>0$ such that 
$$
\|\lambda\|\le C\sup_{e\in\cE(K)}|\langle\lambda,v-e\rangle|
$$
for all $\lambda\in V^*$ and all $v\in K$, which will conclude the proof. Indeed the supremum in the right-hand side is a lower semicontinuous function of $(\lambda,v)\in V^*\times K$. As a consequence it achieves its infimum on the compact set $\{\lambda\in V^*,\,\|\lambda\|=1\}\times K$, and this infimum cannot be zero since $\{v-e,\,e\in\cE(K)\}$ spans $V$ for each $v\in K$. The claim follows by homogeneity.
\end{proof}


\section{Multiplier ideals on $S$-varieties}

The purpose of this section is to define multiplier ideals on regular $S$-varieties and 
establish their basic properties. 
We are grateful to Osamu Fujino, J\'anos Koll\'ar and Mircea Musta\c{t}\v{a} 
for their helpful suggestions. 

It should be noted that most results in this appendix are  also obtained, with simpler arguments and in a slightly more general setting, in \cite{mustata-nicaise}, which appeared after a first version of the present paper was completed. We felt, however, that the alternative arguments presented here might still be of some interest to the reader. 

In this appendix, and as opposed to the main body of the article, it will be more convenient to use multiplicative notation for Picard groups. We also fix the choice of an isomorphism $R\simeq k\cro{t}$.

%
%

\subsection{Kodaira vanishing}
The usual compactification argument that reduces the relative version of Kodaira (or Kawamata-Viehweg) vanishing to its global projective version over $k$ cannot be applied for $S$-varieties. Following suggestions of J\'anos Koll\'ar and Mircea Musta\c{t}\v{a}, we rely instead on a Kodaira-type vanishing theorem on the (possibly reducible) special fiber.

\begin{thm}[Kodaira vanishing]\label{thm:kodaira} 
  Let $\cX$ be an SNC $S$-variety, denote by $\om_\cX$ its dualizing sheaf, and let $\cL\in\Pic(\cX)$ be an ample line bundle. 
  Then we have
  $$
  H^q\left(\cX,\om_\cX\otimes\cL\right)=0\quad\text{for all $q\ge1$}.
  $$ 
\end{thm}
\begin{proof} By flat base change we may assume that $k$ is algebraically closed. The desired result is equivalent to $R^q f_*(\om_\cX\otimes\cL)=0$ for $q\ge 1$ since $S$ is affine, and all fibers of $f:\cX\to S$ are Cohen-Macaulay since $\cX$ is in particular Cohen-Macaulay. We may therefore use relative duality for $f$, which shows that the desired result is equivalent to $R^q f_*\cL^{-1}=0$ for $q<n=\dim X$. 

Let $d\in\N^*$ be a common multiple of the multiplicities of $\cX_0$, set $S_d:=\Spec\cro{t^{1/d}}$ and let $\cY$ be the normalization of $\cX\times_S S_d$, with structure map $g:\cY\to S_d$. The pull-back $\cM$ of $\cL$ to $\cY$ is still ample since $\cY\to\cX$ is finite. By~\cite[pp.200--201]{KKMS} the $S_d$-scheme $\cY$ is toroidal and its special fiber $\cY_0$ is \emph{reduced}.  

The relative trace $\tr_{\cY/\cX}$ shows that $R^q g_*\cM^{-1}$ contains $R^q f_*\cL^{-1}$ as a direct summand, and it is therefore enough to show by semicontinuity that $H^q(\cY_0,\cM^{-1})=0$ for $q<n$. Like any toroidal $S$-scheme, $\cY$ is Cohen-Macaulay. As a consequence, the Cartier divisor $\cY_0$ is Cohen-Macaulay as well. By another application of duality, this time on $\cY_0$, we are reduced to showing that $H^q(\cY_0,\om_{\cY_0}\otimes\cM)=0$ for $q\ge 1$. 

By~\cite{KKMS}  we may choose a toroidal vertical blowup $\pi:\cY'\to\cY$ such that $\cY'_0$ has simple normal crossing support (and hence $\cY'$ is regular). A toric computation (compare~\cite[Proposition 3.7]{Kol}) shows that 
$$
\om_{\cY'}\otimes\cO_{\cY'}(\cY'_{0,\red})\simeq\pi^*\left(\om_{\cY}\otimes\cO_{\cY}(\cY_0)\right). 
$$
Since $\cY_0$ and $\cY'_{0,\red}$ are Cartier divisors on $\cY$ and $\cY'$, respectively, adjunction applies (see for instance~\cite[Proposition 5.73]{KM}) and we get $\om_{\cY'_{0,\red}}\simeq\pi^*\om_{\cY_0}$.
On the other hand, the projective and reduced (but \emph{a priori} reducible) $k$-scheme $\cY'_{0,\red}$ has \emph{embedded} SNC singularities. It is indeed an SNC divisor in $\cY'$, and~\cite{Art} implies the existence of an algebraic $k$-variety containing $\cY'_0$ as a divisor. Since $\pi:\cY'_{0,\red}\to\cY_0$ is projective and $\cM$ is ample on $\cY_0$, we may therefore apply a vanishing theorem originally due to Kawamata and Ambro and corrected by Fujino (\cite[Theorem 4.4]{Kaw},~\cite[Theorem 3.2]{Amb} and~\cite[Theorem 2.39]{Fuj}) to get that $\pi_*\left(\om_{\cY'_{0,\red}}\otimes\pi^*\cM\right)$ is acyclic on $\cY'$. Now we have $\pi_*\om_{\cY'_{0,\red}}=\om_{\cY_0}$, hence
$$
\pi_*\left(\om_{\cY'_{0,\red}}\otimes\pi^*\cM\right)\simeq\om_{\cY_0}\otimes\cM
$$
by the projection formula, and we conclude as desired that $\om_{\cY_0}\otimes\cM$ is acyclic. 
\end{proof}

%
%

\subsection{Kawamata-Viehweg vanishing}
We next explain how to infer from Theorem~\ref{thm:kodaira} a version of the Kawamata-Viehweg vanishing theorem on SNC models. 
We rely as usual on the ``covering trick'' and basically follow the proof of~\cite[Theorem 2.64]{KM} but provide some details for the convenience of the reader. 

\begin{lem}[Covering trick]\label{lem:trick} Assume that $k$ is algebraically closed. Let $\cX$ be an SNC $S$-variety and denote by $(E_i)_{i\in I}$ the set of irreducible components of $\cX_0$. Let also $\cL\in\Pic(\cX)$ and $m\in\N^*$. 
Then there exists an SNC $S$-scheme $\cX'$ and a finite surjective morphism $\rho:\cX'\to\cX$ such that $\rho^*\cL$ is divisible by $m$ in $\Pic(\cX')$ and $\rho^*E_i$ is smooth over $k$ (but possibly disconnected) for all $i\in I$. 
\end{lem} 
We emphasize that the generic fiber of $\cX'$ is a finite cover of the generic fiber of $\cX$.  
\begin{proof} Writing $\cL=(\cA+\cL)-\cA$ for some sufficiently ample $\cA\in\Pic(\cX)$ reduces us to the case where $\cL$ is very ample. We then get a closed embedding $i:\cX\hookrightarrow\P^N_k\times_k S$ over $S$ such that $\cL$ coincides with the restriction of $\cO(1)$. Let $\pi:\P^N_k\to\P^N_k$ be the morphism $[X_0:\dots:X_N]\mapsto[X_0^m:\dots:X_N^m]$, which satisfies $\pi^*\cO(1)=\cO(m)$. For each $\sigma\in\mathrm{PGL}(N+1,k)$ set $i_\sigma:=\sigma\circ i$ and consider $\cX':=\cX\times_{i_\sigma}\pi$ with the finite surjective morphism $\rho:\cX'\to\cX$, so that $\rho^*\cL=\cO(m)|_{\cX'}$ is divisible by $m$ in $\Pic(\cX')$. 

Applying Kleiman's Bertini-type theorem (cf.~\cite[III.10.8]{Har}) to the smooth $k$-varieties $E_J=\bigcap_{j\in J}E_j$ for all subsets $J\subset I$ shows that we may choose $\sigma\in\mathrm{PGL}(N+1,k)$ such that each $\rho^*E_i$ is smooth over $k$ and $\sum_i\rho^*E_i$ has simple normal crossings. This implies in particular that $\cX'$ is an SNC model.
\end{proof} 

\begin{thm}[Kawamata-Viehweg vanishing]\label{thm:KV} Let $\cX$ be an SNC model of $X$. Let $\cL\in\Pic(\cX)$ be a line bundle whose restriction to the generic fiber $X$ is ample and such that $\cL-D$ is nef for some $D\in\Div_0(\cX)_\Q$ with coefficients in $[0,1[$. Then we have
$$
H^q(\cX,\om_\cX\otimes\cL)=0\,\,\text{ for all } q\ge 1.
$$
\end{thm}
\begin{proof} As in Theorem~\ref{thm:kodaira} the desired result is equivalent to $H^q(\cX,\cL^{-1})=0$ for $q<n$ by relative duality. By flat base change we may assume that $k$ is algebraically closed.\newline

\noindent{\bf Step 1}. Assume first that $\cL-D$ is ample. Let $E_1,\dots,E_N$ be the components of 
$\cX_0$ and set $a_i:=\ord_{E_i}D$. Choose $m\in\N^*$ such that $b:=m\,a_1\in\N$. By Lemma~\ref{lem:trick} there exists an SNC $S$-variety $\cX'$ with a finite surjective morphism $\rho:\cX'\to\cX$ such that $\rho^*E_i$ is smooth (possibly disconnected) for each $i$, $\sum_i\rho^*E_i$ is SNC and $\rho^*E_1$ is given as the zero divisor of a section $s\in H^0(\cX',\cM^m)$ for some $\cM\in\Pic(\cX')$. Note that $H^q(\cX,\cL^{-1})$ is a direct summand of $H^q(\cX',\rho^*\cL^{-1})$ thanks to the trace map. Now let 
$$
\cX_1:=\Spec_{\cX'}\left(\bigoplus_{0\le j<m}\cM^{-j}\right)
$$
be the cyclic cover associated with $s\in H^0(\cX',\cM^m)$, where $\bigoplus_{0\le j<m}\cM^{-j}$ is endowed with the $\cO_{\cX'}$-algebra structure induced by $s$. By definition there is a finite surjective morphism $\tau:\cX_1\to\cX'$ which satisfies
$$
\tau_*\cO_{\cX_1}=\bigoplus_{0\le j<m}\cM^{-j}. 
$$
If we set 
$$
\cL_1:=\tau^*\left(\rho^*\cL\otimes\cM^{-b}\right)
$$ 
we thus have
$$
H^q\left(\cX_1,\cL_1^{-1}\right)\simeq\bigoplus_{0\le j<m}H^q\left(\cX',\rho^*\cL^{-1}\otimes\cM^{b-j}\right)
$$ 
by Leray's spectral sequence (since the higher direct images vanish, $\tau$ being finite). 

But $b/m=a_1$ is less than $1$ by assumption, and we thus see that $H^q(\cX_1,\cL_1^{-1})$ contains $H^q(\cX',\rho^ *\cL^{-1})$, hence also $H^q(\cX,\cL^{-1})$, as a direct summand. 

Since $\rho^*{E_i}$ is smooth for each $i$ and $\sum_{i=2}^N\rho^*E_i$ has normal crossings with $\div(s)=\rho^*E_1$, one sees as in~\cite[Claim 2.65]{KM} that $E_i^{(1)}:=\tau^*\rho^ *E_i$ is smooth for each $i$ and $\cX_{1,0}$ has SNC support, so that $\cX_1$ is an SNC $S$-scheme. Finally $\cL_1-\sum_{i=2}^N a_iE_i^{(1)}$ is $\Q$-linearly equivalent to $\tau^*\rho^ *(\cL-D)$, hence is ample. 

We now use Lemma~\ref{lem:trick} to find $c_1:\cX_1'\to\cX_1$ such that $c_1^*E_i^{(1)}$ is smooth for all $i$, $\sum_ic_1^*E_i^{(1)}$ is SNC and $c_1^*E_2^{(1)}$ is divisible in $\Pic(\cX_1')$ by the denominator of $a_2$. We then perform the same cyclic cover construction as above. Iterating the whole process finally yields an SNC $S$-variety $\cX_N$ with an \emph{ample} line bundle $\cL_N$ such that $H^q(\cX,\cL^{-1})$ is a direct summand of $H^q(\cX_N,\cL_N^{-1})$, and we conclude by Theorem~\ref{thm:kodaira}.\newline

\noindent{\bf Step 2}. We now consider the general case where $\cL-D$ is merely nef. Since $\cL|_X$ is ample by assumption there exists a vertical blowup $\pi:\cY\to\cX$ with $\cY$ SNC and a vertical $\pi$-exceptional effective $\Q$-divisor $E\in\Div_0(\cY)_\Q$ such that $\pi^*\cL-E$ is ample. This condition implies in particular that $-E$ is $\pi$-ample. If we fix $0<\e\ll 1$ rational so that $\e E$ has coefficients $<1$ then $\pi^*\cL-\e E=(1-\e)\pi^*\cL+\e(\pi^*\cL-E)$ is also ample since $\pi^*\cL$ is nef, and we get 
$$
H^q(\cY,\om_{\cY}\otimes\pi^*\cL)=0\,\,\text{ for all } q\ge 1
$$ 
by Step 1. 

We are next going to show that $R^q\pi_*(\om_{\cY}\otimes\pi^*\cL)=0$ for each $q\ge 1$. Since we have $\pi_*\om_{\cY}=\om_\cX$ (the relative canonical bundle $K_{\cY/\cX}$ is $\pi$-exceptional and effective since $\cX$ is regular), the degeneration of the Leray spectral sequence of $\pi$ will then yield as desired
$$
H^q(\cX,\om_\cX\otimes\cL)\simeq H^q(\cY,\om_\cY\otimes\pi^*\cL)=0
$$
for $q\ge 1$. Let us now prove the claim. Given $q\ge 1$ choose $\cA\in\Pic(\cX)$ sufficiently ample to guarantee that $\cA\otimes R^q\pi_*(\om_{\cY}\otimes\pi^*\cL)$ is globally generated on $\cX$ and 
$$
H^p\left(\cX,\cA\otimes R^m\pi_*(\om_{\cY}\otimes\pi^*\cL)\right)=0\,\,\text{ for all } p\ge 1\text{ and }m\ge 0 
$$
(note that we are only imposing finitely many non-trivial conditions). The degeneration of the Leray spectral sequence yields
$$
H^0\left(\cX,\cA\otimes R^q\pi_*(\om_{\cY}\otimes\pi^*\cL)\right)\simeq H^q\left(\cY,\om_{\cY}\otimes\pi^*(\cL\otimes\cA)\right)=0
$$
for $q\ge 1$ by Step 1 again, since $\pi^*(\cL\otimes\cA)-\e E$ is also ample. It follows that $\cA\otimes R^q\pi_*(\om_{\cY}\otimes\pi^*\cL)=0$ by global generation, which proves the claim since $\cA$ is invertible. 
\end{proof}

%
%

\subsection{Multiplier ideals}
Let us first give the definition of multiplier ideals in our setting:
\begin{defi} Let $\cX$ be a regular model and let $\fa$ be a vertical ideal sheaf on $\cX$. For each rational number $c>0$ the \emph{multiplier ideal} of $\fa^c$ is the vertical ideal sheaf of $\cX$ defined as 
$$
\cJ(\fa^c):=\pi_*\cO_{\cX'}\left(K_{\cX'/\cX}-\lfloor c\,D\rfloor\right)
$$
where $\pi:\cX'\to\cX$ is a vertical blowup with $\cX'$ SNC such that $\pi^{-1}\fa\cdot\cO_{\cX'}$ is locally principal and $D\in\Div_0(\cX')$ is the corresponding effective Cartier divisor. 
\end{defi}
This definition only depends on the model function $c\log|\fa|$ (cf.~\cite{jonmus}), and would in fact make sense for an arbitrary non-positive model function $\f\in\cD(\Xa)$. 

If $\fa_\bullet$ is a graded sequence of vertical ideals, then $\cJ(\fa_\bullet^c)$ is defined as the largest element of the family of coherent ideals $\cJ(\fa_m^{c/m})$, $m\ge 1$, see~\cite[Definition~11.1.5]{Laz}.

As a matter of terminology, if $\cL$ is a line bundle on a model $\cX$, $\fa$ is a vertical coherent ideal sheaf and $c>0$ then we shall say that $\cL\otimes\fa^c$ is nef if $\pi^*\cL-cD$ is nef, where $\pi:\cX'\to\cX$ is the normalization of the blowup of $\cX$ along $\fa$ and $\fa\cdot\cO_{\cX'}=\cO_{\cX'}(-D)$. In other words, the model function $c\log|\fa|$ is required to be $\theta$-psh, where $\theta$ is the curvature form of the model metric on $L$ induced by $\cL$. 

Using Theorem~\ref{thm:KV} we may follow the usual line of arguments to prove the following basic vanishing property of multiplier ideals:
\begin{thm}[Nadel Vanishing]\label{thm:nadel} Let $\cX$ be a regular model of $X$ and $\cL\in\Pic(\cX)$ a line bundle whose restriction to $X$ is ample. If $\fa$ is a vertical coherent ideal sheaf on $\cX$ and $c>0$ is a rational number such that $\cL\otimes\fa^c$ is nef, then we have
$$
H^q\left(\cX,\om_\cX\otimes\cL\otimes\cJ(\fa^c)\right)=0\,\,\text{ for all } q\ge 1. 
$$
In particular, if $\fa_\bullet$ is a graded sequence of vertical coherent ideal sheaves 
on $\cX$ such that $\cL^m\otimes\fa_m$ is globally generated for all sufficiently divisible $m$, then
$$
H^q\left(\cX,\om_\cX\otimes\cL\otimes\cJ(\fa_\bullet)\right)=0\,\,\text{ for all } q\ge 1. 
$$
\end{thm}
\begin{proof} Let $\pi:\cX'\to\cX$ be an SNC model dominating the blowup of $\cX$ along $\fa$, so that we have $\fa\cdot\cO_{\cX'}=\cO_{\cX'}(-D)$ for some effective divisor $D\in\Div_0(\cX')$. By the projection formula we have
$$
\om_\cX\otimes\cL\otimes\cJ(\fa^c)=\pi_*\left(\om_{\cX'}\otimes\pi^*\cL(-\lfloor c\,D\rfloor)\right).
$$
Now $\pi^*\cL-c\,D$ is nef and $c\,D-\lfloor c\,D\rfloor$ has coefficients in $[0,1[$. Lemma~\ref{lem:local} below together with the projection formula yields
$$
R^q\pi_*\left(\om_{\cX'}\otimes\pi^*\cL\left(-\lfloor c\,D\rfloor\right)\right)=0\,\,\text{ for all } q\ge 1.
$$
The Leray spectral sequence is thus degenerate and we conclude using Theorem~\ref{thm:KV}.  The second point follows since $\cJ(\fa_\bullet)=\cJ(\fa_m^{1/m})$ for some $m$ by construction, while the global generation assumption implies that $\cL\otimes\fa_m^{1/m}$ is nef. 
\end{proof}

\begin{lem}[Local vanishing]\label{lem:local} Let $\cX$ be a regular model, let $\fa$ be a vertical ideal sheaf on $\cX$ and let $\pi:\cX'\to\cX$ be an SNC model such that $\fa\cdot\cO_{\cX'}=\cO_{\cX'}(-D)$ with $D\in\Div_0(\cX')$. Then we have 
$$
R^q\pi_*\om_{\cX'}\left(-\lfloor c\,D\rfloor\right)=0\,\, \text{ for all } q\ge 1.
$$
\end{lem}
\begin{proof} We argue as in the last part of the proof of Theorem~\ref{thm:KV}. Let $\cA\in\Pic(\cX)$ be sufficiently ample to guarantee:
\begin{itemize} 
\item[(i)] $\pi^*\cA-c D$ is nef. 
\item[(ii)] $\cA\otimes R^q\pi_*\om_{\cX'}\left(-\lfloor c\,D\rfloor\right)$ is globally generated on $\cX$.
\item[(iii)] $H^p\left(\cX,\cA\otimes R^m\pi_*\om_{\cX'}\left(-\lfloor c\,D\rfloor\right)\right)=0\,\,\text{ for all } p\ge 1\text{ and }m\ge 0$. 
\end{itemize}
Note that the first condition can be achieved since $-D$ is $\pi$-globally generated. The degeneration of the Leray spectral sequence shows that
$$
H^0\left(\cX,\cA\otimes R^q\pi_*\om_{\cX'}\left(-\lfloor c\,D\rfloor\right)\right)=H^q\left(\cX',\om_{\cX'}\left(-\lfloor c\,D\rfloor\right)\otimes\pi^*\cA\right),
$$
which vanishes by Theorem~\ref{thm:KV}. It follows that $\cA\otimes R^q\pi_*\om_{\cX'}\left(-\lfloor c\,D\rfloor\right)=0$ by global generation, whence the result. 
\end{proof}

We may now deduce from the above results the following two consequences that we 
need in the proof of Theorem B. 
\begin{thm}[Subadditivity]\label{thm:subadd} 
  Let $\cX$ be a regular model, $\fa,\fb$ vertical coherent ideal sheaves on $\cX$ 
  and $c,d>0$. Then we have
  $$
  \cJ(\fa^c\cdot\fb^d)\subset\cJ(\fa^c)\cdot\cJ(\fb^d).
  $$
\end{thm}
\begin{proof} 
  This is proved exactly as in~\cite[Theorem 9.5.20]{Laz} using local vanishing.
  (See also~\cite[Theorem~A.2]{jonmus} for a different proof.)
\end{proof}
\begin{thm}[Uniform generation property]\label{thm:uniform} 
  Let $\cX$ be a regular model. Then there exists an ample line bundle 
  $\cA$ on $\cX$ such that the following holds. 
  Given $\cL\in\Pic(\cX)$, a vertical ideal sheaf $\fa$ and a rational number $c>0$ 
  such that $\cL\otimes\fa^c$ is nef, the sheaf 
  $$
  \cA\otimes\cL\otimes\cJ(\fa^c)
  $$
  is globally generated. In particular, if $\fa_\bullet$ is a graded sequence of 
  vertical coherent ideal sheaves on $\cX$ such that $\cL^m\otimes\fa_m$ 
  is globally generated for all sufficiently divisible $m$, then
  $$
  \cA\otimes\cL^m\otimes\cJ(\fa_\bullet^m)
  $$
  is globally generated for all $m$. 
\end{thm}
\begin{rmk}\label{rmk:uniform}
  Fix any ample line bundle $\cB$ on $\cX$. Since, for any line bundle $\cA'$,
  the line bundle $\cB^{n} \otimes (\cA')^{-1}$ is globally generated for $n\gg0$,
  we may take $\cA := \cB^n$ with $n$ sufficiently large in the previous statement.
\end{rmk}
\begin{proof} 
  Let $\cB$ be a given very ample line bundle such that 
  $\cA:=\om_\cX\otimes\cB^{n+1}$ is ample. 
  By the Castelnuovo-Mumford criterion it is enough to check that
  $$
  H^q\left(\cX,\cA\otimes\cL\otimes\cB^{-q}\otimes\cJ(\fa^c)\right)=0
  $$
  for $q=1,\dots,n$, and this is a consequence of Theorem~\ref{thm:nadel}. 
\end{proof}

%
%


%
%
%
%

\end{document}